\documentclass{article}
\usepackage[lang = british]{cdhmms}

\usepackage{amsfonts}
\usepackage{mathtools}
\usepackage{mathrsfs}
\usepackage{subfigure}
\usepackage{etoolbox}

\numberwithin{equation}{section}
1
\allowdisplaybreaks
\numberwithin{equation}{section}
\newtheorem{theorem}{Theorem}[section]
\newtheorem{proposition}[theorem]{Proposition}
\newtheorem{lemma}[theorem]{Lemma}
\newtheorem{corollary}[theorem]{Corollary}

\newtheorem{remark}[theorem]{Remark}

\newtheorem{example}[theorem]{Example}
\numberwithin{equation}{section}

\newcommand\N{{\mathbb N}}
\newcommand\R{{\mathbb R}}

\newcommand\C{{\mathbb C}}

\newcommand\Pp{{\mathbb P}}
\newcommand{\cB}{\mathcal B}
\newcommand{\cC}{\mathcal C}
\newcommand{\cD}{\mathcal D}
\newcommand{\cF}{\mathcal F}
\newcommand{\cH}{\mathcal H}
\newcommand{\cL}{\mathcal L}
\newcommand{\cM}{\mathcal M}
\newcommand{\cO}{\mathcal O}
\newcommand{\cR}{\mathcal R}
\newcommand{\cS}{\mathcal S}
\newcommand{\cT}{\mathscr T}
\newcommand{\sC}{\mathscr{C}}
\newcommand{\sL}{\mathscr{L}}
\newcommand{\sT}{\mathscr{T}}
\newcommand{\fD}{\mathfrak{D}}
\newcommand{\fE}{\mathfrak{E}}
\newcommand{\fM}{\mathfrak{M}}
\newcommand{\fP}{\mathfrak{P}}
\newcommand{\fR}{\mathfrak{R}}
\newcommand{\dd}{{\,\mathrm d}}

\newcommand{\fa}{\forall\,}

\newcommand{\nrm}[2]{\left\|{#1}\right\|_{#2}}
\newcommand{\scalar}[2]{\langle{#1},{#2}\rangle}
\newcommand{\pih}{h^\parallel}
\newcommand{\vertiii}[1]{{\left\vert\kern-0.25ex\left\vert\kern-0.25ex\left\vert #1 \right\vert\kern-0.25ex\right\vert\kern-0.25ex\right\vert}}
\newcommand{\wgt}{\lfloor v\rceil}

\newcounter{Hequation}
	
	\makeatletter\g@addto@macro\equation{\stepcounter{Hequation}}\makeatother
\newcounter{labelrescounter}
	\newcommand{\labelres}[1]{\addtocounter{labelrescounter}{1}\hypertarget{res-#1}{\kern 0.1pt}\label{#1}}

	\makeatletter
	\@addtoreset{equation}{myequation}
	\makeatother
\newcounter{taggedeq}
	\pretocmd{\equation}{\stepcounter{taggedeq}}{}{}

\begin{document}

\title{Special macroscopic modes and hypocoercivity}
\titlemark{Special macroscopic modes and hypocoercivity}

\emsauthor{1}{Kleber Carrapatoso}{K.~Carrapatoso}
\emsauthor{2}{Jean Dolbeault}{J.~Dolbeault}
\emsauthor{3}{Fr\'ed\'eric H\'erau}{F.~H\'erau}
\emsauthor{4}{St\'ephane Mischler}{S.~Mischler}
\emsauthor{5}{Cl\'ement Mouhot}{C.~Mouhot}
\emsauthor{6}{Christian Schmeiser}{C.~Schmeiser}

\emsaffil{1}{Centre de math\'ematiques Laurent
  Schwartz, \'Ecole Polytechnique,
  Institut Polytechnique de Paris, 91128 Palaiseau Cedex, France
\email{kleber.carrapatoso@polytechnique.edu}}

\emsaffil{2}{Centre de Recherche en Math\'ematiques de
  la D\'ecision (CEREMADE, CNRS UMR n$^\circ$ 7534),
  Universit\'es PSL \& Paris-Dauphine, Place de Lattre de
  Tassigny, 75775 Paris 16, France
\email{dolbeaul@ceremade.dauphine.fr}}

\emsaffil{3}{Nantes Universit\'e, CNRS, Laboratoire de Math\'ematiques Jean Leray, LMJL, UMR 6629, 2, rue de la Houssini\`ere BP 92208 F-44322, France
\email{frederic.herau@univ-nantes.fr}}

\emsaffil{4}{Centre de Recherche en Math\'ematiques de
  la D\'ecision (CEREMADE, CNRS UMR n$^\circ$ 7534),
  Universit\'es PSL \& Paris-Dauphine, Place de Lattre de
  Tassigny, 75775 Paris 16, France
\email{mischler@ceremade.dauphine.fr}}

\emsaffil{5}{Department of Pure Mathematics and
  Mathematical Statistics, University of Cambridge,
  Wilberforce Road, Cambridge CB3 0WA, UK
\email{C.Mouhot@dpmms.cam.ac.uk}}

\emsaffil{6}{Fakult\"at f\"ur Mathematik, Universit\"at Wien,
  Oskar-Morgen\-stern-Platz 1, 1090 Wien, Austria
\email{Christian.Schmeiser@univie.ac.at}}

\classification[\href{https://mathscinet.ams.org/mathscinet/search/mscbrowse.html?sk=default&sk=76P05&submit=Chercher}{76P05};
  \href{https://mathscinet.ams.org/mathscinet/search/mscbrowse.html?sk=default&sk=35Q83&submit=Chercher}{35Q83};
  \href{https://mathscinet.ams.org/mathscinet/search/mscbrowse.html?sk=default&sk=82C70&submit=Chercher}{82C70}]{
\href{https://mathscinet.ams.org/mathscinet/search/mscbrowse.html?sk=default&sk=82C40&submit=Chercher}{82C40}}


\begin{abstract}
  We study linear inhomogeneous kinetic equations with an
  external confining potential and a collision operator admitting
  several local conservation laws (local density, momentum and
  energy). We classify all special macroscopic modes (stationary
  solutions and time-periodic solutions). We also prove the
  convergence of all solutions of the evolution equation to such
  non-trivial modes, with a quantitative exponential rate. This
  is the first hypocoercivity result with multiple special
  macroscopic modes with constructive estimates depending on the
  geometry of the potential.\end{abstract}

\maketitle

\section{Introduction}
\label{sec:intro}
\setcounter{equation}{0}
\setcounter{theorem}{0}

Since the publication of Boltzmann's paper~\cite{Boltzmann} in
1876, the existence of time-periodic steady states of the \emph{
inhomogeneous Boltzmann equation} in the whole Euclidean space
is known, in presence of an external harmonic potential. As
explained in~\cite[p.~147]{Cercignani}, ``equilibrium is not
necessarily achieved in an harmonic field. [...] [D]ensity,
velocity and temperature oscillate with the natural frequency of
the field or with twice such a frequency.'' Beyond such remarks,
the classification of the steady states according to the
symmetries of the domain or the symmetries of the external
potential remained untouched for more than a century, although
some special solutions were
known~\cite{Uhlenbeck,Cercignani}. When symmetry partially or
completely breaks, this turns out to be a difficult issue. With
symmetry, special modes have to be taken into account in some
configurations and local collision laws of the collision operator
add significant difficulties to the understanding of the
convergence in asymptotic regimes in all cases, even if there is
no particular symmetry.

Without external potential and for a bounded domain, the problem
has been studied in~\cite{DV05}. In presence of a given external
potential, the question was so far open and our first result is
to classify all steady solutions for linear kinetic equations
with collision operators satisfying the local conservation laws
of physics. Even more difficult is the problem of the stability
of the (possibly time-periodic) steady states and the proof of
the convergence to such states, with an exponential rate, for
inhomogeneous kinetic equations. The question goes back to the
celebrated $H$-theorem of Boltzmann, but became quantitative only
recently with the theory of hypocoercivity. All results involving
an external potential deal with collision operators admitting
only one collision invariant, up to a few attempts
like~\cite{Dua11,DL12} which discard special modes, with
non-constructive methods. Our second result gives the very first
answer to the question of the convergence rate in the whole space
for an external potential without any \emph{a priori} symmetry,
using an entirely new scheme made of a cascade of several
hypocoercive estimates. Alternatively we also propose a
commutator method in the spirit of~\cite{HN04,Vil09}.

Even when the potential has no specific symmetry, which forbids
the existence of any special mode other than the standard
stationary solution, the fact that the collision operator admits
several \emph{collision invariants} is a source of difficulties:
when the potential is almost symmetric, convergence rates get
deteriorated and the geometric properties of the potential have
therefore to be taken into account. The notion of steady states,
defined as the set of attractors in large time asymptotics, is
widely used in physics, and corresponds in our case to minimizers
of the mathematical entropy (that is, the physical entropy, up to
the sign). In this paper we shall speak of \emph{special
  macroscopic modes} in relation with \emph{special} symmetries
of the potential.

\subsection{Equation and assumptions}
\label{ssec:cadre}

Consider the kinetic equation
\begin{align}
  \label{eq:main}
  \partial_t f = \sL f := \sT f + \sC f\,, \quad f_{|t=0} =
  f_0\,,
\end{align}
for the unknown distribution function $f=f(t,x,v)$ depending on
the time variable $t \ge 0$, the position variable $x \in \R^d$,
and the velocity variable $v \in \R^d$, where $d\ge1$ is an
arbitrary dimension. The \emph{transport operator} $\sT$ is
given~by
\begin{align*}
  \sT f := -\,v \cdot \nabla_x f+ \nabla_x \phi \cdot \nabla_v f
\end{align*}
with a stationary, position dependent \emph{potential}
$\phi : \R^d \to \R$. We assume that the \emph{linear collision
  operator} $\sC$ is acting only along the velocity variable
$v \in \R^d$, is self-adjoint in $\mathrm L^2(\mu^{-1})$, with
weight given by the \emph{local Maxwellian} function
\begin{equation}
  \label{mu}
  \forall\,v\in\R^d\,,\quad\mu(v) :=
  \frac{e^{-|v|^2/2}}{(2\,\pi)^{d/2}}\,,
\end{equation}
and has the $(d+2)$-dimensional kernel of \emph{collision
  invariants} given by
\[
  \label{eq:kersC}
  \tag{H0} \mathrm{Ker}\,\sC = \mathop{\mathrm{Span}} \left\{ \mu,
    v_1\,\mu, \dots, v_d\,\mu, |v|^2\,\mu \right\}\,,
\]
corresponding to the local conservation of mass, momentum and
energy. Here $\mathrm L^2(\mu^{-1})$ is the subspace of
$\mathrm L^2_{\mathrm{loc}}(\R^d,\dd v)$ of the functions $f$
such that
\[
  \|f\|_{\mathrm
    L^2(\mu^{-1})}:=\left(\int_{\R^d}\frac{|f|^2}{\mu(v)}\dd
    v\right)^{1/2}
\]
is finite.

We assume that $\sC$ satisfies the following \emph{spectral gap
  property} (which is a quantitative version of the
\emph{spatially homogeneous linearized $H$-theorem})
\begin{align}
  \label{eq:hyp-sg-C}
  \tag{H1}
  - \int_{\R^d} \big( \sC f(v) \big)\,f(v)\,\mu(v)^{-1} \dd v \ge
  \mathrm c_{\sC}\,\| f - \Pi f \|_{\mathrm L^2(\mu^{-1})}^2
\end{align}
for some constant $\mathrm c_{\sC}>0$ and all $f$ in the domain
of $\sC$, where $\Pi$ denotes the
$\mathrm L^2(\mu^{-1})$-orthogonal projection onto
$\mathrm{Ker}\,\sC$. Moreover, we suppose that for any polynomial
function $p(v) : \R^d \to \R$ of degree at most $4$, the function
$p\,\mu$ is in the domain of $\sC$ and
\begin{align}
  \label{eq:lbound}
  \tag{H2}
  C(p):=\| \sC (p\,\mu)\|_{\mathrm L^2(\mu^{-1})}<\infty\,.
\end{align}
We provide examples of collision operators satisfying these
conditions in Appendix~\ref{ssec:ecol}, including the linearized
Boltzmann and Landau operators.

Throughout the paper, we assume that the potential
$\phi : \R^d \to \R$ is such that $\rho(x) := e^{-\phi(x)}$ is a
centred probability density, \emph{i.e.},
\begin{equation}
  \label{hyp:intnorm}
  \tag{H3} \int_{\R^d}\rho(x) \dd x=1\quad \text{and} \quad
  \int_{\R^d} x\,\rho(x) \dd x=0\,.
\end{equation}
We also assume that $\phi$ is of class $C^2(\R^d;\R)$, and for
all $\varepsilon>0$, there exists a constant $C_\varepsilon$ such
that
\begin{equation}
  \label{hyp:regularity}
  \tag{H4}
  \forall\,x\in \R^d\,,\quad |\nabla_x^2\phi(x)| \leq
  \varepsilon\,|\nabla_x\phi(x)|^2+C_\varepsilon\,,
\end{equation}
where $\nabla_{\!x}^2 \phi$ denotes the Hessian matrix of
$\phi$. We further assume that the measure $\rho(x) \dd x$
satisfies the Poincar\'e inequality with a constant
$c_{\text{\tiny P}}>0$,
\begin{align}
  \label{eq:poincarenormal}
  \tag{H5}
  c_{\text{\tiny P}} \int_{\R^d} | \varphi - \langle \varphi
  \rangle |^2\,\rho\,\dd x \leq \int_{\R^d} | \nabla_x \varphi
  |^2\,\rho\,\dd x\,,
\end{align}
for all $\varphi \in \mathrm L^2(\rho)$, where
\[
  \langle \varphi \rangle := \int_{\R^d}\varphi\,\rho\,\dd x
\]
is the average of $\varphi$. Here $\mathrm L^2(\rho)$ is the
subspace of $\mathrm L^2_{\mathrm{loc}}(\R^d,\dd x)$ of the
functions $\varphi$ such that
$\|\varphi\|_{\mathrm
  L^2(\rho)}^2=\int_{\R^d}|\varphi|^2\,\rho\dd x$ is finite.

We assume moment bounds on $\rho$, namely
\begin{equation}
  \label{eq:momentspace}
  \tag{H6}
  \int_{\R^d} \Big( |x|^4 + |\phi|^2+ |\nabla_x \phi|^4
  \Big)\,\rho\,\dd x \le {\mathrm C}_{\phi}
\end{equation}
for some constant ${\mathrm C}_{\phi} >0$. We also introduce the
normalization
\begin{equation}
  \label{eq:phiid}
  \tag{H7}
  \left\langle \nabla_{\!x}^2 \phi \right\rangle = \int_{\R^d}
  \nabla_{\!x}^2 \phi\,\rho\,\dd x = \mathrm{Id}_{d \times d}\,,
\end{equation}
where $\mathrm{Id}_{d \times d}$ the identity matrix of size
$d$. The assumption that $\langle \nabla_{\!x}^2 \phi \rangle$ is
diagonal is not a restriction since it can be obtained through a
rotation in position space. Note that the same rotation in
velocity space leaves the kinetic equation invariant and all
assumptions made so far remain valid. The stronger
assumption~\eqref{eq:phiid} is made for notational simplicity,
and a discussion of the general case is given in
Appendix~\ref{ssec:rescale-change-phi-L}.

The potential
\[\phi(x) := (1+|a_\gamma\,x|^2)^{\gamma/2}-Z_\gamma\,,\]
with $\gamma > 1$ and real normalization constants $a_\gamma$,
$Z_\gamma$,
satisfies~\eqref{hyp:intnorm}--\eqref{hyp:regularity}--\eqref{eq:poincarenormal}--\eqref{eq:momentspace}--\eqref{eq:phiid}. See
Appendix~\ref{ssec:expot} for other examples. \emph{No sign} is
assumed on $f$: one should think of $f$ as a real valued
fluctuation around the equilibrium in the nonlinear Boltzmann or
Landau equation (see Appendix~\ref{ssec:ecol}). Throughout this
article we shall refer to~\eqref{eq:hyp-sg-C}
and~\eqref{eq:poincarenormal} as \emph{spectral gap properties},
and to~\eqref{eq:lbound} and~\eqref{eq:momentspace} as
\emph{bounded moment properties}. These are the structural
assumptions on $\sC$ and $\phi$ for our theory.

Finally, since we are concerned with large time asymptotic
behaviour, we require that the evolution equation~\eqref{eq:main}
is well-posed, a condition which is satisfied by our standard
examples of application, and assume that
\begin{equation}
  \label{hyp:semigroup}
  \tag{H8}
  \mbox{\emph{$t\mapsto e^{t \sL}$ is a strongly continuous
      semi-group on the space $\mathrm L^2(\cM^{-1})$}},
\end{equation}
where $\mathcal M$ is the \emph{global Maxwellian equilibrium}
function given by
\begin{equation}
  \label{def:calM}
  \forall\,(x,v)\in\R^d\times\R^d\,,\quad\cM(x,v) :=
  \rho(x)\,\mu(v) =
  \frac{e^{-\frac12|v|^2-\phi(x)}}{(2\,\pi)^{d/2}}\,,
\end{equation}
and the space $\mathrm L^2(\cM^{-1})$ is the subspace of
$\mathrm L^2_{\mathrm{loc}}(\R^d\times\R^d,\dd x\dd v)$ of the
functions~$f$ such that
\[
  \|f\|_{\mathrm
    L^2(\cM^{-1})}:=\left(\iint_{\R^d\times\R^d}\frac{|f|^2}{\cM}\dd
    x\dd v\right)^{1/2}
\]
is finite.

\subsection{The main result}
\label{subsec:theo}

{}From here on, we assume the normalization
conditions~\eqref{hyp:intnorm}--\eqref{eq:phiid}. The function
$\cM$ defined by~\eqref{def:calM} is a stationary solution
of~\eqref{eq:main} but it is not the unique attractor of the
time-dependent solutions of~\eqref{eq:main}, even up to a mass
normalization. Let us introduce a larger class of steady
states. \textit{Special macroscopic modes} of~\eqref{eq:main} are
the solutions $F = F(t,x,v)$ to the system
\begin{align}
  \label{eq:mainMicroMacro}
  \sC F = 0\,, \quad \partial_t F = \sT F\,.
\end{align}
Of course we read from~\eqref{eq:kersC} that $F = \alpha\,\cM$,
$\alpha \in \R$, is a special macroscopic mode but we also look
for solutions to~\eqref{eq:mainMicroMacro} that can be written as
\begin{align}
  \label{eq:mainMicroMacro2}
  F = \big(r(t,x) + m(t,x) \cdot v + e(t,x)\,\fE(v)\big)\,\cM\,,
\end{align}
for some functions $r$, $m$ and $e$ with values respectively in
$\R$, $\R^d$ and $\R$, with
\begin{equation}
  \label{eq:mainMicroMacro2E}
  \fE(v) := \frac{|v|^2-d}{\sqrt{2\,d}}\,.
\end{equation}
The \emph{energy mode} $F = \beta\,\cH\,\cM$, $\beta \in \R$, is
another stationary solution to~\eqref{eq:mainMicroMacro} where
$\cH$ defined by
\begin{equation}
  \label{def:calH}
  \cH(x,v) := \tfrac12\left(|v|^2-d\right) + \phi(x)-\langle \phi
  \rangle\,
\end{equation}
is the Hamiltonian energy associated with the characteristics of
the transport equation $\partial_t f = \sT f$. As we shall see in
Section~\ref{ssec:macrosols}, it turns out that the linear
combination of global Maxwellian equilibrium functions and energy
modes are the only special macroscopic modes for ``generic
potentials''. Other special macroscopic modes are available under
additional symmetry properties of $\phi$ as observed by
L.~Boltzmann in~\cite{Boltzmann}. These modes deserve some
explanations.

The set of \emph{infinitesimal rotations compatible with $\phi$}
defined as
\begin{equation}
  \label{eq:Rphi}
  \cR_\phi := \left\{x\mapsto A\,x\,:\,A \in \mathfrak M_{d
      \times d}^{\text{\tiny skew}}(\R) \textrm{ s.t. }
    \forall\,x \in \R^d\,, \;\nabla_x \phi(x) \cdot A\,x = 0
  \right\}
\end{equation}
is identified with a subset of the space of skew-symmetric
matrices
\[
  \mathfrak M_{d \times d}^{\text{\tiny skew}}(\R) := \Big\{ A
  \in \mathfrak M_{d \times d}(\R) \,:\,{^T \!A} = -\,A \Big\}\,.
\]
In other words, $A \in \cR_\phi$ if and only if $\phi$
is invariant by the rotation group $\theta\mapsto e^{\theta A}$,
\emph{i.e.},
\[
  \forall\,(\theta,x) \in \R\times\R^d\,,\quad\phi\big(e^{\theta
    A}x\big) = \phi(x)\,.
\]
The set $\cR_\phi$ gives rise to the set of
\emph{rotation modes compatible with $\phi$} defined by
\[
  \fR_{\phi} := \Big\{ (x,v) \mapsto \left( A\,x\cdot v \right)
  \cM(x,v) : A \in \cR_\phi\Big\}\,.
\]
Functions in $\fR_\phi$ are stationary solutions
of~\eqref{eq:mainMicroMacro} associated to all invariances of
$\phi$ under rotation.

There are also some time-periodic special macroscopic modes when
$\phi$ has \emph{harmonic directions}. Let us define
\begin{equation}
  \label{eq:defephi}
  E_\phi := \mathrm{Span}_{\R^d}\Big(\{ \nabla_x \phi(x) -
  x\}_{x \in \R^d}\Big)\,, \quad d_\phi :=
  \mathrm{dim}\,E_\phi\,.
\end{equation}
Notice that $E_\phi$ is a subspace of $\R^d$ and $d_\phi\le
d$. Alternatively, we can characterize $d_\phi$ by
\[
  d_\phi=\mathrm{dim}\;\mathrm{Span}_{\mathrm
    L^2(\rho)} \Big(\left\{\partial_{x_i} \phi
  \right\}_{i=1,\dots,d} \cup\left\{x_i\right\}_{i=1,\dots,d}\Big)-d
\]
and choose cartesian coordinates $(x_1,x_2,\dots,x_d)$ such that
$\partial_{x_i} \phi = x_i$ if and only if $i\in I_\phi:=\{d_\phi
+ 1 , \ldots , d\}$. We face three different cases.\\
$\rhd$ The case $d_\phi = d$ is called \emph{fully non-harmonic}:
$E_\phi=\R^d$ and $\phi$ has no harmonic direction and it that
case, as we shall see below, there are no time-periodic
solutions.\\
$\rhd$ In the case $1 \le d_\phi \le d - 1$, the potential is
called \emph{partially harmonic}. In the harmonic coordinates
$x_{d_\phi+1},\ldots,x_d$, we have $\partial_{x_i} \phi = x_i$
and define \emph{harmonic directional modes}~by
\begin{equation}
  \label{eq:fD}
  \fD_\phi := \mathop{\mathrm{Span}} \Big\{ (x_i \cos t - v_i
  \sin t )\,\cM\,,\,(x_i \sin t + v_i \cos t)\,\cM\,\,:\,i \in
  I_\phi\Big\}\,.
\end{equation}
Harmonic directional modes are also defined if $d_\phi=0$. By
convention, we set $\fD_\phi := \{ 0 \}$ if $d_\phi=d$. All
functions in $\fD_\phi$ are solutions
to~\eqref{eq:mainMicroMacro} which correspond to an
inertia-driven oscillation of period $1$ of particles in a
potential well along a direction in $E_\phi^\bot$.  These modes
are independent of each other.\\
$\rhd$ In the case $d_\phi=0$, the potential is called
\emph{fully harmonic} and
$\phi(x)= \frac12\,|x|^2 + \frac d2\,\log(2\,\pi)$. In addition
to the harmonic directional modes, the set of \emph{harmonic
  pulsating modes}
\begin{multline}
  \label{eq:fP}
  \fP_\phi := \mathop{\mathrm{Span}}
  \bigg\{\left(\tfrac12\left(|x|^2-|v|^2\right) \cos(2\,t) -
    x\cdot v \sin (2\,t) \right)\cM\,,\\
  \left(\tfrac12\left(|x|^2-|v|^2\right)\sin (2\,t) + x\cdot v
    \cos (2\,t) \right)\cM\, \bigg\}
\end{multline}
is also made of solutions to~\eqref{eq:mainMicroMacro}. By
convention, we set $\fP_\phi = \{0\}$ if $d_\phi \ge 1$. These
macroscopic modes correspond to a radially symmetric pulsation of
period $1/2$ of particles in the potential well.

Summing up the above observations, we have obtained \emph{special
  macroscopic modes} of the form
\begin{equation}
  \label{eq:solMacroF}
  F = \alpha\,\cM+ \beta\,\cH\,\cM + A\,x\cdot v\,\cM +
  F_{\mathrm{dir}} + F_{\mathrm{pul}}
\end{equation}
where $\alpha$, $\beta \in \R$, $(x\mapsto A\,x) \in \cR_\phi$,
$F_{\mathrm{dir}} \in \fD_\phi$ and
$ F_{\mathrm{pul}} \in \fP_\phi$. With these definitions at hand,
we can now state the main result of this paper.
\begin{theorem}[Special macroscopic modes and hypocoercivity]
\label{theo:main}
Assume that the potential $\phi$ and the collision operator $\sC$
satisfy the
assumptions~\eqref{eq:kersC}--\eqref{eq:hyp-sg-C}--\eqref{eq:lbound}--\eqref{hyp:intnorm}--\eqref{hyp:regularity}--\eqref{eq:poincarenormal}--\eqref{eq:momentspace}--\eqref{eq:phiid}--\eqref{hyp:semigroup}. Then
\begin{enumerate}
\item[(1)] All special macroscopic modes
  of~\eqref{eq:mainMicroMacro} are given by~\eqref{eq:solMacroF},
  \emph{i.e.}, are linear combinations of the Maxwellian, the
  energy mode, rotation modes compatible with $\phi$, and
  harmonic directional or pulsating modes if allowed by $\phi$.
\item[(2)] There are explicit constants $C>0$ and $\kappa>0$ such
  that, for any $f\in\cC\big(\R^+; \mathrm L^2(\cM^{-1})\big)$ 
  solving~\eqref{eq:main} with initial datum
  $f_0\in \mathrm L^2(\cM^{-1})$, there exists a unique special
  macroscopic mode $F$ such that
\[
  \forall\,t\ge0\,,\quad\left\| f(t) - F(t) \right\|_{\mathrm
    L^2(\cM^{-1})} \le C\,e^{-\kappa\,t} \left\|
    f_0-F(0)\right\|_{\mathrm L^2(\cM^{-1})}\,.
\]
\end{enumerate}
\end{theorem}
\noindent The constants in the decay estimate being explicit
means that the proof is constructive and provides a finite
algorithm for computing $C$ and $\kappa$.

In the following, the norm and scalar product without subscript,
$\|\cdot\|$ and $\langle \cdot, \cdot \rangle$, refer to the
space $\mathrm L^2(\cM)$, so that
\begin{equation}\label{nrm}
  \|h\|^2=\iint_{\R^d\times\R^d}|h|^2\,\cM\dd x\dd v=\scalar
  hh\,.
\end{equation}
With $h:=f/\cM$, $h_0:=f_0/\cM$ and $\pih:=F/\cM$, Part (2) of
Theorem~\ref{theo:main} amounts to
\[
  \forall\,t\ge0\,,\quad\nrm{h(t)-\pih(t)}{} \le
  C\,e^{-\kappa\,t}\nrm{h_0-\pih(0)}{}\,.
\]
When considering functions of $x$ only, the
$\mathrm L^2(\rho)$-norm coincides with the
$\mathrm L^2(\cM)$-norm. We recall that $\langle \cdot \rangle$ stands for the
average with respect to $\rho\dd x$.

In Theorem~\ref{theo:main}, the constants $C$ and $\kappa$ depend
only on bounded moments constants, spectral gap constants or
explicitly computable quantities associated to $\phi$ such as the
rigidity constant (to be defined later) which admits a
quantitative estimate. Moreover, the special macroscopic mode $F$
can be explicitly computed in terms of the initial data $f_0$
(see Section~\ref{ssec:conslaw}).

\subsection{Framework, comments and methods}
\label{ssec:coco}

During the last two decades, new \emph{hypocoercive methods} were
developed for the study of spatially inhomogeneous kinetic
equations. Many linear or nonlinear models were tackled,
including Fokker-Planck, Boltzmann and Landau equations in
various geometries, ranging from bounded domains to the whole
Euclidean space, with or without confining potentials. The
central issue is the trend to equilibrium for these equations, in
the spirit of the celebrated H-Theorem by Boltzmann on the decay
of the entropy, but with constructive estimates which measure the
rate of convergence towards asymptotic regimes described by
steady states. The set of steady states is not fully
characterized by the entropy dissipation, but also depends on the
transport operator and the geometric setting governed either by
boundary conditions or by properties of the potential. The goal
of this paper is to make the notion of steady states explicit by
classifying all special macroscopic modes, and to derive
quantitative estimates on the rate of convergence, with explicit
constants.

Let us give a brief account of the literature. In a series of
papers~\cite{Guo02b,Guo02a,Guo03a} on Landau, Boltzmann and
Vlasov-Boltzmann equations in a periodic box, Y.~Guo used
micro-macro methods inspired from Grad's 13 moments method
introduced in~\cite{Gra65}. The approach of~\cite{DV05} relies on
the derivation of a suitable set of ordinary differential
inequalities. It provides an algebraic rate of convergence to
equilibrium under strong smoothness assumptions on the
solution. The study of linear inhomogeneous kinetic equations
with single conservation laws, such as the linear Boltzmann or
Fokker-Planck equations, and nonlinear equations in a nonlinear
but perturbative regime, took advantage of various ideas of the
theory of hypoellipticity, for instance of~\cite{Hor67}, and gave
rise to robust Hilbertian hypocoercive methods. T.~Gallay coined
the word \emph{hypocoercivity}, by analogy with hypoellipticity,
when coercivity is degenerate in the ambient space but recovered
using commutators, in the context of convergence to steady
states. Hypocoercivity is well adapted to kinetic equations with
general collision operators. We refer to the memoir~\cite{Vil09}
by C.~Villani for an overview of the initial developments of this
theory and to~\cite{HN04,MN05,Her06,DMS09,DMS15} for various
other contributions in exponentially weighted spaces. The theory
of \emph{enlargement of spaces} of~\cite{GMM17} allows to extend
convergence rates to larger, and physically more relevant,
polynomially weighted spaces.

Usually, explicit and constructive estimates cannot be obtained
via compactness arguments. Such estimates are essential for
applications in physics (typical time-scale for relaxation) but
also for a wide array of mathematical questions: range of
validity of perturbation methods applied to nonlinear kinetic
equations, conditions of convergence in the study of diffusive or
macroscopic limits, control of the limiting processes leading to
hydrodynamical equations when the Knudsen number tends to zero,
control of the range of parameters, time and length scales in the
corresponding asymptotic regimes, \emph{etc.} Among a huge
literature, we can refer for instance to~\cite{BGL93,Vil02} and
to~\cite{MR2172804,HR18,BMM19} in polynomially weighted spaces.

In this article, we focus on an important and old problem. We
study kinetic equations involving an external confining potential
as well as several local conservation laws in the collision
process. The linear problem was solved for a fully harmonic
potential in~\cite{Dua11} and under full asymmetry assumptions on
the potential in~\cite{DL12}, both with non-constructive
arguments and for well-prepared initial data so that, in
particular, there are no special macroscopic modes beyond the
Maxwellian stationary solution. Such an assumption destroys the
rich structure of special macroscopic modes and bypasses the
non-trivial consequences of the geometric properties of the
potential on convergence rates. Our contribution is precisely the
study of these consequences, which requires new methods, by
classifying all \emph{special macroscopic modes} and proving
hypocoercivity results with constructive convergence rates in a
natural Hilbertian structure. As in~\cite{Dua11,DL12}, we
restrict our analysis to the linear framework and, for
simplicity, to exponentially weighted spaces, but cover rather
general confining potentials and discuss the consequences of
their geometric properties in terms of symmetry, partial symmetry
or lack of symmetry under rotations. On the one hand, the
extension of our results to polynomially weighted spaces in the
spirit of~\cite{GMM17} is probably doable. On the other hand,
nonlinear stability for Boltzmann and Landau equations with
confining potentials, close to special macroscopic modes,
presents additional difficulties.

The special macroscopic modes other than the global Maxwellian
stationary solutions and the energy modes are consequences of the
symmetries of the potential. Some of these modes are known in the
literature, although no systematic study seems to have been
done. From the point of view of mechanics, any function
$F(x,v)=G\big(\cH(x,v),A\,x\cdot v\big)$ is a stationary solution
of the transport equation where $A\,x\cdot v$ is known as the
\emph{angular momentum} generated by
$A\in\mathfrak M_{d \times d} (\R)$. The property that there is
no other stationary solution, under appropriate conditions, is
known in astrophysics as \emph{Jeans' theorem} (usually
considered with a potential induced by a mean field coupling) and
has to do with Noether's theorem: see~\cite{binney2011galactic}
and references therein. Of course the only profile $G$ compatible
with~\eqref{eq:kersC} is
$G(\mathsf h,a)=p(\mathsf h,a)\,e^{-\mathsf h}$ where $p$ is a
polynomial of order at most two. Proving that any stationary
solution of~\eqref{eq:main} has to
solve~\eqref{eq:mainMicroMacro} is also known as a
\emph{factorization} result.

The existence of time-periodic steady states for the fully
harmonic potential was shown by L.~Boltzmann in~\cite{Boltzmann}
and is mentioned in some references: see for
instance~\cite{Cercignani,Uhlenbeck,BosiCac,Trizac}. In~\cite{Trizac},
time-periodic modes are called \emph{breathing modes}. The
consideration of partially harmonic potentials and their
corresponding harmonic directional modes seems to be new. The
fact that special macroscopic modes also exist for the nonlinear
Boltzmann equation is discussed in Appendix~\ref{ssec:fullnl}.

Now let us review some of the tools which are used in our
paper. To estimate the convergence rate, a major difficulty is to
quantify ``how far'' the potential $\phi$ is from having certain
partial symmetries. Inspired by~\cite{DV02,DV05}, we use some
Korn inequalities for bounded domains which go back
to~\cite{Kor06,Kor09} and adapt them to the Euclidean space, in
presence of a confining potential: see~\cite{CDHMM21}. A typical
quantity involved in our approach is the \emph{rigidity constant}
\begin{multline}
  \label{grad}
  c_{\mathrm K} := \min \left\{ \int_{\R^d} |\nabla \phi(x) \cdot
    A\,x |^2\,\rho(x) \dd x\,:\right.\\
  \left. A \in \cR_\phi^\bot\, \textrm{ such that }\,\int_{\R^d}
    | A\,x |^2\,\rho(x) \dd x = 1 \right\}>0\,,
\end{multline}
where $\cR_\phi^\bot$ is the orthogonal complement in
$\mathrm L^2(\rho)$ of the set $\cR_\phi$ defined
in~\eqref{eq:Rphi}. The time-periodic special macroscopic modes
in $\fD_\phi$ and $\fP_\phi$, defined in~\eqref{eq:fD}
and~\eqref{eq:fP} respectively, are related to the (partial)
harmonicity of~$\phi$ and another difficulty is to quantify ``how
far'' the potential $\phi$ is from being (partially) harmonic. As
for Korn type inequalities, the analysis relies on the finite
dimension of the space $E_\phi$ defined in~\eqref{eq:defephi}.

The spectral gap assumptions~\eqref{eq:hyp-sg-C} in $v$
and~\eqref{eq:poincarenormal} in $x$ reflect the corresponding
confining properties respectively in velocity and space. The
Poincar\'e inequality introduced in~\eqref{eq:poincarenormal} is
linked with the natural Hodge Laplacian associated to the
geometry, sometimes called the \emph{Witten Laplacian}. Denote by
$\nabla_x^*$ the adjoint of $\nabla_x$ in $\mathrm L^2(\rho)$
acting on vector fields $\varphi:\R^d\to\R^d$ according to
$\nabla_x^*\cdot\varphi= (-\nabla_x + \nabla_x \phi ) \cdot
\varphi$. The Witten-Laplace operator $\nabla_x^* \cdot \nabla_x$
is self-adjoint in $\mathrm L^2(\rho)$, with kernel spanned by
constant functions and its first non-zero eigenvalue determines
the optimal Poincar\'e constant $c_{\text{\tiny P}}$. The
operator
\begin{equation}
  \label{omegadef}
  \Omega : = \nabla_x^* \cdot \nabla_x +1 =
  -\,\Delta_x+\nabla_x\phi\cdot \nabla_x +1\,
\end{equation}
is used in the \emph{$0$th-order Poincar\'e inequality}
\[
  c_{\text{\tiny P},1} \left\| \varphi - \langle \varphi \rangle
  \right\|^2 \leq \big\|\,\Omega^{-\frac12}\, \nabla_x
  \varphi\big\|^2
\]
which holds for some constant $c_{\text{\tiny P},1}>0$ under
assumptions~\eqref{eq:poincarenormal} and~\eqref{hyp:regularity},
in the spirit of Poincar\'e-Lions inequalities studied
in~\cite{CDHMM21}.

We provide two proofs of Theorem~\ref{theo:main}. The first proof
follows a micro-macro decomposition as
in~\cite{Guo02b,Guo02a,Guo03a} and~\cite{DL12}. Due to the lack
of \emph{a priori} symmetry assumptions and the delicate interaction of
local conservation laws corresponding to the collision invariants
with the potential, the complexity is significantly
increased. There are also deep similarities with the analysis of
hyperbolic equations with damping studied
in~\cite{HN03,RS04,SW03} after the seminal paper~\cite{KS88} by
S.~Kawashima and Y.~Shizuta. The second proof is given under
slightly more restrictive hypotheses, namely that the collision
operator $\mathscr C$ is bounded and $\phi$ has bounded
derivatives of order two and more. The method is based on
commutator estimates as in~\cite{HN04,MN05,Vil09} in the spirit
of the hypoellipticity theory of~\cite{Hor67}. In practice, an
elegant triple cascade of commutators based on the equality
$[\nabla_v, v\cdot \nabla_x] = \nabla_x$ is needed to control all
macroscopic quantities.

The plan of the article is the following. In
Section~\ref{sec:Htheo}, we review all possible conservation laws
and their relations with the special macroscopic modes. Then we
present the so-called macroscopic equations associated to the
evolution equation~\eqref{eq:main} and perform a change of
unknown in order to work in a simplified Hilbertian framework. In
Section~\ref{sec:minimizers}, we classify all steady states
of~\eqref{eq:main} and prove that they correspond to the special
macroscopic modes. At this stage, we already use
entropy-dissipation arguments in order to prove that
factorization occurs and reduce the problem
to~\eqref{eq:mainMicroMacro}. In Section~\ref{sec:micmac}, we
prove the remaining part of Theorem~\ref{theo:main}, that is, the
hypocoercivity result, using the micro-macro method. In
Section~\ref{sec:cascade} we expose the second proof based on the
commutator's method. A number of technical results are collected
in two appendices. Appendix~\ref{app:tech} collects some
computations and intermediate lemmata needed in the
proofs. For completeness, an extension to weakly coercive collision operators is given in Appendix~\ref{app:extension}.
Appendix~\ref{app:exples} is devoted to examples and
remarks, for instance on the normalization, including a spectral
interpretation of Theorem~\ref{theo:main}, the extension of our
special macroscopic modes to the fully nonlinear Boltzmann
equation, and various examples of collision operators and
potentials.

\section{Conservation laws and macroscopic equations}
\label{sec:Htheo}
\setcounter{equation}{0} \setcounter{theorem}{0}

In this section we characterize the special macroscopic modes, as
defined by~\eqref{eq:mainMicroMacro}, for generic potentials. We
also identify the global conservation laws and the macroscopic
equations associated to~\eqref{eq:main}. From here on, we assume
that~\eqref{eq:momentspace} holds. This assumption is needed to
justify the computations, which are given below only at formal
level, for sake of simplicity.

\subsection{The equations for the special macroscopic modes}
\label{ssec:macrosols}

We recall that by~\eqref{eq:mainMicroMacro2}, any special
macroscopic mode $F$ can be written as
\[
  \label{F}
  F = \pih\,\cM\,, \quad \pih =r + m \cdot v + e\,\fE(v)\,.
\]
By~\eqref{eq:mainMicroMacro}, we know that
$\partial_t F = \sT F$. By integrating in $v$ the evolution
equation against $1$, $v$, $v \otimes v$ and $v\,\fE$, we obtain
\emph{macroscopic equations} on the macroscopic quantities
$r=r(t,x)$, $m=m(t,x)$ and $e=e(t,x)$:
\begin{subequations}
  \label{SMrmem}
  \begin{align}
    & \partial_t r = \nabla_x^* \cdot m\,,\label{SMrmem-a}\\
    & \partial_t m = -\,\nabla_x
      r+\sqrt{\tfrac2d}\,e\,\nabla_x\phi\label{SMrmem-b}\\
    & \partial_t e = -\,\sqrt{\tfrac2d}\,\nabla_x \cdot
      m\,,\label{SMrmem-c}\\
    & \tfrac1{\sqrt{2\,d}}\left( \partial_t e \right)
      \mathrm{Id}_{d \times d} = -\,\nabla_x^{\text{\tiny
      sym}}\,m\,,\label{SMrmem-d}\\
    & 0 = \nabla_x e\,,\label{SMrmem-e}
\end{align}
\end{subequations}
where the \emph{symmetric gradient} is defined by
\begin{equation}
  \label{SymmetricGradient}
  \forall\,i,\,j=1,\ldots,d\,,\quad(\nabla_x^{\text{\tiny
      sym}}\,m)_{ij}:=\frac12\left( \partial_j m_i + \partial_i
    m_j\right).
\end{equation}

{}From~\eqref{SMrmem-e}, we deduce that $e$ does not depend on
the space variable, and therefore
\begin{align}
  \label{SMeq:expe}
  e = \langle e \rangle=:c
\end{align}
is a function $t\mapsto c(t)$ depending on $t$ only. We recall
that the average is defined as
$\langle \varphi \rangle := \int_{\R^d}\varphi\,\rho\,\dd
x$. Then we read from~\eqref{SMrmem-d} that
\begin{equation}
  \label{SMeq:nablasym-m}
  \tfrac{c'}{\sqrt{2\,d}}\,\mathrm{Id}_{d \times d}=
  -\,\nabla_x^{\text{\tiny sym}}\,m\,.
\end{equation}
By the Schwarz Lemma applied to $m=(m_k)_{k=1}^d$, we have
\begin{align}
  \label{skewsym}
  (\nabla_{\!x}^2 m_k)_{i,j}=\partial^2 _{x_i x_j} m_k =
  \partial_{x_i} \left( \nabla^{\mbox{{\tiny sym}}}
  m\right)_{j,k} + \partial_{x_j} \left( \nabla^{\mbox{{\tiny
  sym}}} m\right)_{i,k} - \partial_{x_k} \left(
  \nabla^{\mbox{{\tiny sym}}} m\right)_{i,j}
\end{align}
for any $i$, $j$, $k = 1 , \ldots , d$. By
differentiating~\eqref{SMeq:nablasym-m} with respect to $x_i$,
$x_j$ and $x_k$, we get that $\nabla_{\!x}^2 m = 0$, so that in
particular the \emph{skew-symmetric gradient}, defined by
\begin{equation}
  \label{SkewSymmetricGradient}
  \forall\,i,\,j=1,\ldots,d\,,\quad(\nabla_x^{\text{\tiny
      skew}}m)_{ij}:=\frac12\left( \partial_j m_i - \partial_i
    m_j\right),
\end{equation}
is constant in the $x$-variable and equal to its
average. Together with~\eqref{SMeq:nablasym-m}, using
$\nabla_xm=\nabla_x^{\text{\tiny skew}}m+\nabla_x^{\text{\tiny
    sym}}\,m$, we deduce that
\begin{equation}
  \label{SMeq:expm-bis}
  m(t,x) = \langle \nabla m \rangle\,x + \langle m \rangle =
  A(t)\,x + b(t) - \tfrac1{\sqrt{2\,d}}\,c'(t)\,x\,,
\end{equation}
with $A(t) := \langle \nabla_x^{\text{\tiny skew}}m \rangle$ and
$ b(t) := \langle m \rangle$. Taking~\eqref{SMeq:expe}
and~\eqref{SMeq:expm-bis} into account in~\eqref{SMrmem-b}
implies
\[
  \nabla_x r = -\,\partial_t m+\sqrt{\tfrac2d}\,e\,\nabla_x\phi
  =-\,A'\,x- b' + \tfrac1{\sqrt{2\,d}}\,c''\,x +
  \sqrt{\tfrac2d}\,c\,\nabla \phi\,.
\]
Taking the skew-symmetric gradient of this equation gives
$0 = -A'$. Hence~$A$ is a skew-symmetric matrix which does not
depend on $t$ and
\begin{equation}
  \label{SMeq:expm}
  m(t,x) = A\,x + b(t) - \tfrac1{\sqrt{2\,d}}\,c'(t)\,x\,.
\end{equation}
Taking~\eqref{SMeq:expe} and~\eqref{SMeq:expm} into account, we
can then take the primitive in space of~\eqref{SMrmem-b} and we
immediately deduce that the macroscopic density satisfies
\begin{equation}
  \label{SMeq:expr}
  r(t,x) = r_0 - b'(t)\cdot x +
  \tfrac1{2\,\sqrt{2\,d}}\,c''(t)\,\xi_2(x) +
  \sqrt{\tfrac2d}\,c(t)\,\xi_\phi(x)
\end{equation}
where $r_0$ is an integration constant,
\begin{equation}
  \label{xi2}
  \xi_2(x) := |x|^2 - \langle |x|^2 \rangle
\end{equation}
and
\begin{equation}
  \label{xiphi}
  \xi_\phi (x):= \phi-\langle \phi \rangle\,.
\end{equation}
An integration against $\rho$ shows that
$r_0=\langle r(t,\cdot)\rangle$ is in fact independent
of~$t$. Inserting the expressions of $r$ and $m$ given
by~~\eqref{SMeq:expm} and~\eqref{SMeq:expr} into~\eqref{SMrmem-a}
yields a differential equation satisfied by $A$, $b(t)$ and
$c(t)$:
\begin{proposition}
  \label{Prop:EDO}
  Assume that $r$, $m$ and $e$ solve~\eqref{SMrmem}. With the
  above notations,
  $A(t):=\langle\nabla_x^{\text{\tiny skew}}m\rangle$,
  $b(t):=\langle m\rangle$ and $c(t):=\langle e\rangle$ solve
  \begin{equation}
    \label{SMeq:Abcm}
    \tfrac{2\,\xi_\phi(x) + \nabla_x \phi \cdot x -
      d}{\sqrt{2\,d}}\,c' + \tfrac{\xi_2(x)}{2\,\sqrt{2\,d}}\,c''' -
    \nabla_x \phi \cdot b - b''\cdot x - \nabla_x \phi \cdot A\,x
    =0\,.
  \end{equation}
\end{proposition}
Equation~\eqref{SMeq:Abcm} suggests, on the one hand, that
(partial) harmonicity of the potential $\phi$ allows for
non-trivial choices of $b$ and~$c$, as we shall indeed see
later. On the other hand, for a generic potential $\phi$ in the
sense that the functions $1$, $x$, $\phi$,
$\nabla_x \phi \cdot x$, $\nabla\phi$, $|x|^2$ and
$\nabla_x \phi \cdot A\,x$ (if $A\neq0$) are linearly
independent, equation~\eqref{SMeq:Abcm} implies that
\[
  c' = 0\,,\quad b = 0\quad\mbox{and}\quad A = 0\,,
\]
so that $r=r_0 + c\,\sqrt{\tfrac2d}\,\xi_\phi$, $m=0$ and $e= c$,
for two constants $r_0$ and $c\in \R$. In other words, we have
\[
  \pih = r_0 + \sqrt{\tfrac2d}\,c\,\cH -
  \sqrt{\tfrac2d}\,c\,\big(\langle \phi \rangle-\tfrac d2\big)
\]
if $\phi$ is not a (partially) harmonic potential: any special
macroscopic mode is then a linear combination of a Maxwellian
function and an energy mode.

\subsection{Global conservation laws}\label{ssec:conslaw}

Consider a solution
$f \in \cC\big(\R^+ ; \mathrm L^2(\cM^{-1})\big)$
to~\eqref{eq:main} with initial datum
$f_0\in \mathrm L^2(\cM^{-1})$. Associated with the symmetries of
the equation, there are local conservations which, after
integration on the phase space $\R^d\times\R^d$, give rise to
global conservation laws. These laws allow us to identify the
special macroscopic modes compatible with~$\phi$ which, as we
shall see later, attract the solutions to the Cauchy problem.

The \emph{conservation of mass} writes
\[
  \frac{\mathrm d}{\mathrm dt} \iint_{\R^d\times\R^d} f(t,x,v)
  \dd x \dd v = 0\,.
\]
Hence $\alpha\,\cM$ is a solution to~\eqref{eq:mainMicroMacro}
with same mass
\begin{equation}
  \label{msm}
  \alpha := \iint_{\R^d\times\R^d}f_0 \dd x \dd v
\end{equation}
as~$f$. With~$\cH$ defined by~\eqref{def:calH}, the
\emph{conservation of energy} amounts to
\begin{align*}
  \frac{\mathrm d}{\mathrm dt} \iint_{\R^d\times\R^d} \cH(x,v)
  f(t,x,v) \dd x \dd v = 0\,.
\end{align*}
The distribution function $\beta\,\cH\,\cM$, with
\begin{equation}
  \label{hsm}
  \beta := \frac{\iint_{\R^d\times\R^d} \cH\,f_0 \dd x\dd
    v}{\iint_{\R^d\times\R^d} \cH^2\,\cM\,\dd x \dd v}\,,
\end{equation}
is a solution to~\eqref{eq:mainMicroMacro} with same energy as
the conserved energy of $f$. With $\mathsf f_1:=\cM$ and
$\mathsf f_\cH:=\cH\,\cM/\nrm{\cH\,\cM}{\mathrm L^2(\cM^{-1})}$,
we have that
\[
  \alpha\,\cM=\scalar{f_0}{\mathsf f_1}_{\mathrm
    L^2(\cM^{-1})}\,\mathsf
  f_1\quad\mbox{and}\quad\beta\,\cH\,\cM=\scalar{f_0}{\mathsf
    f_\cH}_{\mathrm L^2(\cM^{-1})}\,\mathsf f_\cH\,.
\]
Moreover, the global conservations of mass and energy write
\[
  \label{hsm-rsm}
  \begin{aligned}
    & \forall \, t \ge 0, \quad
    \iint_{\R^d\times\R^d}\big(f(t,x,v)-\alpha\,\cM(x,v)\big)\dd
    x\dd v=0\,,\\
    & \forall \, t \ge 0, \quad
    \iint_{\R^d\times\R^d} \cH(x,v) \, \big(f(t,x,v)-\beta\,
    \cH(x,v)\, \cM(x,v)\big) \dd x \dd v=0\,.
\end{aligned}
\]

The transport operator can be written as
$\mathcal Tf=X\cdot\nabla_{x,v}f=\nabla_{x,v}\cdot(X\,f)$, where
$X=(-v,\nabla_x\phi)=(-\nabla_v\cH,\nabla_x\cH)$ is a divergence
free vector field, in the sense that
\begin{equation}
  \label{DivergenceFree}
  \nabla_{x,v}\cdot X=0\,.
\end{equation}
As a consequence, the volume conservation in the phase space
under the action of the flow induces the local mass conservation
and the~\eqref{DivergenceFree} symmetry gives rise to the global
mass conservation. Another symmetry is associated with the fact
that $\cH$ is conserved along the characteristics of Newton's
equations $\dot x = v$ and $\dot v =-\,\nabla_x\phi$. This is
reflected by the Poisson brackets: a stationary solution $F$
satisfies
\begin{equation}
  \label{PoissonBrackets}
  0=\{F,\cH\}:=\nabla_v\cH\cdot\nabla_xF-\nabla_x\cH\cdot\nabla_vF\,,
\end{equation}
but by replacing $F$ by $\cH\,F$, it is also clear that
\[
  \{\cH\,F,\cH\}=\cH\,\{F,\cH\}+\{\cH,\cH\}\,\cH=0\,.
\]
The underlying reason is that the transport dynamics involving a
time-in\-de\-pen\-dent potential is invariant under a translation
in time, which gives rise to the global conservation of
energy. These considerations can be generalized. To any
continuous group of transformations which leaves $\sT$ invariant,
we can associate an infinitesimal transformation
$\mathcal G(x,v)$ such that $\{\mathcal G,\cH\}=0$ and as a
consequence, if $f$ solves the transport equation
$\partial_tf=\sT f$, then
\begin{equation}
  \label{GlobalConservation}
  \frac{\mathrm d}{\mathrm dt}\iint_{\R^d\times\R^d} \mathcal
  G(x,v)\,f(t,x,v) \dd x\dd v=0\,.
\end{equation}
Additionally, if $v\mapsto\mathcal G(x,v)\,\mu(v)$ is in the
kernel of the collision operator $\sC$ for any $x\in\R^d$, then
$\gamma\,\mathcal G\,\cM$ is a solution
of~\eqref{eq:mainMicroMacro}, \emph{i.e.}, a \emph{special
  macroscopic mode}, for any
$\gamma\in\R$. Then~\eqref{GlobalConservation} holds for any
solution $f$ of~\eqref{eq:main} with initial datum $f_0$ and
there is a unique~$\gamma\in\R$ such that
\[
  \iint_{\R^d\times\R^d} \mathcal
  G(x,v)\,\big(f_0(x,v)-\gamma\,\mathcal G(x,v)\,\cM(x,v)\big)
  \dd x\dd v=0\,.
\]
More considerations on symmetries, local and global conservation
laws, and Noether's theorem can be found in textbooks on
classical mechanics like, for
instance,~\cite{MR3752660,MR2761185}. The case of rotational
symmetries enters this framework:

When $\phi$ is invariant under a rotation, stationary
\emph{rotation modes} appear. Let $A\in\cR_\phi$ as defined
in~\eqref{eq:Rphi} and consider the rotation group
$(R_\theta)_{\theta \in \R}$ defined by
$R_\theta := e^{\theta A}$ and a point $x_0 \in \R^d$ so that
$\phi\big(R_\theta(x-x_0) + x_0\big) = \phi(x)$ for any
$\theta\in\R$. By differentiation with respect to $\theta$, we
get
\begin{align*}
  \fa x \in \R^d\,, \quad \left( A\,x + u \right) \cdot \nabla_x
  \phi(x) = 0
\end{align*}
with $u=-A\,x_0$. Integrating the above identity against
$(u\cdot x)\,\rho$ yields $u=0$ after an integration by parts
because $\rho(x) \dd x$ is centred according to
Assumption~\eqref{hyp:intnorm}. \emph{Rotation modes compatible
  with~$\phi$} are therefore restricted to $x_0=0$:
see~\cite{CDHMM21} for similar computations. As a consequence, if
we compute the Poisson bracket as defined
in~\eqref{PoissonBrackets}, we find that
\[
  \{\mathcal G,\cH\}=0\quad\text{if}\quad\mathcal G(x,v)=\left(
    A\,x\cdot v \right)
\]
and the conservation of the \emph{total angular momentum}
associated with this rotation writes
\begin{align*}
  \frac{\mathrm d}{\mathrm dt} \iint_{\R^d\times\R^d} \left(
  A\,x\cdot v \right) f(t,x,v) \dd x \dd v = 0\,.
\end{align*}
Given $f_0$, let us identify the corresponding special
macroscopic mode.

Associated with $f_0$, we introduce the initial momentum
\[m_0(x) := \left(\int_{\R^d}v\,f_0(x,v) \dd v \right)e^{\phi(x)}
\] and the infinitesimal rotation
$x \mapsto A\,x := \Pp_\phi m_0(x)$, where $\Pp_\phi$ is the
orthogonal projection onto the vector space $\cR_\phi$ in
$\mathrm L^2(\rho)$. We can then check that $x \mapsto A\,x$
belongs to $\cR_\phi$ and thus the function (rotational mode
compatible with $\phi$)
\begin{equation}
  \label{rsm}
  F_{\mathrm{rot}}(x,v):=\left( A\,x\cdot v \right) \cM
\end{equation}
belongs to $\fR_{\phi}$, so that $F_{\mathrm{rot}}$ is a solution
to~\eqref{eq:mainMicroMacro} with same conserved total angular
momen\-tum as $f$. Denoting
\[
  m_f(t,x) := \left( \int_{\R^d} v\,f(t,x,v) \dd v \right) e^{\phi(x)}
\]
the momentum of $f$, the associated conservation law then reads
\begin{equation}
  \label{eq:consm}
  \Pp_\phi m_f(t) = \Pp_\phi m_0 \quad \textrm{ or equivalently }
  \quad \Pp(m_f-m_0) \in \cR_\phi^\bot\,,
\end{equation}
where $\Pp$ is the orthogonal projection onto the vector space of
all infinitesimal rotations
$\cR := \{ x \mapsto A\,x\,:\,A \in \mathfrak M_{d\times
  d}^{\text{\tiny skew}} (\R) \}$, identified with
$\mathfrak M_{d\times d}^{\text{\tiny skew}} (\R)$, in
$\mathrm L^2(\rho)$, and $\cR_\phi^\bot$ is the orthogonal of
$\cR_\phi$, seen as a subspace of $\cR$, for the scalar product
induced by $\mathrm L^2(\rho)$. We refer to
Lemma~\ref{lem:proj-ort-append} in Appendix~\ref{ssec:momvs} for
a precise statement and a short proof.

If we denote by
$(A_j)_{j\in J_\phi} = (x\mapsto A_j\,x)_{j\in J_\phi}$ a basis
of $\cR_\phi$ normalized by the condition
$\nrm{\mathsf f_{\mathrm{rot},j}}{\mathrm L^2(\cM^{-1})}=1$ for
any $j\in J_\phi$, where
$\mathsf f_{\mathrm{rot},j}(x,v):=\left( A_j\,x\cdot v
\right)\cM(x,v)$, then the conservation of the total angular
momentum implies for all $j \in J_\phi$:
\[
  \forall \, t \ge 0, \quad \iint_{\R^d\times\R^d} \left[
    f(t,x,v)-\scalar{f_0}{\mathsf f_{\mathrm{rot},j}}_{\mathrm
      L^2(\cM^{-1})}\,\mathsf f_{\mathrm{rot},j}(x,v) \right] \dd x
  \dd v=0.
\]

Now let us turn our attention to the \emph{time-periodic special
  macroscopic modes} and start with the harmonic directional
modes, which appear when \hbox{$d_\phi \leq d-1$}. We choose a
coordinate system such that $\partial_{x_i} \phi = x_i$ for
$i\in I_\phi=\{d_\phi + 1 , \ldots , d\}$. In that case, the
potential $\phi$ is such that
$x\mapsto\phi(x)-\frac12\sum_{i\in I_\phi}x_i^2$ depends only on
$(x_1,\ldots,x_{d_\phi})$ and for any $i\in I_\phi$ the
\emph{harmonic directional modes}, defined for all
$(t,x,v)\in\R^+\times\R^d\times\R^d$ by
\begin{align*}
  & \mathsf f^+_{\mathrm{dir}, i}(t,x,v):=\big( x_i\,\cos t +
    v_i\, \sin t\big)\,\cM(x,v)\,,\\
  & \mathsf f^-_{\mathrm{dir},i}(t,x,v):=\big( v_i\,\cos t -x_i\,
    \sin t\big)\,\cM(x,v)\,,
\end{align*}
solve~\eqref{eq:mainMicroMacro}. A direct computation of the
solution of~\eqref{eq:main} with initial datum
$f_0\in \mathrm L^2(\cM^{-1})$ shows that
\begin{align*}
  & \frac{\mathrm d}{\mathrm dt} \iint_{\R^d\times\R^d} x_i\,f
    \dd x \dd v = \iint_{\R^d\times\R^d} v_i\,f \dd x \dd v\,,\\
  & \frac{\mathrm d}{\mathrm dt} \iint_{\R^d\times\R^d} v_i\,f
    \dd x \dd v = -\iint_{\R^d\times\R^d} x_i\,f \dd x \dd v\,,
\end{align*}
which implies that these two global quantities evolve as an
harmonic oscillator with period equal to $1$. For any
$i \in I_\phi$, let us define
\[
  \gamma_i := \iint_{\R^d\times\R^d} x_i\,f_0 \dd x\dd v\,,\quad
  \bar{\gamma}_i := \iint_{\R^d\times\R^d} v_i\,f_0 \dd x\dd
  v\,.
\]
The function
\begin{equation}
  \label{dsm}
  F_{\mathrm{dir}} := \sum_{i=d_\phi+1}^d\left(\gamma_i\,\mathsf
    f^+_{\mathrm{dir}, i}+\bar{\gamma}_i\,\mathsf
    f^-_{\mathrm{dir},i}\right)
\end{equation}
solves~\eqref{eq:mainMicroMacro} and belongs to $\fD_\phi$ as
defined by~\eqref{eq:fD}. Moreover $(f-F_{\mathrm{dir}})$
satisfies the following two global conservation laws: for any
$t\ge0$,
\[
  \label{eq:law-dir1}
  \begin{aligned}
    & \iint_{\R^d\times\R^d} x_i\,\big(
    f(t,x,v)-F_{\mathrm{dir}}(t,x,v) \big) \dd x \dd v=0\,,\\
    & \iint_{\R^d\times\R^d} v_i\,\big(
    f(t,x,v)-F_{\mathrm{dir}}(t,x,v) \big) \dd x \dd v = 0\,.
\end{aligned}
\]

When \emph{all} coordinates are harmonic ($d_\phi=0$), then
$\phi(x)= \frac12 |x|^2 + \frac d2 \log(2\,\pi)$ due to the
normalization~\eqref{eq:phiid} and the \emph{harmonic pulsating
  modes}, defined for all $(t,x,v)\in\R^+\times\R^d\times\R^d$ by
\begin{align*}
  & \mathsf f^+_{\mathrm{pul}}(t,x,v):=\tfrac1{\sqrt d}\,\big(
    x\cdot v \cos (2\,t) + \tfrac12\left(|x|^2-|v|^2\right) \sin
    (2\,t)\big)\,\cM(x,v)\,,\\
  &\mathsf f^-_{\mathrm{pul}}(t,x,v):=\tfrac1{\sqrt d}\,\big(
    \tfrac12\left(|x|^2-|v|^2\right) \cos(2\,t) - x\cdot v \sin
    (2\,t)\big)\,\cM(x,v)\,,
\end{align*}
solve~\eqref{eq:mainMicroMacro}. A direct computation of the
solution of~\eqref{eq:main} with initial datum
$f_0\in \mathrm L^2(\cM^{-1})$ shows that
\begin{align*}
  & \frac{\mathrm d}{\mathrm dt} \iint_{\R^d\times\R^d} (x\cdot
    v)\,f \dd x \dd v = -\,2 \iint_{\R^d\times\R^d}
    \frac12\left(|x|^2-|v|^2\right)f \dd x \dd v\,,\\
  & \frac{\mathrm d}{\mathrm dt} \iint_{\R^d\times\R^d}
    \frac12\left(|x|^2-|v|^2\right)f \dd x \dd v = 2
    \iint_{\R^d\times\R^d} (x\cdot v)\,f \dd x \dd v\,,
\end{align*}
which implies that these two global quantities evolve as an
harmonic oscillator with period equal to $1/2$. With
\[
  \delta := \frac1{\sqrt d} \iint_{\R^d\times\R^d} (x\cdot
  v)\,f_0 \dd x\dd v\,, \quad \bar{\delta} := \frac1{\sqrt d}
  \iint_{\R^d\times\R^d} \frac12\left(|x|^2-|v|^2\right) f_0 \dd
  x \dd v\,,
\]
the function
\begin{equation}
  \label{psm}
  F_{\mathrm{pul}} := \delta\,\mathsf
  f^+_{\mathrm{pul}}+\bar\delta\,\mathsf f^-_{\mathrm{pul}}
\end{equation}
solves~\eqref{eq:mainMicroMacro} and belongs to $\fP_\phi$ as
defined by~\eqref{eq:fP}. Moreover $(f-F_{\mathrm{pul}})$
satisfies the following two global conservation laws: for any
$t\ge0$,
\[
  \label{eq:law-dir2}
  \begin{aligned}
    & \iint_{\R^d\times\R^d} (x\cdot v)\,\big(
    f(t,x,v)-F_{\mathrm{pul}}(t,x,v) \big) \dd x \dd v=0\,,\\
    & \iint_{\R^d\times\R^d} \frac12\left(|x|^2-|v|^2\right) \big(
    f(t,x,v)-F_{\mathrm{pul}}(t,x,v) \big) \dd x \dd v = 0\,.
  \end{aligned}
\]

Let us consider the set of the generators of all above special
macroscopic modes
\[
  \widehat{\cS}:=\{\mathsf f_1,\mathsf f_\cH\}\cup\{\mathsf
  f_{\mathrm{rot},j}\}_{j\in J_\phi}\cup\{\mathsf
  f^\pm_{\mathrm{dir},i}\}_{i\in I_\phi,\pm}\cup\{\mathsf
  f^\pm_{\mathrm{pul}}\}_\pm\,.
\]
We have the following orthogonality property.
\begin{lemma}
  \label{lem:orthn}
  The functions of $\widehat{\cS}$ are orthonormal in
  $\mathrm L^2(\cM)$.
\end{lemma}
\begin{proof}[Proof of Lemma~\ref{lem:orthn}]
  This follows from direct computation using standard properties
  of Hermite functions.
\end{proof}

As a straightforward consequence of Lemma~\ref{lem:orthn}, we
obtain
\begin{corollary}
  \label{Cor:GlobalConservationLaws}
  Assume that $f \in \cC\big(\R^+ ; \mathrm L^2(\cM^{-1})\big)$
  is a solution to~\eqref{eq:main} with initial datum
  $f_0\in \mathrm L^2(\cM^{-1})$. With the above notations, for
  any $\mathsf f\in\widehat{\cS}$ and any $t\ge0$, we have
  \[
    \iint_{\R^d\times\R^d} \Bigl( f(t)- \alpha\,\cM -
    \beta\,\cH\,\cM -
    F_{\mathrm{rot}}-F_{\mathrm{dir}}(t)-F_{\mathrm{pul}}(t)
    \Bigr)\frac{\mathsf f}\cM\,\dd x \dd v = 0\,.
  \]
\end{corollary}

\subsection{A micro-macro decomposition}
\label{ssec:rescaling}

Let us consider a solution
$f \in \mathcal C\big(\R^+ ; \mathrm L^2(\cM^{-1})\big)$
to~\eqref{eq:main}. We get rid of the special macroscopic modes
built
in~\eqref{msm}--\eqref{hsm}--\eqref{rsm}--\eqref{dsm}--\eqref{psm}
and rewrite the evolution problem in $\mathrm L^2(\cM)$ in terms
of
\begin{equation}
  \label{eq:defh}
  h := \frac{ f- \alpha\,\cM - \beta\,\cH\,\cM -
    F_{\mathrm{rot}}-F_{\mathrm{dir}}-F_{\mathrm{pul}}}{\cM}
\end{equation}
for all $(t,x,v)$, where only $h$, $f$, $F_{\mathrm{dir}}$ and
$F_{\mathrm{pul}}$ depend on $t$. Then $h$ satisfies
\begin{equation}
  \label{eq:mainh}
  \partial_t h = \cL\,h := \cT h + \cC h\,, \quad h|_{t=0} = h_0
\end{equation}
with $\sT$ defined as before by
$\sT h = \nabla_x \phi \cdot \nabla_v h - v \cdot \nabla_x h$ and
the new collision operator
\begin{equation}
  \label{eq:cC}
  \cC h := \mu^{-1}\,\sC\left( \mu\,h \right)
\end{equation}
where $\mu$ is defined in~\eqref{mu}.

The operator $\cC$ acts only on the velocity variable, is
self-adjoint in $\mathrm L^2(\cM)$ (when integrating in $x$, $v$)
and $\mathrm L^2(\mu)$ (when integrating in $v$) and
\[
  \label{eq:kerC}
  \mathrm{Ker}\,\cC = \mathop{\mathrm{Span}} \left\{ 1, v_1, \dots,
    v_d, |v|^2 \right\}\,.
\]
Let us consider the \emph{micro-macro decomposition}
\[
  \label{decomph}
  h = \pih + h^\bot\,,\quad \pih := r + m \cdot v + e\,\fE(v)\,,
\]
where $\pih$ is the $\mathrm L^2(\mu)$-orthogonal projection of
$h$ on $\mathrm{Ker}\,\cC$ and $\fE$ defined
by~\eqref{eq:mainMicroMacro2E} is a normalized Hermite polynomial
of degree $2$. In other words,
\begin{equation}
  \label{Eqns:macro}
  \begin{aligned}
    & r(t,x) := \int_{\R^d} h(t,x,v)\,\mu(v) \dd v\\
    & m(t,x) := \int_{\R^d} v\,h(t,x,v)\,\mu(v) \dd v\\
    & e(t,x) := \int_{\R^d} \fE(v)\,h(t,x,v)\,\mu(v) \dd v
  \end{aligned}
\end{equation}
are the \emph{macroscopic quantities} corresponding to the
spatial density, the local flux and thermal energy, while
$h^\bot$ is the \emph{microscopic} part. The
definition~\eqref{Eqns:macro} coincides with the
definition~\eqref{eq:mainMicroMacro2} used in
Section~\ref{ssec:macrosols} to define the special macroscopic
modes.

With these notations~\eqref{eq:hyp-sg-C} reads
\[
  -\,\langle \cC h,h \rangle \ge \mathrm c_{\sC}\,\| h^\bot
  \|^2\,.
\]
According to Corollary~\ref{Cor:GlobalConservationLaws}, $h$ has
\emph{multiple global conservation laws}.

\begin{corollary}
  \label{Cor:GlobalConservationLaws2}
  Assume that $f \in \cC\big(\R^+ ; \mathrm L^2(\cM^{-1})\big)$
  is a solution to~\eqref{eq:main} with initial datum
  $f_0\in \mathrm L^2(\cM^{-1})$. With $h$ and $(r,m,e)$
  respectively defined by~\eqref{eq:defh}--\eqref{eq:mainh}
  and~\eqref{Eqns:macro}, we have the following properties.

  \noindent
  $\rhd$ Conservation of total mass and total energy
  \begin{equation}
    \label{consme}
    \langle r \rangle = 0 \quad \text{and} \quad \sqrt{\tfrac
      d2}\,\langle e \rangle + \langle \phi\,r \rangle = 0\,.
  \end{equation}

  \noindent
  $\rhd$ Global conservation laws associated to rotational
  symmetries of $\phi$
  \begin{equation}
    \label{consm}
    \Pp_\phi m=0\,.
  \end{equation}
  This also means $\Pp (m) \in \cR_\phi^\bot$ as
  in~\eqref{eq:consm}.

  \noindent
  $\rhd$ Global conservation laws corresponding to the
  \emph{harmonic directional modes}
  \begin{equation}
    \label{consd}
    \forall\,i \in I_\phi\,, \quad \langle r\,x_i \rangle = 0
    \quad\text{and}\quad\langle m_i \rangle = 0\,.
  \end{equation}

  \noindent
  $\rhd$ In the fully harmonic case $d_\phi=0$, global
  conservation laws corresponding to the \emph{harmonic pulsating
    modes}
  \begin{equation}
    \label{consp}
    \langle m \cdot x \rangle = 0\quad\text{and}\quad \sqrt{\tfrac
      d2}\,\langle e \rangle - \langle \phi\,r \rangle =0\,.
  \end{equation}
\end{corollary}

\subsection{The equations for the macroscopic modes}

We write the evolution equations for $r$, $m$ and $e$ defined
in~\eqref{Eqns:macro}. In the mathematical literature, such
equations are sometimes called \emph{local conservation laws} but
as this might introduce confusions with the \emph{local
  conservation law of the collision operator}, which we call here
\emph{collision invariants}, and the \emph{global conservation
  laws} studied in Section~\ref{ssec:conslaw}, we shall simply
refer to these equations as the \emph{equations for the
  macroscopic modes} or simply the \emph{macroscopic equations}.

Assume that $h$ solves~\eqref{eq:mainh}. For any Hermite
polynomial \hbox{$p: \R^d \to \R$} considered as a function of the
velocity variable, we compute
$J_p[h] = \int_{\R^d} p\,h\,\mu \dd v$ using standard properties
of Hermite functions:
\[
  J_p[h] = r \int_{\R^d} p\,\mu \dd v + m \cdot \int_{\R^d}
  v\,p\,\mu \dd v + e \int_{\R^d} \fE\,p\,\mu \dd v + J_p[h^\bot]
\]
and, using~\eqref{eq:mainh}, we also get
\begin{align*}
  \partial_t J_p[h] =
  & -\nabla_x r \cdot \int_{\R^d} v\,p\,\mu \dd v -\nabla_x m :
    \int_{\R^d} v \otimes v\,p\,\mu \dd v + m \cdot \nabla_x \phi
    \int_{\R^d} p\,\mu \dd v\\
  & - \nabla_x e \cdot \int_{\R^d} v\, \fE\,p\,\mu \dd v +
    \sqrt{\tfrac2d}\,e\,\nabla \phi \cdot \int_{\R^d} v\,p\,\mu
    \dd v + \int_{\R^d} (\cL\,h^\bot)\,p\,\mu \dd v\,.
\end{align*}
Plugging successively $p=1$, $v$, $\fE$,
$v \otimes v - \mbox{Id}_{d \times d}$ and
$v ( \fE - \sqrt{\frac{2}{d}})$ we get
\begin{subequations}
  \label{rme}
  \begin{align}
    & \partial_t r = \nabla_x^* \cdot m\,,\label{rmer}\\
    & \partial_t m = -\,\nabla_x r + \sqrt{\tfrac2d}\,\nabla_x^*
      e + \nabla_x^* \cdot E[h^\bot]\,,\label{rmem}\\
    & \partial_t e = -\,\sqrt{\tfrac2d}\,\nabla_x \cdot m +
      \nabla_x^* \cdot \Theta[h^\bot]\,,\label{rmee}\\[2mm]
    & \partial_t E[h] = -\,2\,\nabla_x^{\text{\tiny sym}}\,m +
    E[\cL\,h^\bot]\,,\label{rmeE}\\[2mm]
    & \partial_t \Theta[h] = -\left(1+\tfrac2d\right) \nabla_x e +
      \Theta[\cL\,h^\bot]\,,\label{rmeTheta}
  \end{align}
\end{subequations}
where $\nabla_x^{\text{\tiny sym}}\,m$ is defined
by~\eqref{SymmetricGradient} and the matrix valued function
$E[h]$ and the vector valued function $\Theta[h]$ are
higher-order moments of $h$ defined by
\begin{subequations}
\label{eq:EhEhperp}
\begin{align}
\label{eq:EhEhperp1}
&E[h] := \int_{\R^d} \left( v \otimes v - \mathrm{Id}_{d \times d} \right) h\,\mu \dd v = \sqrt{\tfrac2d}\,e\, \mathrm{Id}_{d \times d} + E[h^\bot]\,,\\
\label{eq:EhEhperp2}
&\Theta[h] := \int_{\R^d} v\, \Big(\fE(v)-\sqrt{\tfrac2d}\,\Big)\,h\,\mu \dd v = \Theta[h^\bot]\,.
\end{align}
\end{subequations}
If $f$ is a special macroscopic mode, then~\eqref{rme} is reduced to~\eqref{SMrmem} because, in that case, $h^\bot=0$ and $\Theta[\pih]=0$.

\section{Classification of the special macroscopic modes}
\label{sec:minimizers}
\setcounter{equation}{0}
\setcounter{theorem}{0}

In this section, we prove Part~(1) of Theorem~\ref{theo:main}. We
write $\mathsf a\lesssim\mathsf b$ if there is some positive
constant~$\mathsf c$ such that $\mathsf a\le\mathsf b\,\mathsf c$
and $\mathsf a\simeq\mathsf b$ if and only if
$\mathsf a\lesssim\mathsf b\lesssim\mathsf a$. Throughout this
section, we assume
that~\eqref{eq:kersC}--\eqref{eq:hyp-sg-C}--\eqref{eq:lbound}--\eqref{hyp:intnorm}--\eqref{hyp:regularity}--\eqref{eq:poincarenormal}--\eqref{eq:momentspace}--\eqref{eq:phiid}--\eqref{hyp:semigroup}
hold, without further notice.

\subsection{Statement and preliminary results}

Theorem~\ref{theo:main}--(1) writes:
\begin{proposition}[Special macroscopic modes]
  \label{prop:min}
  If $f$ is a solution of~\eqref{eq:mainMicroMacro}, then $h$
  given by~\eqref{eq:defh} is such that $h=0$.
\end{proposition}
We recall that $f$ is a special macroscopic mode if and only if,
by definition~\eqref{eq:mainMicroMacro}, $\sC f = 0$ and
$\partial_t f = \sT f$. With the definitions
of~\eqref{msm}--\eqref{hsm}--\eqref{rsm}--\eqref{dsm}--\eqref{psm},
the function
\[
  F:=f- \alpha\,\cM - \beta\,\cH\,\cM -
  F_{\mathrm{rot}}-F_{\mathrm{dir}}-F_{\mathrm{pul}}
\]
is also a special macroscopic mode and~\eqref{eq:mainMicroMacro2}
implies that $h$ defined in~\eqref{eq:defh} satisfies
$h=h^\parallel=r+ m \cdot v + e\,\fE(v)$. According
to~\eqref{eq:mainh}, $h$ solves the transport equation
$\partial_t h = \cT h$ because $\cC h=0$ (with $\cC$ defined
by~\eqref{eq:cC}) and Section~\ref{ssec:macrosols} proved that
$r$, $m$ and $e$ solve~\eqref{SMrmem}. Proposition~\ref{prop:min}
means that $r=0$, $m=0$, and $e=0$. We split the proof into
several steps. To start with, since $1$, $v$ and $\fE(v)$ are
orthonormal Hermite polynomials, we have
\begin{equation}
\label{Hermite-pih}
\|h\|^2=\|r\|^2+\|m\|^2+\|e\|^2\,.
\end{equation}
\begin{lemma}
\label{lem:piL2}
With the above notations, the function $h$ as in Proposition~\ref{prop:min} satisfies
\[
\frac{{\mathrm d}}{{\mathrm d}t} \| h \|^2 = 0\,.
\]
\end{lemma}
\begin{proof}[Proof of Lemma~\ref{lem:piL2}]
It follows from $\cC h = 0$, $\cT^* = -\,\cT$ and $\frac{{\mathrm d}}{{\mathrm d}t} \| h \|^2 = 2\,\langle h, \cT h\rangle = 0$.
\end{proof}

Collecting the results of Section~\ref{ssec:macrosols} and using
$r_0=\langle r \rangle = 0$, we get:
\begin{lemma}
  \label{lem:exprme}
  Consider the function $h$ as in
  Proposition~\ref{prop:min}. With the above notations, let
  $A := \langle \nabla_x^{\text{\tiny skew}}m \rangle$,
  $b(t) := \langle m \rangle$ and $c(t) := \langle e
  \rangle$. Then we have
  \begin{subequations}
    \label{eq:exp}
    \begin{align}
      \label{eq:expr}
      & r(t,x) = -\,x \cdot b'(t) +
        \tfrac{\xi_2(x)}{2\,\sqrt{2\,d}}\,c''(t) +
        \sqrt{\tfrac2d}\,\xi_\phi\,c(t)\,,\\
      \label{eq:expm}
      & m(t,x) = A\,x + b(t) - \tfrac x{\sqrt{2\,d}}\,c'(t)\,,\\
      \label{eq:expe}
      & e(t,x) = c(t)\,,
    \end{align}
  \end{subequations}
  where $A$ is a constant skew-symmetric matrix, while $b$
  and~$c$ are respectively vector valued and scalar functions of
  $t$ related by~\eqref{SMeq:Abcm}.
\end{lemma}
We recall that the functions $\xi_2$ and $\xi_\phi$ are defined
respectively by~\eqref{xi2}
and~\eqref{xiphi}. Equation~\eqref{SMeq:Abcm} in
Proposition~\eqref{Prop:EDO} provides us with various estimates
on $A$, $b$ and $c$ which are collected in
Sections~\ref{controlAm},~\ref{ssec:controlbcm}
and~\ref{ssec:controlbcmprimitive} in order to prove
Proposition~\ref{prop:min} in Section~\ref{Proof3.1}.

\subsection{Control of \texorpdfstring{$A$}{A}}
\label{controlAm}

For any $A\in\fM^{\text{\tiny skew}}_{d \times d}(\R)$, let us
define
\[
  |A|^2 := \int_{\R^d} |A\,x|^2\,\rho(x) \dd x\,.
\]
The vector space $\fM^{\text{\tiny skew}}_{d \times d}(\R)$ is of
finite dimension: all norms are equivalent to $|\cdot|$. This is
also why the rigidity constant $c_{\mathrm K}$ given
by~\eqref{grad} is positive, which implies the Korn-type
inequality
\begin{equation}
  \label{decayA2}
  \forall\,A \in \cR_\phi^\bot\,,\quad\left\| \nabla_x \phi \cdot
    A\,x \right\|^2 \geq c_{\mathrm K}\,| A |^2\,.
\end{equation}

By multiplying~\eqref{SMeq:Abcm} by $x_k$ for $k=1, \ldots, d$,
then integrating against $\rho(x)$ and performing some
integrations by part, using that $\rho(x) \dd x$ is centred and
that the terms invol\-ving $\nabla \phi$ vanish, we obtain
\begin{equation}
  \label{eq:Abcm2}
  \tfrac1{\sqrt{2\,d}}\,\langle 2\,\phi\,x \rangle\,c' +
  \tfrac1{2\,\sqrt{2\,d}}\,\langle |x|^2\,x \rangle\,c''' -
  \langle x \otimes x \rangle\,b'' = b\,.
\end{equation}
With the notation of~\eqref{eq:exp}, let us define $X$ and $Y$ by
\begin{equation}
  \label{eq:defX}
  X := \tfrac1{\sqrt{2\,d}}\,\big( 2\,\xi_\phi + \nabla_x \phi
  \cdot x - d \big)\, c + \tfrac{\xi_2}{2\,\sqrt{2\,d}}\,c'' - x
  \cdot b'
\end{equation}
and
\begin{equation}
  \label{eq:defY}
  Y := \sqrt{\tfrac2d}\,\langle \phi\,x \rangle\,c +
  \tfrac1{2\,\sqrt{2\,d}}\,\langle |x|^2\,x \rangle\,c'' -
  \langle x \otimes x \rangle\,b'\,.
\end{equation}
Identities~\eqref{SMeq:Abcm} and~\eqref{eq:Abcm2} yield
\[
  \frac{\mathrm d}{\mathrm dt} \left( X - Y \cdot \nabla_x \phi
  \right) = \nabla_x \phi \cdot A\,x
\]
where, according to Lemma~\ref{lem:exprme}, the r.h.s.~is
independent of $t$. As a consequence, we have the following
estimate.
\begin{lemma}
  \label{lem:decayA}
  Consider the function $h$ as in Proposition~\ref{prop:min}. The
  infinitesimal rotation matrix $A$ of Lemma~\ref{lem:exprme}
  satisfies
  \[
    \label{decayA1}
    -\frac{\mathrm d}{\mathrm dt} \big\langle \left( X - Y \cdot
      \nabla_x \phi \right) , \nabla_x \phi \cdot A\,x
    \big\rangle = -\,\| \nabla_x \phi \cdot A\,x \|^2\le
    -\,c_{\mathrm K}\,| A |^2\,.
  \]
\end{lemma}
\begin{proof}
  By the conservation law~\eqref{consm}, we know that
  $A\,x=\Pp(m)\in\cR_\phi^\bot$, so that~\eqref{decayA2}
  applies.
\end{proof}

\subsection{Control of \texorpdfstring{$b$, $b''$, $c'$ and
    $c'''$}{b, b'', c' and c'''}}
\label{ssec:controlbcm}

\begin{lemma}
  \label{bbccm}
  Consider the function $h$ as in Proposition~\ref{prop:min}. The
  functions~$b$ and $c$ as defined in Lemma~\ref{lem:exprme} are
  such that
  \begin{equation}
    \label{eq:bbccm}
    |b| + |b''| + |c'| + |c'''| \lesssim |A|\,.
  \end{equation}
\end{lemma}
\begin{proof}[Proof of Lemma~\ref{bbccm}]
  Multiplying~\eqref{SMeq:Abcm} by $\nabla_x \phi$ and
  integrating against $\rho(x)$, after integration by parts,
  using that $\rho$ is centred and observing that the terms
  involving $2\,\xi_\phi - d$ and $c'''$ vanish, it follows that
  \begin{equation}
    \label{eq:b''m}
    b'' = -\,\langle \nabla^2_x \phi \rangle\,b +
    \tfrac1{\sqrt{2\,d}}\,\langle \nabla^2_x \phi\,x \rangle\,c'
    + R_1 = -\,b + \tfrac1{\sqrt{2\,d}}\,\langle \nabla^2_x
    \phi\,x \rangle\,c' + R_1
  \end{equation}
  with $R_0 := -\,\nabla_x \phi\cdot A\,x$ and
  $R_1 := \langle R_0\,\nabla_x \phi \rangle =\cO(A)$. By
  inserting~\eqref{eq:b''m} into~\eqref{SMeq:Abcm}, one gets
  \begin{equation}
    \label{eq:bc-bism}
    \Psi_1(x)\,c' + \Psi_2(x)\,c''' - \Phi(x) \cdot b =R_2
  \end{equation}
  with
  \[
    \Phi (x) := \nabla_x \phi - \langle \nabla^2_x \phi
    \rangle\,x = \nabla_x \phi - x\,, 
  \]
  \[
    \Psi_1(x) := \tfrac{2\,\xi_\phi(x) + \nabla_x \phi \cdot x -
      d}{\sqrt{2\,d}}-x\cdot\tfrac{\langle \nabla^2_x \phi\,x
      \rangle}{\sqrt{2\,d}}\,, \quad \Psi_2(x) :=
    \tfrac{\xi_2(x)}{2\,\sqrt{2\,d}}\,,
  \]
  and $ R_2 := R_1 \cdot x - R_0$. Let
  \begin{equation}
    \label{eq:Mphim}
    M_\phi := \langle \Phi \otimes \Phi \rangle \in \fM_{d\times
      d}^{\text{\tiny sym}}(\R)\,,\quad \alpha_i := \langle
    \Psi_i\,\Phi \rangle \in\R^d
  \end{equation}
  with $i=1$, $2$. A multiplication of~\eqref{eq:bc-bism} by
  $\Phi$ and an integration against $\rho$ yields
  \begin{equation}
    \label{eq:Mphibm}
    M_\phi\,b = \alpha_1\,c' + \alpha_2\,c''' + R_3
  \end{equation}
  where $ R_3 := -\,\langle R_2\,\Phi \rangle = \cO(A)$ thanks to
  the moment bounds on $\phi$ deduced
  from~\eqref{eq:momentspace}. Inverting the matrix $M_\phi$
  allows to control $b$ by $c$ and $c''$, and
  rewrite~\eqref{eq:bc-bism} as an ordinary differential equation
  on $c$, up to an error term of the order of $\cO(A)$. If
  $M_\phi$ is not invertible, a similar estimate can still be
  done after taking into account the global conservation laws of
  Corollary~\ref{Cor:GlobalConservationLaws2}.

  We recall that $E_\phi$ is defined by~\eqref{eq:defephi}. We
  distinguish three cases.

  \medskip\noindent$\rhd$ \emph{Fully non-harmonic case
    $(d_\phi=d)$}. The matrix $M_\phi$ is invertible (see
  Lemma~\ref{lem:Mphi-caseI} in Appendix~\ref{app:tech})
  and~\eqref{eq:Mphibm} yields
  \begin{equation}
    \label{eq:b-caseIm}
    b = M_\phi^{-1}\left(\alpha_1\,c' + \alpha_2\,c''' +
      R_3\right)
  \end{equation}
  and hence, together with~\eqref{eq:bc-bism}, it follows that
  \begin{equation}
    \label{eq:Psi-c-caseIm}
    \widetilde \Psi_1(x)\,c' + \widetilde \Psi_2(x)\,c''' =
    R_4\,,
  \end{equation}
  with $ R_4 := R_2 + \Phi(x) \cdot M_\phi^{-1}R_3 $ and
  \begin{equation}
    \label{tildepsipsi}
    \widetilde \Psi_1(x) := \Psi_1(x) - \Phi(x) \cdot
    M_\phi^{-1}\,\alpha_1\,, \quad \widetilde \Psi_2(x) :=
    \Psi_2(x) - \Phi(x) \cdot M_\phi^{-1}\,\alpha_2\,.
  \end{equation}
  {}From Lemma~\ref{lem:tildePsi} we know that $\mathrm{Rank}
  (\widetilde \Psi_1 , \widetilde \Psi_2 ) = 2$ and deduce
  from~\eqref{eq:Psi-c-caseIm} that $c'=\cO(A)$ and
  $c'''=\cO(A)$. Using then~\eqref{eq:b-caseIm}
  and~\eqref{eq:b''m}, we also deduce $b=\cO(A)$ and
  $b''=\cO(A)$, and the proof is complete in this case.

  \medskip\noindent$\rhd$ \emph{Partially harmonic case
    $(1 \le d_\phi \le d-1)$}. Let
  $\{ e_1 , \ldots , e_{d_\phi} , e_{d_\phi+1} , \ldots , e_d \}$
  be a basis of $\R^d$ such that
  $\{ e_1 , \ldots , e_{d_\phi} \}$ generates $E_\phi$. For any
  vector $x \in \R^d$, we shall write $x = (\hat x , \check x )$
  with $\hat x \in \R^{d_\phi}$ and $\check x \in
  \R^{d-d_\phi}$. Similarly, we use the notation
  $\xi(x) = ( \hat \xi(x) , \check \xi(x))$ for a vector-field
  $\xi : \R^d \to \R^d$. In particular one has
  $\Phi = (\widehat\Phi , 0)$ and also $\check b =0$ so that
  $b=(\hat b, 0)$ as a consequence
  of~\eqref{consd}. Hence~\eqref{eq:bc-bism} becomes
  \begin{equation}
    \label{eq:bc-IIm}
    \Psi_1\,c' + \Psi_2\,c''' - \widehat\Phi \cdot \hat b = R_2\,.
  \end{equation}
  The matrix $M_\phi$ defined in~\eqref{eq:Mphim} is given by
  \[
    M_\phi =
    \begin{pmatrix}
      \widehat M_\phi & 0\\
      0 & 0
    \end{pmatrix}
  \]
  where
  \begin{equation}
    \label{eq:hatMphim}
    \widehat M_\phi := \langle \widehat\Phi \otimes \widehat\Phi
    \rangle \in \fM^{\text{\tiny sym}}_{d_\phi \times d_\phi}
    (\R)\,.
  \end{equation}
  Following the same procedure as in the fully non-harmonic case,
  we obtain after multiplication by $\widehat\Phi$ and
  integration against $\rho$ that
  \begin{equation}
    \label{eq:Mphib-IIm}
    \widehat M_\phi\,\hat b = \hat\alpha_1\,c' +
    \hat\alpha_2\,c''' + \hat R_3\,,
  \end{equation}
  with
  $\hat R_3 := -\,\langle R_2\,\widehat\Phi \rangle = \cO(A)$,
  $\hat \alpha_1 := \langle \Psi_1\,\widehat\Phi \rangle$ and
  $ \hat \alpha_2 := \langle \Psi_2\,\widehat\Phi \rangle$. The
  matrix $\widehat M_\phi$ is invertible (see
  Lemma~\ref{lem:Mphi-caseI} in Appendix~\ref{app:tech})
  and~\eqref{eq:Mphib-IIm} yields
  \begin{equation}
    \label{eq:b-caseIIm}
    \hat b = \widehat M_\phi^{-1}\,\big(\hat \alpha_1\,c' + \hat
    \alpha_2\,c''' + \hat R_3\big)\,.
  \end{equation}
  Hence, together with~\eqref{eq:bc-IIm}, it follows that
  \[
    \widehat \Psi_1 (x)\,c' + \widehat \Psi_2(x)\,c''' = \hat R_4
  \]
  with
  $\hat R_4 := R_2 + \widehat\Phi(x) \cdot \widehat
  M_\phi^{-1}\,\hat R_3 = \cO(A)$ and
  \begin{equation}
    \label{hatpsipsi}
    \widehat \Psi_1(x) := \Psi_1(x) - \widehat\Phi(x) \cdot
    \widehat M_\phi^{-1}\,\hat \alpha_1\,, \quad \widehat
    \Psi_2(x) := \Psi_2(x) - \widehat\Phi(x) \cdot \widehat
    M_\phi^{-1}\,\hat \alpha_2\,.
  \end{equation}
  As in the full rank case,
  $\mathrm{Rank} (\widehat \Psi_1 , \widehat \Psi_2 ) = 2$
  according to Lemma~\ref{lem:tildePsi} and we deduce that
  $c'=\cO(A)$ and $c'''=\cO(A)$. From~\eqref{eq:b-caseIIm}
  and~\eqref{eq:b''m}, we also get $\hat b=\cO(A)$ and
  $\hat b''=\cO(A)$, and since $\check b =0$ we eventually get
  $b=\cO(A)$ and $b''=\cO(A)$, which completes the proof of the
  case partially harmonic case.

  \medskip\noindent$\rhd$ \emph{Fully harmonic case
    $(d_\phi = 0)$}. We read from~\eqref{eq:expe}
  and~\eqref{consp} that $\langle e\rangle=c=0=c'=c''$ and
  from~\eqref{eq:expm},~\eqref{consd} and~\eqref{hyp:intnorm}
  that $\langle m\rangle=b=0=b'=b''$, which completes the proof
  of the case fully harmonic case.
\end{proof}

\subsection{Control of \texorpdfstring{$b'$, $c''$ and $c$}{b',
    c'' and c}}
\label{ssec:controlbcmprimitive}

\begin{lemma}
  \label{lem:Liapccc-bb}
  Consider the function $h$ as in Proposition~\ref{prop:min}. The
  functions~$b$ and $c$ as defined in Lemma~\ref{lem:exprme} obey
  the two differential inequalities
  \[
    \label{eq:Liapccc-bb-macro}
    \frac{\mathrm d}{\mathrm dt} \langle -b,b' \rangle \le
    -\,|b'|^2 + \cO(|A|^2)\quad\mbox{and}\quad \frac{\mathrm
      d}{\mathrm dt} \langle - c' , c'' \rangle \le -\,|c''|^2 +
    \cO(|A|^2)\,.
  \]
\end{lemma}
\begin{proof}[Proof of Lemma~\ref{lem:Liapccc-bb}]
  We write
  \[
    \frac{\mathrm d}{\mathrm dt} \langle -b,b' \rangle = \langle
    -b',b' \rangle + \langle -b,b''
    \rangle\quad\mbox{and}\quad\frac{\mathrm d}{\mathrm dt} \langle
    -c',c'' \rangle = \langle -c'',c'' \rangle + \langle -c',c'''
    \rangle
  \]
  and notice that $\langle -b,b'' \rangle=\cO(|A|^2)$ and
  $\langle -c',c''' \rangle=\cO(|A|^2)$ by~\eqref{eq:bbccm}.
\end{proof}

\begin{lemma}
  \label{lem:eq:c<bbccc}
  The function $c$ as defined in Lemma~\ref{lem:exprme} is such
  that
  \begin{equation}
    \label{eq:c<bbccc2}
    |c| \lesssim |b'| + |c''|
    \quad \mbox{ and } \quad 
    |c''| \lesssim |b'| + |c|\,.
  \end{equation}
\end{lemma}
\begin{proof}[Proof of Lemma~\ref{lem:eq:c<bbccc}.]
  Multiplying~\eqref{eq:expr} by $\xi_\phi$ and integrating
  against $\rho$, we obtain
  \[
    c \left(\sqrt{\tfrac2d}\left\langle \xi_\phi^2 \right\rangle
      + \sqrt{\tfrac d2}\,\right) = \langle
    x\,\xi_\phi\rangle\cdot
    b'-\tfrac{\langle\xi_2\,\xi_\phi\rangle}{2\,\sqrt{2\,d}}\,c''
  \]
  using $\langle r\,\phi\rangle=-\sqrt{\tfrac d2}\,c$ and
  $\langle r\rangle=0$ by~\eqref{consme}, which completes the
  proof.
\end{proof}

\subsection{A Lyapunov function method}
\label{Proof3.1}

We define the Lyapunov function
\[
  \cF [h] := \| h \|^2 -
  \varepsilon_A\,\big\langle(X-Y\cdot\nabla_x\phi),\nabla_x\phi\cdot
  A\,x\big\rangle - \varepsilon_b\,\langle b,b' \rangle -
  \varepsilon_c\,\langle c' , c'' \rangle\,,
\]
for some positive constants $\varepsilon_A$, $\varepsilon_b$ and
$\varepsilon_c$ to be chosen later.
\begin{lemma}
  \label{LyapunovSpecial}
  With the above notations, if $h$ is defined as in
  Proposition~\ref{prop:min}, then
  \begin{equation}
    \label{eq:cFgepi}
    \cF[h]\simeq\|h\|^2
  \end{equation}
  for $\varepsilon_A$, $\varepsilon_b$ and $\varepsilon_c$ small
  enough.
\end{lemma}
\begin{proof}[Proof of Lemma~\ref{LyapunovSpecial}]
  {}From~\eqref{Hermite-pih}, we know that
  $\|h\|^2=\|r\|^2+\|m\|^2+\|e\|^2$ and it follows
  from~\eqref{eq:expr},~\eqref{eq:expm} and~\eqref{eq:expe} that
  \begin{equation}
    \label{eq:mre}
    \|r\|^2\lesssim|b'|^2+|c|^2+|c''|^2\,,\quad \|m\|^2 \simeq
    |b|^2 + |A|^2 + |c'|^2 \quad \mbox{and} \quad \|e\|^2=|c|^2\,.
  \end{equation}
  By Lemma~\ref{bbccm}, we obtain
  \[
    \|m\|^2\lesssim|A|^2\,,\quad\big|\langle b,b'
    \rangle\big|\lesssim|A|^2+|b'|^2\,,\quad\big|\langle c',c''
    \rangle\big|\lesssim|A|^2+|c''|^2\,.
  \]
  By Lemma~\ref{lem:eq:c<bbccc}, we know that
  $|c|^2\lesssim|b'|^2+|c''|^2$ and obtain
  \[
    \|r\|^2\lesssim|b'|^2+|c''|^2\,, \quad \|c\|^2 \lesssim
    |b'|^2 + |c''|^2\,, \quad \|e\|^2 \lesssim |b'|^2 + |c''|^2\,,
  \]
  \[
    \big|\big\langle(X-Y\cdot\nabla_x\phi),\nabla_x\phi\cdot
    A\,x\big\rangle\big|\lesssim|A|^2+|b'|^2+|c''|^2\,.
  \]
  Altogether we have the upper estimate
  \begin{equation}
    \label{eq:FAbc}
    \cF[h]\lesssim|A|^2+|b'|^2+|c''|^2
  \end{equation}
  and, using~\eqref{eq:c<bbccc2},
  \[
    \cF[h]\lesssim|A|^2+|b'|^2+|c|^2\,.
  \]
  Using~\eqref{SMrmem-b}, we notice that
  \begin{equation}
    \label{eq:br}
    b'=\langle\partial_tm\rangle=-\,\langle\nabla_xr\rangle=-\,\langle
    r\,\nabla_x\phi\rangle\le\|r\|\,\|\nabla_x\phi\|\lesssim\|r\|
  \end{equation}
  performing one integration by parts and using Cauchy-Schwarz
  inequality. It is then clear that $|A|^2\lesssim\|m\|^2$ and
  $|c|^2=\|e\|^2$, so that by~\eqref{Hermite-pih},
  \[
    \cF[h]\lesssim\|r\|^2+\|m\|^2+\|e\|^2=\|h\|^2\,.
  \]

  Then, using~\eqref{eq:mre} again, we have the lower bound
  estimate
  \begin{multline*}
    2\,\cF[h]-\|h\|^2\gtrsim\|r\|^2+\big(|b|^2+|A|^2+|c'|^2\big)+|c|^2\\
    \hspace*{2cm}-\,2\,\varepsilon_A\,\big( |A|^2 + |b'|^2 +
    |c''|^2 \big) - 2\, \varepsilon_b\,\big(|A|^2+|b'|^2\big)
    -2\,\varepsilon_c\,\big(|A|^2+|c''|^2\big)\,.
  \end{multline*}
  We know from~\eqref{eq:br} that $|b'|^2\lesssim\|r\|^2$ and,
  using~\eqref{eq:c<bbccc2}, we also have that
  $|c''|^2\lesssim|b'|^2+|c|^2\lesssim\|r\|^2+|A|^2$. As a
  consequence, we have that
  \[
    2\,\cF[h]-\|h\|^2\ge0
  \]
  if $\varepsilon_A$, $\varepsilon_b$ and $\varepsilon_c$ are
  chosen small enough, which completes the proof.
\end{proof}

\begin{lemma}
  \label{LyapunovSpecialExpDecay}
  With the above notations, if $h$ is defined as in
  Proposition~\ref{prop:min}, then for some $\varepsilon_A$,
  $\varepsilon_b$ and $\varepsilon_c$ small enough, there is a
  positive constant $\lambda$ such that
  \[
    \frac{\mathrm d}{\mathrm dt} \cF [h] \le -\,\lambda\,\cF[h]\,.
  \]
\end{lemma}
\begin{proof}[Proof of Lemma~\ref{LyapunovSpecialExpDecay}]
  Using Lemma~\ref{lem:decayA}, Lemma~\ref{lem:Liapccc-bb}
  and~\eqref{decayA2}, we have
  \begin{multline*}
    -\,\frac{\mathrm d}{\mathrm
      dt}\cF[h]=-\,\varepsilon_A\,\|\nabla_x\phi\cdot
    A\,x\|^2+\varepsilon_b\,|b'|^2 + \varepsilon_c\, |c''|^2-
    (\varepsilon_b+\varepsilon_c) \, \cO(|A|^2)\\
    \ge\varepsilon'_A\,|A|^2 + \varepsilon_b\,|b'|^2 +
    \varepsilon_c\, |c''|^2 \gtrsim \cF[h]\,,
  \end{multline*}
  by choosing $\varepsilon_b$ and $\varepsilon_c$ small enough
  compared to $\varepsilon_A$ and using~\eqref{eq:FAbc} in the
  last inequality.
\end{proof}

\begin{proof}[Proof of Proposition~\ref{prop:min}]
  Let $h_0=h(t=0)$. Thanks to Gr\"onwall's lemma and the
  equivalence~\eqref{eq:cFgepi}, we deduce
  \[
    \forall\,t\ge0\,,\quad\| h(t) \|^2 \lesssim \cF[h(t)] \le
    e^{-\lambda\,t}\,\cF[h_0]\lesssim
    e^{-\lambda\,t}\,\|h_0\|^2\,.
  \]
  By Lemma~\ref{lem:exprme}, we know that $A$ is constant in
  time. Using for instance~\eqref{eq:mre}, we deduce from
  \[
    |A|^2 \lesssim \lim_{t \to +\infty}\| h(t) \|^2 = 0
  \]
  that $A = 0$. By Lemma~\ref{bbccm}, we get that $b=0$ and
  $c'=0$ for any $t\ge0$ so that $c$ is independent of
  $t$. Taking for instance~\eqref{eq:FAbc} into account, we
  conclude that $h = 0$.
\end{proof}

Completing the proof of Proposition~\ref{prop:min} means that
Part~(1) of Theorem~\ref{theo:main} is established.

\section{Proof of hypocoercivity by the micro-macro method}
\label{sec:micmac}
\setcounter{equation}{0}
\setcounter{theorem}{0}

In this section we prove Part (2) of Theorem~\ref{theo:main} on
hypocoercivity using the \emph{micro-macro decomposition} of the
solution as in Section~\ref{ssec:rescaling}. The proof of
Proposition~\ref{prop:min} is our a guideline for a new cascade
of estimates, but the analysis is however more complex due to the
presence of microscopic terms.

\subsection{Statement}
\label{ssec:Statement2}

Theorem~\ref{theo:main}, Part (2) can be rewritten as follows.
\begin{proposition}
  \label{prop:hypomm}
  Consider a solution $h$ to~\eqref{eq:defh}--\eqref{eq:mainh} in
  $\mathrm L^2(\cM)$ with initial datum $h_0$. There exist two
  positive constants $C$ and $\kappa$ such that
  \[
    \forall\,t\ge0\,,\quad\| h(t) \| \le C\,e^{-\kappa\,t}
    \left\| h_0 \right\|\,.
  \]
\end{proposition}
Here $C$ and $\kappa$ depend only on bounded moments constants,
spectral gap constants or explicitly computable quantities
associated to $\phi$ such as the rigidity constant defined
in~\eqref{grad}. We split $h$ into a microscopic part~$h^\bot$
and a macroscopic part $\pih$ such that
\[
  h = \pih+h^\bot = r + m \cdot v + e\,\fE(v) + h^\bot
\]
where $r$, $m$ and $e$ defined by~\eqref{Eqns:macro} evolve
according to the macroscopic equations~\eqref{rme} involving the
matrix valued function $E[h]$ and the vector valued function
$\Theta[h]$ defined by~\eqref{eq:EhEhperp1}
and~\eqref{eq:EhEhperp2}. By construction, we~have
\begin{equation}
  \label{eq:Nh=Nrmehperp}
  \| h \|^2 = \| r \|^2 + \| m \|^2 + \| e \|^2 + \| h^\bot
  \|^2\,.
\end{equation}
Let deviations from averages, or \emph{space} inhomogeneous,
terms be defined~by
\begin{subequations}
  \label{emrwws}
  \begin{align}
    & r_s := r - \langle \nabla_x r \rangle \cdot x -
      \tfrac1{2\,d}\,\langle \Delta_x r
      \rangle\,\xi_2 \label{rs}\,,\\
    & m_s := m - \langle \nabla_x^{\text{\tiny skew}}m \rangle\,x
      - \tfrac1d\,\langle \nabla_x \cdot m \rangle\,x - \langle m
      \rangle \label{ms}\,,\\
    & e_s := e - \langle e \rangle \label{es}\,,\\
    & w := r - \sqrt{\tfrac2d}\,\langle e
      \rangle\,\phi \label{w}\,,\\
    & w_s := r_s - \sqrt{\tfrac2d}\,\langle e \rangle\,\phi_s
      \quad \mbox{with} \quad \phi_s := \xi_\phi -
      \tfrac1{2\,d}\,\langle \Delta_x \phi
      \rangle\,\xi_2 \label{ws}\,.
  \end{align}
\end{subequations}
We recall that $\xi_2(x) := |x|^2 - \langle |x|^2 \rangle$ and
$\xi_\phi (x):= \phi-\langle \phi \rangle$ were already defined
in~\eqref{xi2} and~\eqref{xiphi} while
$\nabla_x^{\text{\tiny skew}}m$ refers
to~\eqref{SkewSymmetricGradient}. In particular
\begin{equation}
  \label{wws}
  w_s = w - \langle \nabla_x w \rangle\,x -
  \tfrac1{2\,d}\,\langle \Delta w \rangle\,\xi_2 +
  \sqrt{\tfrac2d}\,\langle e \rangle\,\langle \phi \rangle\,.
\end{equation}

After introducing some geometric tools in
Section~\ref{ssec:toolbox}, we split the proof of
Proposition~\ref{prop:hypomm} by considering infinite-dimensional
quantities in Section~\ref{ssec:hypomicro} and finite-dimen\-sional
quantities in Section~\ref{ssec:hypomacro}; in the latter the
analysis closely follows the strategy of
Section~\ref{sec:minimizers}. From now on, we assume that $h$ is
as in Proposition~\ref{prop:hypomm}.

\subsection{Witten-Hodge operator and Korn inequality: a toolbox}
\label{ssec:toolbox}

Here we collect several classical and less classical estimates
that will be used to control the macroscopic quantities. We refer
to~\cite{CDHMM21} for references and details of constructive
proofs. Assumptions~\eqref{hyp:intnorm}--\eqref{hyp:regularity}--\eqref{eq:poincarenormal}
coincide with the hypotheses of~\cite[Section~1.2]{CDHMM21}. Let
\[
\lfloor \nabla \phi \rceil := \sqrt{ 1+ |\nabla \phi|^2}\,.
\]
$\rhd$ The \emph{strong Poincar\'e inequality}
\begin{align}
  \label{eq:poincarestrong}
  \forall\,\varphi\in\mathrm H^1(\rho)\,,\quad\int_{\R^d} |
  \varphi - \langle \varphi \rangle |^2\,\lfloor \nabla_x \phi
  \rceil^2\,\rho\,\dd x \lesssim \int_{\R^d} | \nabla_x \varphi
  |^2\,\rho\,\dd x
\end{align}
is proven in~\cite[Proposition~5]{CDHMM21}.\\
$\rhd$ In order to work in $\mathrm L^2(\rho)$, we shall use the
operator $\Omega$ introduced in~\eqref{omegadef} and considered
as an operator acting either on scalar or vector-valued
functions. As a consequence of~\eqref{eq:poincarestrong}, we have
the \emph{zeroth-order strong Poincar\'e inequality}
(see~\cite[Proposition 8]{CDHMM21}) according to which, for any
$\varphi \in \mathrm L^2(\rho)$,
\begin{equation}
  \label{eq:boundOmega-1BIS}
  \big\|\,\Omega^{-1}\,\nabla^2_x \varphi \big\| +
  \big\|\,\Omega^{-1}\,\big(\lfloor \nabla_x \phi
  \rceil\,\nabla_x \varphi\big)\big\| +
  \big\|\,\Omega^{-1}\,\big(\lfloor \nabla_x \phi
  \rceil^2\,\varphi\big)\big\| \lesssim \| \varphi \|\,.
\end{equation}
$\rhd$ The following zeroth order Poincar\'e inequality,
sometimes called the \emph{Poincar\'e-Lions inequality},
\begin{equation}
  \label{eq:poincareL2}
  \forall\,\varphi\in\mathrm L^2(\rho)\,,\quad\| \varphi -
  \langle \varphi \rangle \| \lesssim
  \big\|\,\Omega^{-\frac12}\,\nabla_x \varphi \big\| \lesssim \|
  \varphi- \langle \varphi \rangle \|\,,
\end{equation}
is proven in~~\cite[Proposition 5]{CDHMM21}.\\
$\rhd$ The \emph{$(-1)^{th}$ order Poincar\'e-Lions inequality}
\begin{equation}
  \label{eq:poincareH-1}
  \forall\,\varphi\in\mathrm
  H^{-1}(\rho)\,,\quad\big\|\,\Omega^{-\frac12}\,\big( \varphi -
  \langle \varphi \rangle \big) \big\| \lesssim
  \big\|\,\Omega^{-1}\,\nabla_x \varphi \big\| \lesssim
  \big\|\,\Omega^{-\frac12}\,\big( \varphi - \langle \varphi
  \rangle \big) \big\|\,,
\end{equation}
is proven in~\cite[Lemma~10]{CDHMM21} as well as its variant
\begin{equation}
  \label{eq:poincareVar}
  \forall\,\varphi\in\mathrm L^2(\rho)\,,\quad\| \varphi - \langle
  \varphi \rangle \| \lesssim \big\|
  \nabla_x\,\Omega^{-\frac12}\,\varphi \big\| +
  \big\|\,\Omega^{-\frac12}\,\nabla_x \varphi \big\| \lesssim
  \big\| \varphi - \langle \varphi \rangle \big\|\,.
\end{equation}
$\rhd$ Another key estimate is the \emph{zeroth-order
  Korn-Poincar\'e inequality}: for any vector field
$u : \R^d \to \R^d$ such that $\langle u \rangle = 0$ and
$ \langle \nabla_x^{\text{\tiny skew}}u \rangle = 0$,
\begin{equation}
  \label{eq:Korn1}
  \| u \| \lesssim
  \big\|\,\Omega^{-\frac12}\,\nabla_x^{\text{\tiny sym}} u
  \big\|\,,
\end{equation}
which is established in~\cite[Theorem~4]{CDHMM21}
using~\eqref{skewsym}.

\subsection{Control of infinite-dimensional quantities}
\label{ssec:hypomicro}

We build an entropy function by assembling dissipative
functionals for $h^\bot$ and the space inhomogeneous terms
defined in~\eqref{emrwws}.

\subsubsection{Control of \texorpdfstring{$h^\bot$}{hbot}}

We first control the dissipation of the microscopic part.
\begin{lemma}
  \label{lem1:hperpLyap}
  If $h$ is a solution to~\eqref{eq:mainh} in $\mathrm L^2(\cM)$,
  then
  \begin{equation}
    \label{eq:hperpLyap}
    \frac{\mathrm d}{\mathrm dt} \| h \|^2 \le -\,2 \, \mathrm
    c_{\sC}\,\| h^\bot \|^2\,.
  \end{equation}
\end{lemma}
\begin{proof}[Proof of Lemma~\ref{lem1:hperpLyap}.]
Since $\cC^* = \cC$ and $\cT^* = -\,\cT$, there holds
\[
  \frac12 \frac{\mathrm d}{\mathrm dt} \| h \|^2 = \scalar{\cC h}h\,.
\]
We conclude that~\eqref{eq:hperpLyap} holds by the spectral gap
assumption~\eqref{eq:hyp-sg-C} on $\sC$.
\end{proof}

\subsubsection{Control of \texorpdfstring{$e_s$}{es}}

Let us consider $e_s$ as defined in~\eqref{es}.
\begin{lemma}
  \label{lem2:esLyap}
  There are some positive constants $\kappa_1$ and $C$ such that
  \begin{equation}
    \label{eq:esLyap}
    \frac{\mathrm d}{\mathrm dt} \left\langle
      \Omega^{-1}\,\nabla_x e , \Theta[h] \right\rangle \le
    -\,\kappa_1\,\|e_s\|^2 + C\,\| h^\bot \|\,\| h \|\,.
  \end{equation}
\end{lemma}
\begin{proof}[Proof of Lemma~\ref{lem2:esLyap}.]
  Recall that $\Theta[h] = \Theta[h^\bot]$
  from~\eqref{eq:EhEhperp2}. We compute
  \begin{align*}
    & \frac{\mathrm d}{\mathrm dt} \left\langle
      \Omega^{-1}\,\nabla_x e , \Theta[h] \right\rangle \\
    & = \left\langle \Omega^{-1}\,\nabla_x e\,,
      -\left(1+\tfrac2d\right) \nabla_x e + \Theta[\cL\,h^\bot]
      \right\rangle + \left\langle \Omega^{-1}\,\nabla_x
      (\partial_t e)\,, \Theta[h] \right\rangle \\
    &\le -\tfrac12\left(1+\tfrac2d\right)
      \big\|\,\Omega^{-\frac12}\,\nabla_x e \big\|^2 +
      C\,\big\|\,\Omega^{-\frac12}\,\Theta[\cL\,h^\bot] \big\|^2 +
      C\,\big\|\,\Omega^{-1}\,\nabla_x (\partial_t e) \big\|\,\|
      h^\bot \|\,,
  \end{align*}
  by using Cauchy-Schwarz and Young inequalities. We read
  from~\eqref{eq:poincareL2} that
  \[\|\,\Omega^{-1/2}\,\nabla_x e \|^2\gtrsim\|e_s\|^2\,.\] According
  to~\eqref{eq:poincareL2},~\eqref{eq:lbound}
  and~\eqref{eq:momentspace}, we have
  \[
    \big\|\Omega^{-\frac12}\,\Theta[\cL\,h^\bot] \lesssim \|h^\bot\|\,.
  \]
  It follows from~\eqref{rmee} that
  \[
    \Omega^{-1}\,\nabla_x (\partial_t e) =
    -\,\sqrt{\tfrac2d}\,\Omega^{-1}\,\nabla_x(\nabla_x \cdot m) +
    \Omega^{-1}\,\nabla_x\big(\nabla_x^* \cdot
    \Theta[h^\bot]\big)\,,
  \]
  so that
  $\big\|\,\Omega^{-1}\,\nabla_x (\partial_t e) \big\|\lesssim
  \|h\|$ by~\eqref{eq:boundOmega-1BIS}. This completes the proof
  of~\eqref{eq:esLyap}.
\end{proof}

\subsubsection{Control of \texorpdfstring{$m_s$}{ms}}

Let us consider $\nabla_x^{\text{\tiny sym}} m_s$ as defined
by~\eqref{SkewSymmetricGradient} and~\eqref{ms}.
\begin{lemma}
  \label{lem3:msLyap}
  There are some positive constants $\kappa_2$ and $C$ such that
  \begin{equation}
    \label{eq:msLyap}
    \frac{\mathrm d}{\mathrm dt} \Bigl\langle
    \Omega^{-1}\,\nabla_x^{\text{\tiny sym}} m_s , E[h] -
    \sqrt{\tfrac2d}\,\langle e \rangle\,\mathrm{Id}_{d \times d}
    \Bigr\rangle \le -\,\kappa_2\,\| m_s \|^2 + C\,\big( \| e_s
    \| + \| h^\bot \| \big)\, \| h \|\,.
  \end{equation}
\end{lemma}
\begin{proof}[Proof of Lemma~\ref{lem3:msLyap}]
  Let us remark that from~\eqref{ms} one has
  \[
    \nabla_x^{\text{\tiny sym}}\,m = \nabla_x^{\text{\tiny
        sym}}\,m_s + \tfrac1d\,\langle \nabla_x \cdot m
    \rangle\,\mathrm{Id}_{d \times d}\,,
  \]
  and from~\eqref{eq:EhEhperp1},
  \[
    E[h] - \sqrt{\tfrac2d}\,\langle e \rangle\,\mathrm{Id}_{d
      \times d} = \sqrt{\tfrac2d}\,e_s\,\mathrm{Id}_{d \times d}
    + E[h^\bot]\,.
  \]
  Moreover, from~\eqref{rmee}, one gets
  \[
    \frac{\mathrm d}{\mathrm dt} \langle e \rangle =
    -\,\sqrt{\tfrac2d}\,\langle \nabla_x \cdot m \rangle\,.
  \]
  As a consequence, from~\eqref{rmeE}, one obtains
  \begin{align*}
    & \frac{\mathrm d}{\mathrm dt} \Big\langle
      \Omega^{-1}\,\nabla_x^{\text{\tiny sym}}\,m_s , E[h] -
      \sqrt{\tfrac2d}\,\langle e \rangle\,\mathrm{Id}_{d \times
      d} \Big\rangle \\
    & = \left\langle \Omega^{-1}\,\nabla_x^{\text{\tiny sym}} m_s
      , - 2\,\nabla_x^{\text{\tiny sym}}\,m + E[\cL\,h^\bot ] +
      \tfrac2d \langle \nabla_x \cdot m \rangle\, \mathrm{Id}_{d
      \times d} \right\rangle \\
    &\hspace*{1cm} + \Big\langle
      \Omega^{-1}\,\nabla_x^{\text{\tiny sym}}(\partial_t m_s ) ,
      E[h] - \sqrt{\tfrac2d}\,\langle e \rangle\,\mathrm{Id}_{d
      \times d} \Big\rangle \\
    &= -\,2\,\big\|\,\Omega^{-\frac12}\,\nabla_x^{\text{\tiny
      sym}}\,m_s \big\|^2 + \Big\langle
      \Omega^{-\frac12}\,\nabla_x^{\text{\tiny sym}}\,m_s\,,
      \Omega^{-\frac12}\,E[\cL\,h^\bot] \Big\rangle \\
    &\hspace*{1cm} + \Big\langle
      \Omega^{-1}\,\nabla_x^{\text{\tiny sym}}(\partial_t m_s) ,
      \sqrt{\tfrac2d}\,e_s\,\mathrm{Id}_{d \times d} + E[h^\bot]
      \Big\rangle\,.
  \end{align*}
  Using the Cauchy-Schwarz inequality, we deduce
  \begin{multline*}
    \frac{\mathrm d}{\mathrm dt} \Big\langle
    \Omega^{-1}\,\nabla_x^{\text{\tiny sym}}\,m_s , E[h] -
    \sqrt{\tfrac2d}\,\langle e \rangle \mathrm{Id}_{d \times d}
    \Big\rangle \\
    \le -\,\big\|\,\Omega^{-\frac12}\,\nabla_x^{\text{\tiny
        sym}}\,m_s \big\|^2 +
    C\,\big\|\,\Omega^{-\frac12}\,E[\cL\,h^\bot] \big\|^2 \\
    +\,C\,\big\|\,\Omega^{-1}\,\nabla_x^{\text{\tiny
        sym}}(\partial_t m_s ) \big\|\,\Big\|
    \sqrt{\tfrac2d}\,e_s \mathrm{Id}_{d \times d} + E[h^\bot]
    \Big\|\,.
  \end{multline*}
  Using Korn's inequality~\eqref{eq:Korn1} and observing
  by~\eqref{rmem} that
  \[
    \big\|\,\Omega^{-1}\,\nabla_x^{\text{\tiny sym}}(\partial_t
    m_s ) \big\| = \cO(\| h \|)\quad\mbox{and}\quad
    \big\|\,\Omega^{-\frac12}\,E[\cL\,h^\bot] \big\| =
    \cO(\|h^\bot\|)
  \]
  from~\eqref{eq:boundOmega-1BIS} and~\eqref{eq:poincareVar} as
  in the proof of Lemma~\ref{lem2:esLyap}, we
  prove~\eqref{eq:msLyap}.
\end{proof}

\subsubsection{Control of \texorpdfstring{$w_s$}{ws}}

Let us consider $w_s$ as defined in~\eqref{ws}.
\begin{lemma}
  \label{lem4:wsLyap}
  There are some positive constants $\kappa_3$ and $C$ such that
  \begin{equation}
    \label{eq:wsLyap}
    \frac{\mathrm d}{\mathrm dt} \left\langle
      \Omega^{-1}\,\nabla_x w_s , m_s \right\rangle \le
    -\,\kappa_3\,\| w_s \|^2 + C\,\| e_s \|^2 + C\,\| h^\bot \|^2
    + C\,\| m_s \|\,\| h \|
  \end{equation}
  and
  \begin{equation}
    \label{eq:dtwsLyap}
    \frac{\mathrm d}{\mathrm dt} \left\langle
      -\,\Omega^{-1}\,\partial_t w_s , w_s \right\rangle \le
    -\,\big\|\,\Omega^{-\frac12}\,\partial_t w_s \big\|^2 + C\,\|
    w_s \|\,\| h \|\,.
  \end{equation}
\end{lemma}
\begin{proof}[Proof of Lemma~\ref{lem4:wsLyap}.]
  Observe that~\eqref{rmem},~\eqref{es} and~\eqref{wws} imply
  \begin{align*}
    \partial_t m
    & = -\,\nabla_x w + \sqrt{\tfrac2d}\,\nabla_x^* e_s +
      \nabla_x^* \cdot E[h^\bot] \\
    & = -\,\nabla_x w_s - \langle \nabla_x w \rangle -
      \tfrac1d\,\langle \Delta_x w \rangle\,x + \sqrt{\tfrac2d}\,
      \nabla_x^* e_s + \nabla_x^* \cdot E[h^\bot]\,.
  \end{align*}
  Integrating~\eqref{rmem} and using~\eqref{w}, one gets
  \[
    \frac{{\mathrm d}}{{\mathrm d}t} \langle m \rangle = -\,\langle
    \nabla_x r \rangle = -\,\langle \nabla_x w \rangle
  \]
  and
  \begin{align*}
    \frac{{\mathrm d}}{{\mathrm d}t} \langle \nabla_x \cdot m \rangle
    & = -\,\langle \Delta_x r \rangle + \sqrt{\tfrac2d}\,\langle
      e\,\Delta_x \phi \rangle + \big\langle \mathrm{Tr}
      (E[h^\bot]\,\nabla_{\!x}^2 \phi) \big\rangle \\
    & = -\,\langle \Delta_x w \rangle + \sqrt{\tfrac2d}\, \langle
      e_s\,\Delta_x \phi \rangle + \big\langle \mathrm{Tr}
      (E[h^\bot]\,\nabla_{\!x}^2 \phi) \big\rangle\,.
  \end{align*}
  Finally, by differentiating~\eqref{rmem}, one has
  \[
    \partial_t \nabla_x m = -\,\nabla^2_x r + \sqrt{\tfrac2d}\,
    \nabla_x^*\,\nabla_x e + \sqrt{\tfrac2d}\,\nabla^2_x \phi\,e
    + \nabla_x^* \cdot (\nabla_x \otimes E[h^\bot]) +
    E[h^\bot]\,\nabla_{\!x}^2 \phi
  \]
  and the integration of the skew-symmetric part yields
  \begin{equation}
    \label{eq:dtA}
    \frac{{\mathrm d}}{{\mathrm d}t} \langle \nabla_x^{\text{\tiny
        skew}}m \rangle = \big\langle \big(
    E[h^\bot]\,\nabla_x^2\phi \big)^{\text{\tiny skew}}
    \big\rangle\,.
  \end{equation}
  As a consequence of these identities and~\eqref{ms}, one gets
  \[
    \partial_t m_s = -\,\nabla_x w_s +
    \sqrt{\tfrac2d}\,\Big(\nabla_x^* e_s - \tfrac1d\,\langle
    e_s\,\Delta_x \phi \rangle\,x \Big)+m_E
  \]
  where $m_E:=\nabla_x^* \cdot E[h^\bot]-\big\langle \big(
  E[h^\bot]\,\nabla_x^2\phi \big)^{\text{\tiny skew}}
  \big\rangle\,x - \tfrac1d\,\big\langle \mathrm{Tr}\big(
  E[h^\bot]\,\nabla_x^2\phi \big) \big\rangle\,x$. Hence
  \begin{align*}
    \left\langle\Omega^{-1}\,\nabla_xw_s,\partial_t m_s\right\rangle=
    & -\,\big\|\Omega^{-\frac12}\,\nabla_xw_s\big\|^2
      +\,\big\|\Omega^{-\frac12} \, \nabla_xw_s \big\| \, \big\| 
      \Omega^{-\frac12} \, m_E\big\| \\
    & +\,\sqrt{\tfrac2d}\, \big\|\Omega^{-\frac12} \, \nabla_xw_s
      \big\| \, \big\|\Omega^{-\frac12}\, \big( \nabla_x^* e_s -
      \tfrac1d \, \langle e_s\,\Delta_x \phi \rangle\,x\big)\big\|
  \end{align*}
  Using the zeroth order Poincar\'e
  inequality~\eqref{eq:poincareL2} and~\eqref{es}, we can
  estimate $\|\Omega^{-1/2}\,\nabla_x^* e_s\|$ by $\|e_s\|$. Up
  to a few integrations by parts,
  using~\eqref{hyp:regularity},~\eqref{eq:momentspace}
  and~\eqref{eq:phiid}, we end up for some constant $C>0$ with
  \begin{equation}
    \label{eq:dtms}
    \left\langle\Omega^{-1}\,\nabla_xw_s,\partial_t
      m_s\right\rangle\le -\,\frac12 \, \big\|\Omega^{-\frac12}
    \, \nabla_xw_s \big\|^2 + C \left(\|e_s\|^2+\|h^\bot\|^2\right).
  \end{equation}

  {}From the definitions~\eqref{rs} and~\eqref{ws}, we have
  \[
    \partial_t w_s = \partial_t r - \langle \nabla_x \partial_t r
    \rangle \cdot x - \tfrac1{2\,d}\,\langle \Delta_x \partial_t
    r \rangle\,\xi_2, - \sqrt{\tfrac2d}\,\langle \partial_t e
    \rangle\,\phi_s
  \]
  so that, by~\eqref{rmer} and~\eqref{rmee},
  \begin{equation}
    \label{eq:dtws}
    \partial_t w_s = \nabla_x^* m - \langle \nabla_x \nabla_x^*
    \cdot m \rangle \cdot x - \tfrac1{2\,d}\,\langle \Delta_x
    \nabla_x^* \cdot m \rangle\,\xi_2 - {\tfrac2d} \langle
    \nabla_x \cdot m \rangle\,\phi_s\,.
  \end{equation}
  Using~\eqref{eq:boundOmega-1BIS} in order to estimate the first
  term, and performing several integration by parts and using the
  boundedness assumption~\eqref{eq:momentspace} on $\phi$ in
  order to estimate the three last terms, we obtain
  \begin{equation}
    \label{eq:Omega-1dtws}
    \|\Omega^{-1}\,\nabla_x\,\partial_t w_s\|
    \lesssim\|m\|\lesssim\|h\|\,.
  \end{equation}
  Inserting~\eqref{eq:dtms} and~\eqref{eq:Omega-1dtws}
  in
  \[
    \frac{\mathrm d}{\mathrm dt}
    \left\langle \Omega^{-1} \, \nabla_xw_s, m_s \right\rangle =
    \left\langle \Omega^{-1}\,\nabla_x\,\partial_t w_s,
      m_s\right\rangle + \left\langle\Omega^{-1}\,\nabla_xw_s,\partial_t
      m_s\right\rangle
  \]
  completes the proof of~\eqref{eq:wsLyap}.

  In order to control the time-derivative of $w_s$, we write
  \begin{equation}
    \label{eq:Omega-1d2ttwsws}
    \frac{\mathrm d}{\mathrm dt} \big\langle
    -\Omega^{-1}\,\partial_t w_s , w_s \big\rangle =
    -\,\big\|\,\Omega^{-\frac12}\,\partial_t w_s \big\|^2 -
    \big\langle \Omega^{-1}\,\partial^2_{tt} w_s\,, w_s
    \big\rangle\,.
  \end{equation}
  Differentiating~\eqref{eq:dtws} with respect to $t$, we have
  \[
    \partial^2_{tt} w_s = \nabla_x^* \cdot (\partial_t m) -
    \langle \nabla_x \nabla_x^* \cdot (\partial_t m) \rangle
    \cdot x - \tfrac1{2\,d}\langle \Delta_x \nabla_x^* \cdot
    (\partial_t m) \rangle\,\xi_2 - {\tfrac2d} \langle \nabla_x
    \cdot (\partial_t m) \rangle\,\phi_s\,,
  \]
  where the first term is obtained by
  differentiating~\eqref{rmem} and amounts to
  \[
    \nabla_x^* (\partial_t m) = -\,\nabla^*_x \cdot \nabla_x w +
    \sqrt{\tfrac2d}\,\nabla^*_x \cdot \nabla_x^* e_s + \nabla^*_x
    \cdot\nabla_x^* \cdot E[h^\bot]\,,
  \]
  using~\eqref{es} and~\eqref{w}. Similar expressions hold for
  the three next terms. Arguing similarly as
  for~\eqref{eq:Omega-1dtws}, we have
  \[
    \Omega^{-1}\,\partial^2_{tt} w_s = \cO ( \| h \|)\,.
  \]
  Together with~\eqref{eq:Omega-1d2ttwsws} this
  proves~\eqref{eq:dtwsLyap}.
\end{proof}

\subsubsection{First Lyapunov functional}

We end this section by introducing a first, partial Lyapunov
functional
\begin{multline}
  \label{eq:F1h}
  \cF_1 [h] := \| h \|^2 + \varepsilon_1 \left\langle
    \Omega^{-1}\,\nabla_x e , \Theta[h] \right\rangle +
  \varepsilon_2\,\big\langle\Omega^{-1}\,\nabla_x^{\text{\tiny
      sym}}\,m_s , E[h] - \sqrt{\tfrac2d}\,\langle e
  \rangle\,\mathrm{Id}_{d \times d} \big\rangle \\
  + \varepsilon_3 \left\langle \Omega^{-1}\,\nabla_x w_s , m_s
  \right\rangle + \varepsilon_4\,\langle -\Omega^{-1}\,\partial_t
  w_s , w_s \rangle
\end{multline}
where
\begin{equation}
  \label{varepsilon}
  \varepsilon_1=\varepsilon\,,\quad \varepsilon_2=\varepsilon^{3/2}\,,
  \quad \varepsilon_3 = \varepsilon^{7/4}\,,\quad
  \varepsilon_4=\varepsilon^{15/8}\,.
\end{equation}
Let us define the dissipation functional
\begin{equation}
  \label{eq:D1h}
  \cD_1 [h] := \| h^\perp \|^2 + \| e_s \|^2 + \| m_s \|^2 + \|
  w_s\|^2 + \|\,\Omega^{-\frac12}\,\partial_t w_s \big\|^2\,.
\end{equation}
\begin{lemma}
  \label{lem5:F1Lyap}
  There are some positive constants $\kappa_0$, $C_0$ and
  $\kappa$ such that for any $\varepsilon > 0$ small enough, we
  have
  \[
    \label{eq:dtF1mMh}
    \frac{{\mathrm d}}{{\mathrm d}t} \cF_1 [h] \le -\,\kappa_0\, \|
    h^\bot \|^2 - \varepsilon^\frac{15}8\,\kappa\,\cD_1[h] +
    \varepsilon^2\,C_0\,\| h \|^2\,.
  \]
\end{lemma}
\begin{proof}[Proof of Lemma~\ref{lem5:F1Lyap}.]
  By collecting the results of
  Lemmata~\ref{lem1:hperpLyap},~\ref{lem2:esLyap},~\ref{lem3:msLyap}
  and~\ref{lem4:wsLyap}, we obtain
  \[
    \label{dF1dt0}
    \begin{aligned}
      \frac{\mathrm d}{\mathrm dt} \mathcal F_1 [h] \le
      -\,2\,\mathrm c_{\sC}\,\| h^\bot \|^2
      &- \varepsilon_1\,\kappa_1\,\| e_s \|^2 +
      \varepsilon_1\,C\,\| h^\bot \|\,\| h \| \\
      &- \varepsilon_2\,\kappa_2\,\| m_s \|^2 +
      \varepsilon_2\,C\,\big(\| h^\bot \| + \| e_s \|\big)\, \| h
      \| \\
      & - \varepsilon_3\,\kappa_3\,\| w_s \|^2 +
      \varepsilon_3\,C\,\big(\| e_s \|^2 + \| h^\bot \|^2 + \|
      m_s \|\,\| h \| \big) \\
      &- \varepsilon_4\,\|\,\Omega^{-\frac12}\,\partial_t w_s
      \|^2 + \varepsilon_4\,C\,\| w_s \|\,\| h \|
    \end{aligned}
  \]
  for any $\varepsilon_i \in (0,1)$, $i=1,2,3,4$, up to a
  renaming of the generic constant $C>0$. Using repeatedly
  Young's inequality, we have
  \begin{align*}
    & \varepsilon_1\,C\,\| h^\bot \|\,\| h \| \le
      \tfrac12\,\mathrm c_{\sC}\,\| h^\bot \|^2 +
      \varepsilon_1^2\,\tfrac{C^2}{2\,\mathrm c_{\sC}}\,\| h
      \|^2\,, \\
    &\varepsilon_2\,C\,\| h^\bot \|\,\| h \| \le
      \tfrac12\,\mathrm c_{\sC}\,\| h^\bot \|^2 +
      \varepsilon_2^2\,\tfrac{C^2}{2\,\mathrm c_{\sC}}\,\| h
      \|^2\,, \\
    &\varepsilon_2\,C\,\| e_s \|\,\| h \| \le
      \tfrac12\,\varepsilon_1\,\kappa_1\,\| e_s \|^2 +
      \tfrac{\varepsilon_2^2}{\varepsilon_1}\,\tfrac{C^2}{2\,\kappa_1}\,\|
      h \|^2\,, \\
    &\varepsilon_3\,C\,\| m_s \|\,\| h \| \le
      \tfrac12\,\varepsilon_2\,\kappa_2\,\| m_s \|^2 +
      \tfrac{\varepsilon_3^2}{\varepsilon_2}\,\tfrac{C^2}{2\,\kappa_2}\,\|
      h \|^2\,, \\
    & \varepsilon_4\,C\,\| w_s \|\,\| h \| \le
      \tfrac12\,\varepsilon_3\,\kappa_3\,\| w_s \|^2 +
      \tfrac{\varepsilon_4^2}{\varepsilon_3}\,\tfrac{C^2}{2\,\kappa_3}\,\|
      h \|^2\,,
  \end{align*}
  and therefore
  \begin{multline*}
    \label{dFdt1}
    \frac{\mathrm d}{\mathrm dt} \mathcal F_1 [h] \le
    -\left(\mathrm c_{\sC} - \varepsilon_3\,C \right) \| h^\bot
    \|^2 - \varepsilon_1\,\kappa_1\,\big(\tfrac12 -
    \tfrac{\varepsilon_3}{\kappa_1\,\varepsilon_1}\,C \big) \|
    e_s \|^2 - \tfrac12\,\varepsilon_2\,\kappa_2\,\| m_s \|^2 \\
    - \tfrac12\,\varepsilon_3\,\kappa_3\,\| w_s \|^2 -
    \varepsilon_4\,\big\|\,\Omega^{-\frac12}\,\partial_t w_s
    \big\|^2 \\
    +\,\tfrac12\,C^2\left(\tfrac{\varepsilon_1^2}{\mathrm
        c_{\sC}} + \tfrac{\varepsilon_2^2}{\mathrm c_{\sC}} +
      \tfrac{\varepsilon_2^2}{\kappa_1\,\varepsilon_1} +
      \tfrac{\varepsilon_3^2}{\kappa_2\,\varepsilon_2} +
      \tfrac{\varepsilon_4^2}{\kappa_3\,\varepsilon_3}\right) \|
    h \|^2\,.
  \end{multline*}
  The choice $\kappa_0=\tfrac14\,\mathrm c_{\sC}$,
  $\kappa=\min\left\{\tfrac{\mathrm
      c_{\sC}}4,\tfrac{\kappa_1}4,\tfrac{\kappa_2}2,\tfrac{\kappa_3}2,1\right\}$
  and
  $C_0=\tfrac12\,C^2\,\big(\tfrac2{\mathrm
    c_{\sC}}+\tfrac1{\kappa_1}+\tfrac1{\kappa_2}+\tfrac1{\kappa_3}\big)$
  with $\varepsilon_i$, $i=1,\dots,4,$ given
  by~\eqref{varepsilon} and
  \[
    0<\varepsilon<\min\left\{ 1, (\tfrac{4\,C}{\kappa_1})^{-4/3},
      (\tfrac{2\,C}{\mathrm c_{\sC}})^{-4/7} \right\}
  \]
  completes the proof.
\end{proof}

\subsection{Control of finite-dimensional quantities}
\label{ssec:hypomacro}

After estimating the decay of the \emph{deviations from averages
  terms} defined~by~\eqref{emrwws}, let us consider the
time-dependent \emph{global scalar} quantities
$\langle e \rangle$,
$\langle \nabla_x^{\text{\tiny skew}}m \rangle$,
$\langle \nabla_x \cdot m \rangle$, $\langle m \rangle$,
$\langle \nabla_x r \rangle$, $\langle \Delta_x r \rangle$. We
proceed as in the proof of Proposition~\ref{prop:min}. Let
\begin{subequations}
  \label{eq:defAbcz}
  \begin{align}
    \label{eq:defAbc}
    & A(t) := \langle \nabla_x^{\text{\tiny skew}}m
      \rangle\,,\quad b (t):= \langle m \rangle\,,\quad c (t) :=
      \langle e \rangle\,, \\
    \label{eq:defz}
    & z(t,x) := r(t,x) + b'(t) \cdot x -
      c''(t)\,\tfrac{\xi_2(x)}{2\,\sqrt{2\,d}} -
      c(t)\,\sqrt{\tfrac2d}\,\xi_\phi(x)\,.
  \end{align}
\end{subequations}
By comparison with~\eqref{eq:expr}, we know that $z=0$ if $h$
corresponds to a special macroscopic mode. We observe here that
there is no reason for $A(t)$ to be neither the orthogonal
projection of $m$ onto infinitesimal rotation matrices nor
independent of $t$. We shall have to take this fact into account
later and remember that, according to~\eqref{eq:dtA},
\begin{equation}
  \label{eq:dtA2}
  A'(t) = \big\langle \big( E[h^\bot]\,\nabla_x^2\phi
  \big)^{\text{\tiny skew}} \big\rangle\,.
\end{equation}

\subsubsection{The macroscopic equations}

As defined by~\eqref{Eqns:macro}, the functions $r$, $m$ and $e$
can be rewritten in the new variables as follows.

\begin{lemma}
  \label{lem:hypoexprme}
  With previous notations, if $h$
  solves~\eqref{eq:defh}--\eqref{eq:mainh} in $\mathrm L^2(\cM)$,
  then
  \begin{subequations}
    \label{eq:exphypo}
    \begin{align}
      \label{eq:exprhypo}
      & r(t,x) = -\,b'(t) \cdot x +
        c''(t)\,\tfrac{\xi_2(x)}{2\,\sqrt{2\,d}} +
        c(t)\,\sqrt{\tfrac2d}\,\xi_\phi(x) + z(t,x)\,,\\[0.5mm]
      \label{eq:expmhypo}
      & m(t,x) = A(t)\,x + b(t) - c'(t)\,\tfrac1{\sqrt{2\,d}}\,x
        + m_s(t,x)\,,\\[2mm]
      \label{eq:expehypo}
      & e(t,x) = c(t) + e_s(t,x)\,,
    \end{align}
  \end{subequations}
  where $z$ obeys the bounds
  \begin{subequations}
    \label{eq:BDweh}
    \begin{align}
      \label{eq:zBDweh}
      & \| z \|^2 \lesssim \| w_s \|^2 + \| e_s \|^2 + \| h^\bot
        \|^2,\\[2mm]
      \label{eq:dtzBDweh}
      & \big\|\,\Omega^{-\frac12}\,\partial_t z \big\|^2 \lesssim
        \big\|\,\Omega^{-\frac12}\,\partial_t w_s \big\|^2 + \|
        m_s \|^2 + \| h^\bot \|^2\,.
    \end{align}
  \end{subequations}
\end{lemma}
\begin{proof}[Proof of Lemma~\ref{lem:hypoexprme}.]
  The expression~\eqref{eq:exprhypo} follows from the definition
  of $z$ in~\eqref{eq:defz}, while~\eqref{eq:expehypo} is no more
  than a rewriting of~\eqref{es}. From~\eqref{rmee} one observes
  that
  \begin{equation}
    \label{c'=(Dm)}
    c' = \frac{\mathrm d}{\mathrm dt} \langle e \rangle =
    -\,\sqrt{\tfrac2d}\,\langle \nabla_x \cdot m \rangle\,,
  \end{equation}
  so that~\eqref{eq:expmhypo} follows from the
  definition~\eqref{ms} of $m_s$.

  {}From~\eqref{rmem} we have
  \begin{equation}
    \label{b'=(Dr)}
    b' = \frac{\mathrm d}{\mathrm dt} \langle m \rangle =
    -\,\langle \nabla_x r \rangle\,.
  \end{equation}
  Using~\eqref{rs} and~\eqref{ws}, we write
  \begin{align*}
    r & = w_s + \sqrt{\tfrac2d}\,\langle e \rangle\,\phi_s +
        \langle \nabla_x r \rangle \cdot x +
        \tfrac1{2\,d}\,\langle \Delta_x r \rangle\,\xi_2 \\
      &= w_s + \sqrt{\tfrac2d}\,c\,\big( \xi_\phi -
        \tfrac1{2\,d}\,\langle \Delta_x \phi \rangle\,\xi_2 \big)
        - b' \cdot x + \tfrac1{2\,d}\,\langle \Delta_x r
        \rangle\,\xi_2\,.
  \end{align*}
  {}From~\eqref{eq:exprhypo}, we deduce
  \begin{equation}
    \label{zws}
    z = w_s + \tfrac1{\sqrt{2\,d}} \, \Big(\tfrac1{\sqrt{2\,d}}
    \, \langle \Delta_x r \rangle - \tfrac12\,c'' - \tfrac1d \, \langle
    \Delta_x \phi \rangle\,c \Big)\,\xi_2\,.
  \end{equation}
  Finally, thanks to~\eqref{c'=(Dm)} and~\eqref{rmem}, we compute
  \[
    c'' = -\,\sqrt{\tfrac2d} \left\langle \nabla_x \cdot
      \partial_t m \right\rangle = \sqrt{\tfrac2d} \left\langle
      \Delta_x r \right\rangle - \tfrac2d \left\langle \nabla_x
      \cdot \nabla_x^* e \right\rangle -
    \sqrt{\tfrac2d}\,\big\langle \nabla_x \cdot \big( \nabla_x^*
    \cdot E[h^\bot] \big) \big\rangle\,,
  \]
  and thus obtain
  \begin{equation}
    \label{eq:c''=(D)}
    c'' = \sqrt{\tfrac2d}\,\langle \Delta_x r\rangle -
    \tfrac2d\,\langle e\,\Delta\,\phi \rangle -
    \sqrt{\tfrac2d}\,\big\langle \mathrm{Tr} \big(E
    [h^\bot]\,\nabla_{\!x}^2 \phi \big) \big\rangle\,.
  \end{equation}
  By inserting~\eqref{eq:c''=(D)} in~\eqref{zws}, we obtain
  \[
    z = w_s + \tfrac1{2\,d}\,\big\langle E[h^\bot] : \nabla_{\!x}^2
    \phi \big\rangle\,\xi_2 + \tfrac1{d\,\sqrt{2\,d}} \left\langle
      e_s\,\Delta_x \phi \right\rangle\,\xi_2\,,
  \]
  from which~\eqref{eq:zBDweh} follows. By differentiating $z$
  with respect to $t$, we get
  \[
    \partial_t z = \partial_t w_s + \tfrac1{2\,d}\,\big\langle
    \partial_tE[h^\bot] : \nabla_{\!x}^2 \phi \big\rangle\,\xi_2
    + \tfrac1{d\,\sqrt{2\,d}} \left\langle
      \partial_te_s\,\Delta_x \phi \right\rangle\,\xi_2\,.
  \]
  By~\eqref{rmeE} and~\eqref{eq:EhEhperp1}, we know that
  \[
    \partial_t E[h] = -\,2\,\nabla_x^{\text{\tiny sym}}\,m +
    E[\cL\,h^\bot]= \sqrt{\tfrac2d}\,\partial_te\, \mathrm{Id}_{d
      \times d} + \partial_tE[h^\bot]\,.
  \]
  Besides, we learn from~\eqref{ms} and~\eqref{c'=(Dm)} that
  \[
    \nabla_x^{\text{\tiny sym}}\,m=\nabla_x^{\text{\tiny
        sym}}\,m_s+\tfrac1d\,\langle\nabla_x\cdot
    m\rangle\,\mathrm{Id}_{d\times d}=\nabla_x^{\text{\tiny
        sym}}\,m_s-\tfrac1{\sqrt{2\,d}}\,\langle \partial_te
    \rangle\,\mathrm{Id}_{d\times d}
  \]
  and, as a consequence,
  \[
    \partial_tE[h^\bot]=-\,2\,\nabla_x^{\text{\tiny sym}}\,m_s+
    E[\cL\,h^\bot]-\sqrt{\tfrac2d}\,\partial_te_s\,
    \mathrm{Id}_{d \times d}\,.
  \]
  Hence
  \[
    \partial_t z = \partial_t w_s - \tfrac1d\,\big\langle
    \nabla_x^{\text{\tiny sym}}\,m_s : \nabla_{\!x}^2 \phi
    \big\rangle\,\xi_2 + \tfrac1{2\,d}\,\big\langle
    E[\cL\,h^\bot] : \nabla_{\!x}^2 \phi \big\rangle\,\xi_2
  \]
  and Estimate~\eqref{eq:dtzBDweh} follows using an integration
  by parts and~\eqref{eq:momentspace}.
\end{proof}

Using~\eqref{eq:exprhypo} on the one hand, and~\eqref{rmer}
combined with~\eqref{eq:expmhypo} on the other hand, we write
\begin{align*}
  \partial_t r
  & = \sqrt{\tfrac2d}\,\xi_\phi\,c' - x \cdot b'' +
    \tfrac{\xi_2}{2\,\sqrt{2\,d}}\,c''' +\partial_t z \\
  &= \nabla_x^* \cdot m_s + \nabla_x \phi\cdot A\,x -
    \tfrac1{\sqrt{2\,d}}\,(\nabla_x \phi \cdot x - d)\,c' +
    \nabla_x \phi \cdot b\,.
\end{align*}
We deduce a differential equation which is very similar
to~\eqref{SMeq:Abcm} up to additional terms involving $m_s$
and~$z$, namely
\begin{proposition}
  \label{Prop:EDO2}
  The functions $A$, $b$, $c$, $z$ and $m_s$ defined
  by~\eqref{eq:defAbcz} and~\eqref{ms} solve
  \begin{multline}
    \label{eq:Abchypo}
    \tfrac{2\,\xi_\phi(x) + \nabla_x \phi \cdot x -
      d}{\sqrt{2\,d}}\,c' + \tfrac{\xi_2}{2\,\sqrt{2\,d}}\,c''' -
    \nabla_x \phi \cdot b - x \cdot b'' -\nabla_x \phi\cdot A\,x \\
    = \nabla_x^* \cdot m_s - \partial_t z\,.
  \end{multline}
\end{proposition}

\subsubsection{Control of \texorpdfstring{$A$}{A}}

The counterpart of Section~\ref{controlAm} goes as follows.
\begin{lemma}
  \label{lem:decayAhypo}
  There are some positive constants $\kappa_4$ and $C$ such that
  \begin{equation}
    \label{decayAhypo}
    -\frac{\mathrm d}{\mathrm dt} \big\langle \left(X - Y \cdot
      \nabla_x \phi\right) , \nabla_x \phi \cdot A\,x \big\rangle
    \le-\,\kappa_4\, |A|^2 + C\left(\cD_1[h] + C\,\| h^\bot
      \|\,\| h \|\right),
  \end{equation}
  where $X = X(b',c,c'')$, $Y = Y(b',c,c'')$ and $\cD_1[h]$ are
  defined respectively in~\eqref{eq:defX},~\eqref{eq:defY}
  and~\eqref{eq:D1h}.
\end{lemma}
\begin{proof}[Proof of Lemma~\ref{lem:decayAhypo}]
  We argue as for Lemma~\ref{lem:decayA}. We
  multiply~\eqref{eq:Abchypo} by $x_k$ for $k=1, \ldots, d$ and
  after integration, we get
  \begin{equation}
    \label{eq:Abc2}
    \sqrt{\tfrac2d}\,\big\langle \xi_\phi\,x\big\rangle\,c' +
    \tfrac1{2\,\sqrt{2\,d}}\,\langle \xi_2\,x\,\rangle\,c'''- b -
    \langle x\otimes x \rangle\,b'' = \langle m_s \rangle -
    \langle x\,\partial_t z \rangle\,.
  \end{equation}
  Using the definitions of $X$ and $Y$,~\eqref{eq:Abchypo}
  and~\eqref{eq:Abc2} yield
  \[
    \frac{\mathrm d}{\mathrm dt} ( X - Y \cdot \nabla_x \phi) =
    \nabla_x \phi\cdot A\,x + \nabla_x^* \cdot m_s - \langle m_s
    \rangle \cdot \nabla_x \phi - \partial_t z + \langle
    x\,\partial_t z \rangle \cdot \nabla_x \phi\,.
  \]
  Using~\eqref{eq:defAbc} and~\eqref{eq:dtA2}, we obtain
  \begin{align*}
    & \frac{\mathrm d}{\mathrm dt} \big\langle -(X - Y \cdot
      \nabla_x \phi) , \nabla_x \phi\cdot A\,x \big\rangle \\
    & = -\,\| \nabla_x \phi\cdot A\,x \|^2 - \langle \nabla_x^*
      \cdot m_s , \nabla_x \phi \cdot A\,x \rangle + \big\langle
      \langle m_s \rangle \cdot \nabla_x \phi\,, \nabla_x
      \phi\cdot A\,x\big\rangle \\
    & \quad +\,\langle \partial_t z , \nabla_x \phi\cdot
      A\,x\rangle - \big\langle \langle x\,\partial_t z \rangle
      \cdot \nabla_x \phi\,, \nabla_x \phi\cdot A\,x \big\rangle \\
    & \quad - \Big\langle \big(X - Y \cdot \nabla_x \phi \big) ,
      \nabla_x \phi \cdot \big\langle \big( E
      [h^\bot]\,\nabla_{\!x}^2 \phi \big)^{\text{\tiny skew}}
      \big\rangle \, x \Big\rangle\,.
  \end{align*}
  For the first term and thanks to the conservation
  law~\eqref{consm}, we note that
  \[
    \cR_\phi^\bot \ni \Pp(m) = A\,x + \Pp(m_s)\,,
  \]
  so that we can apply inequality~\eqref{decayA2} to
  $x \mapsto A\,x+\Pp(m_s)$ to get
  \[
    c_{\mathrm K}\, \| A\,x + \Pp(m_s) \|^2 \leq \| \nabla_x
    \phi\cdot A\,x + \nabla_x \phi \cdot \Pp(m_s)\|^2\,,
  \]
  which yields
  \begin{equation}
    \label{controlAhypo}
    c_{\mathrm K}\, |A|^2= c_{\mathrm K}\, \| A\,x \|^2 \leq
    4\,\| \nabla_x \phi\cdot A\,x \|^2 + C\,\|m_s\|^2
  \end{equation}
  for any $C>4+c_{\mathrm K}$. In order to estimate the other
  terms, we use
  \begin{multline*}
    \Big\langle \big(X - Y \cdot \nabla_x \phi \big) , \nabla_x
    \phi \cdot \big\langle \big( E [h^\bot]\,\nabla_{\!x}^2 \phi
    \big)^{\text{\tiny skew}} \big\rangle \, x \Big\rangle \\
    \lesssim \big\|\,\Omega^{-1}\,(X - Y \cdot \nabla_x \phi)
    \big\|\,\big\|\,\Omega\,\big(\nabla_x \phi\,x\big)
    \big\|\,\big| \big\langle \big( E [h^\bot]\,\nabla_{\!x}^2
    \phi
    \big)^{\text{\tiny skew}} \big\rangle\big| \\
    \lesssim \big(|b'|+|c|+|c''| \big)\,\| h^\bot \|\,,
  \end{multline*}
  and
  \[
    \big\langle \langle m_s \rangle \cdot \nabla_x \phi\,, \nabla_x
    \phi \cdot A\,x \big\rangle \lesssim \| m_s \|\,\| \nabla_x \phi
    \cdot A\,x \| \lesssim \| m_s \|\,|A|\,.
  \]
  Thanks to the zeroth order Poincar\'e
  inequality~\eqref{eq:poincareL2}, we also have
  \[
    \big\langle \nabla_x^* \cdot m_s , \nabla_x \phi\cdot A\,x
    \big\rangle = \big\langle \Omega^{-\frac12}\,\big(\nabla_x^*
    \cdot m_s \big)\,,\,\Omega^{\frac12}\,\big(\nabla_x \phi\cdot
    A\,x \big) \big\rangle \lesssim \| m_s \|\,|A|
  \]
  as well as similar estimates for the terms in $\partial_t z$
  and $ \langle x\,\partial_t z \rangle \cdot \nabla_x
  \phi$. Collecting these estimates with~\eqref{controlAhypo}, we
  get
  \begin{multline*}
    \frac{\mathrm d}{\mathrm dt} \big\langle - \big(X - Y \cdot
    \nabla_x \phi \big) , \nabla_x \phi\cdot A\,x \big\rangle \\
    \le -\,\tfrac14\,c_{\mathrm K}\,|A|^2 +C\,\big(\,\| m_s \| +
    \big\|\,\Omega^{-\frac12}\,\partial_t z \big\|\,\big)\,|A| +
    \big(|b'|+|c|+|c''| \big)\,\| h^\bot \|\,,
  \end{multline*}
  for some $\kappa_4<\tfrac14\,c_{\mathrm K}$ and $C>0$ large
  enough. Young's inequality and~\eqref{eq:dtzBDweh} conclude the
  proof of~\eqref{decayAhypo}.
\end{proof}

\subsubsection{Control of \texorpdfstring{$b$, $b'$, $b''$ and
    $c$, $c'$, $c''$ and $c'''$}{b b' b'' c c' c''}}

It follows from three lemmata.
\begin{lemma}
  \label{bbccmhypo}
  The following estimate holds
  \begin{equation}
    \label{eq:bbccmhypo}
    |b| + |b''| + |c'| + |c'''| \lesssim |A| +
    \big\|\,\Omega^{-\frac12}\,\partial_t w_s \big\| + \| m_s \|
    + \| h^\bot \|\,.
  \end{equation}
\end{lemma}
\begin{proof}
  Using~\eqref{eq:Abchypo}, we can write
  \[
    \label{eq:bc}
    \tfrac{2\,\xi_\phi + \nabla_x \phi \cdot x -
      d}{\sqrt{2\,d}}\,c' + \tfrac{\xi_2}{2\,\sqrt{2\,d}}\,c''' -
    \nabla_x \phi \cdot b - x \cdot b'' =R_0
  \]
  with
  $R_0 := \nabla_x \phi\cdot A\,x + \nabla_x^* \cdot m_s -
  \partial_t z$. Arguing as in the proof of Lemma~\ref{bbccm}
  with this new definition of $R_0$, we obtain
  that~\eqref{eq:Mphibm} holds with
  \[
    R_3 := -\,\langle R_0\, \widetilde\Phi \rangle\,,\quad
    \widetilde\Phi := \nabla \phi - x - \langle (\nabla \phi - x
    ) \otimes x \rangle\,\nabla \phi\,.
  \]
  Observing that
  \begin{multline*}
    R_3 = -\,\big\langle \nabla_x \phi\cdot A\,x\,\widetilde\Phi
    \big\rangle - \big\langle\,{}^T\! \big(D\,
    \widetilde\Phi\big)\,m_s\big\rangle + \big\langle
    \Omega^{\frac12}\,\widetilde\Phi,
    \Omega^{-\frac12}\,\partial_t z \big\rangle \\
    = \cO \big(\,|A| + \| m_s \| +
    \big\|\,\Omega^{-\frac12}\,\partial_t z \big\|\,\big)
  \end{multline*}
  and using~\eqref{eq:dtzBDweh}, the proof
  of~\eqref{eq:bbccmhypo} follows for the same reasons as in the
  proof of Lemma~\ref{bbccm}.
\end{proof}

With $\cD_1[h]$ defined in~\eqref{eq:D1h}, we obtain the
counterpart of Lemma~\ref{lem:Liapccc-bb}.
\begin{lemma}
  \label{lem:Liapccc-bbhypo}
  There exists a constant $C > 0$ so that
  \[
    \label{eq:Liapccc-bb}
    \begin{aligned}
      & \frac{{\mathrm d}}{{\mathrm d}t} \langle -b,b' \rangle \le
        -\,|b'|^2 + C\,|A|^2 + C\,\cD_1[h]\,, \\
      & \frac{{\mathrm d}}{{\mathrm d}t} \langle - c' , c'' \rangle \le
        -\,|c''|^2 + C\,|A|^2 + C\,\cD_1[h]\,.
    \end{aligned}
  \]
\end{lemma}
\begin{lemma}
  \label{lem:boundchypo}
  The following estimates hold
  \begin{align}
    \label{eq:boundchypo}
    & |c| \lesssim |b'| + |c''| + \| w_s \|^2 + \| e_s \|^2 + \|
      h^\bot \|^2, \\
    \label{eq:boundrhypo}
    & \| r \| \lesssim |b'| + |c''| + \| w_s \|^2 + \| e_s \|^2 +
      \| h^\bot \|^2\,.
  \end{align}
\end{lemma}
\begin{proof}
  According to~\eqref{eq:exprhypo}, we can write $r$ as
  \begin{equation}
    \label{eq:r=c}
    r = \sqrt{\tfrac2d}\,\xi_\phi\,c + R_5
  \end{equation}
  where
  $R_5 := z - x \cdot b' + \tfrac{\xi_2}{2\,\sqrt{2\,d}}\,c'' =
  \cO( |b'| + |c''| + \cD_1[h])$ because
  of~\eqref{eq:zBDweh}. Using this expression in~\eqref{consme}
  and recalling that $\langle e \rangle=c$ yields
  \[
    \sqrt{\tfrac d2}\,c \left(1 + \tfrac2d\,\langle \xi_\phi^2
      \big\rangle\right) =- \left\langle \xi_\phi\,R_5
    \right\rangle\,,
  \]
  from which~\eqref{eq:boundchypo} follows. Coming back
  to~\eqref{eq:r=c}, we establish~\eqref{eq:boundrhypo}.
\end{proof}

\subsubsection{Second Lyapunov functional}

Let us introduce the Lyapunov function
\begin{equation}
  \label{eq:F2h}
  \cF_2[h] := \cF_1[h] - \varepsilon_5\, \langle ( X - Y \cdot
  \nabla_x \phi ) , \nabla_x \phi \cdot A\,x \rangle -
  \varepsilon_6 \,\langle b,b' \rangle- \varepsilon_6 \,\langle
  c' , c'' \rangle
\end{equation}
for some additional small parameters $\varepsilon_5$ and
$\varepsilon_6$, and the associated dissipation functional
\begin{equation}
  \label{eq:D2h}
  \cD_2[h] := \cD_1[h] + |A|^2 + |b'|^2 + |c''|^2\,.
\end{equation}
\begin{lemma}
  \label{lem:equivD2h}
  For any
  $0 < \varepsilon_6 < \varepsilon_5 < \varepsilon_4 <
  \varepsilon_3 < \varepsilon_2 < \varepsilon_1$ with
  $\varepsilon_1$ small enough, there holds
  \begin{equation}
    \label{eq:equivD2h}
    \| h \|^2 \lesssim \cF_2[h] \lesssim \cD_2[h] \lesssim \| h
    \|^2\,.
\end{equation}
\end{lemma}
\begin{proof}[Proof of Lemma~\ref{lem:equivD2h}]
  We can control all quantities involved in the definitions of
  $\cF_2$ and $\cD_2$ by $\| h \|^2$. Indeed,
  from~\eqref{eq:Nh=Nrmehperp} and~\eqref{eq:defAbc}, we have
  \begin{equation}
    \label{eq:rmehbcLESSh}
    \| r \| + \| m \| + \| e \| + \| h^\bot\| + |b| + |c|
    \lesssim \| h \|
  \end{equation}
  and thus also
  $\| e_s \| \lesssim \| e \| + |c| \lesssim \| h \|$
  from~\eqref{eq:expehypo}. Next, we observe from~\eqref{c'=(Dm)}
  that
  \[
    |c'| = \sqrt{\tfrac2d}\;|\langle m \cdot \nabla \phi \rangle|
    \lesssim \| m \| \le \| h \|\,,
  \]
  from~\eqref{eq:defAbc} that
  \[
    |A| = |\langle m \nabla_x^{\text{\tiny skew}}\phi \rangle|
    \lesssim \| m \| \le \| h \|\,,
  \]
  and thus also
  $\| m_s \| \lesssim \| m \| + |A| + |b| + |c'| \lesssim \| h
  \|$ from~\eqref{eq:expmhypo}. Similarly, we observe
  from~\eqref{b'=(Dr)} that
  \[
    |b'| = |\langle r \nabla_x \phi \rangle| \lesssim \| r \| \le
    \| h \|\,.
  \]
  Coming back to the definition of $w_s$ and
  using~\eqref{b'=(Dr)}, we get
  \[
    w_s = r - \sqrt{\tfrac2d}\,c\,\big( \xi_\phi -
    \tfrac1{2\,d}\,\langle \Delta_x \phi \rangle\,\xi_2 \big) +
    b' \cdot x - \tfrac1{2\,d}\,\left\langle
      r\,\big(|\nabla\phi|^2 - \Delta \phi\big)
    \right\rangle\xi_2\,,
  \]
  and deduce
  $\| w_s \| \lesssim \| r \| + |c| + |b'| \lesssim \| h
  \|$. Similarly, from~\eqref{eq:c''=(D)}, we also have
  $|c''| \lesssim \| r \| + \| e \| + \| h^\bot \| \le \| h
  \|$. Summing up, we have proved
  \begin{equation}
    \label{eq:esmswsAcbcLESSh}
    \| e_s \| + \| m_s \| + \| w_s \| + |A| + |c'| + |b'| + |c''|
    \lesssim \| h \|\,.
  \end{equation}
  
  We finally have to control the terms
  $\|\,\Omega^{-1/2}\,\partial_t w_s
  \|$. From~\eqref{ws},~\eqref{rs} and~\eqref{rme}, we have
  \[
    \partial_t w_s = \nabla_x^* \cdot m - \langle \nabla_x
    \nabla_x^* \cdot m \rangle \cdot x -
    \tfrac1{2\,d}\,\left\langle \Delta_x \nabla_x^* \cdot m
    \right\rangle\,\xi_2 -\sqrt{\tfrac2d}\, c' \phi_s
  \]
  and, after performing several integration by parts,
  \begin{equation}
    \label{eq:OdtweLESSh}
    \left\|\Omega^{-\frac12}\,\partial_t w_s \right\| \lesssim \|
    m \| + |c'| \lesssim \| h \|\,. 
  \end{equation}
  As a consequence of the
  estimates~\eqref{eq:rmehbcLESSh},~\eqref{eq:esmswsAcbcLESSh},~\eqref{eq:OdtweLESSh}
  and of the definition~\eqref{eq:F2h} of $\cF_2$ (also
  see~\eqref{eq:F1h}), we have
  \[
    \left| \| h \|^2 - \cF_2[h] \right| \le C\,\varepsilon_1\,\|
    h \|^2\,.
  \]
  This completes the proof of the first equivalence
  in~\eqref{eq:equivD2h}. For the same reason, we have
  $\cD_2[h] \lesssim \| h \|^2$. On the other way round,
  from~\eqref{eq:exprhypo} and~\eqref{eq:boundchypo}, we have
  \[
    \| r \| \lesssim |b'| + |c''| + |c| + \| w_s \| + \| e_s \| +
    \| h^\bot \| \lesssim |b'| + |c''| + \| w_s \| + \| e_s \| +
    \| h^\bot \|
  \]
  and similarly, from~\eqref{eq:expmhypo} and~\eqref{eq:bbccmhypo}, we have
  \[
    \| m \| \lesssim |A| + |b| + |c'| + \| m_s \| \lesssim |A| +
    |c'| + \| m_s \| + \|\,\Omega^{-\frac12}\,\partial_t w_s \| +
    \| h^\bot \|\,.
  \]
  Combining the last two
  estimates,~\eqref{eq:Nh=Nrmehperp},~\eqref{eq:expehypo} and the
  definition~\eqref{eq:D1h} of~$\cD_2$, we deduce the reverse
  inequality $\| h \|^2 \lesssim \cD_2[h]$, which completes the
  proof of the second equivalence in~\eqref{eq:equivD2h}.
\end{proof}

\subsection{Proof of Proposition~\ref{prop:hypomm}}
\label{ssec:hypomacro2}

\begin{proof}[Proof of Proposition~\ref{prop:hypomm}]
  We differentiate with respect to $t$ the Lyapunov function
  $\cF_2[h]$ and use
  Lemmata~\ref{lem5:F1Lyap},~\ref{lem:decayAhypo}
  and~\ref{lem:Liapccc-bbhypo} to get
  \begin{align*}
    \frac{\mathrm d}{\mathrm dt} \mathcal F_2 [h] \le
    & \, -\,\kappa_0\, \| h^\bot \|^2 -
      \varepsilon^\frac{15}8\,\kappa\,\cD_1[h] +
      \varepsilon^2\,C_0\,\| h \|^2 \\
    & \,-\,\kappa_4\,\varepsilon_5\, |A|^2 +
      C\,\varepsilon_5\left(\cD_1[h] + C\,\| h^\bot \|\,\| h
      \|\right) \\
    & \,+\,\,\varepsilon_6\left(-\,|b'|^2-\,|c''|^2 + 2\,C\,|A|^2
      + 2\,C\,\cD_1[h]\right).
  \end{align*}
  Using Young's inequality we have
  \[
    \varepsilon_5\,C^2\,\| h^\bot \|\,\| h \| \le
    \kappa_0\,\|h^\bot\|^2+\varepsilon_5^2\,\kappa_0^{-1}\,C^4\,\|h\|^2
  \]
  and we deduce for some new constants $C_1$, $C_2$, $C_3>0$ that
  \begin{align*}
    \frac{\mathrm d}{\mathrm dt} \mathcal F_2 [h] \le
    & \,-\,\varepsilon^\frac{15}8\,\kappa\,\cD_1[h] +
      \varepsilon^2\,C_0\,\| h \|^2 \\
    & \,-\,\kappa_4\,\varepsilon_5\, |A|^2 +
      C_1\,(\varepsilon_5+\varepsilon_6)\,\cD_1[h]
      +C_2\,\varepsilon_5^2\,\|h\|^2 \\
    & \,-\,\varepsilon_6 \left( |b'|^2+|c''|^2 \right) + C_3\,
      \varepsilon_6 \,|A|^2.
  \end{align*}
  As in the proof of Lemma~\ref{lem5:F1Lyap}, we choose
  appropriately the small parameters~$\varepsilon_i$ such that
  the quantities $\varepsilon_5$, $\varepsilon_6$,
  $\varepsilon_5/\varepsilon_4$ and $\varepsilon_6/\varepsilon_5$
  are small enough in terms of $\varepsilon$. With
  $\varepsilon_5 := \varepsilon^{61/32}$ and
  $\varepsilon_6 := \varepsilon^{62/32}$, we obtain
  \begin{align*}
    \frac{{\mathrm d}}{{\mathrm d}t} \mathcal F_2 [h] \le
    & -\, \varepsilon^{15/8} \left(\kappa -
      2\,\varepsilon^{1/32}\,C_1\right) \cD_1[h] -
      \varepsilon^{61/32} \left( \kappa_4 -
      \varepsilon^{1/32}\,C_3 \right)|A|^2 \\
    & -\, \varepsilon^{62/32} \left( |b'|^2+|c''|^2 \right) +
      \varepsilon^2 \left( C_0+ C_2\,\varepsilon^{29/16} \right)
     \| h \|^2\,.
  \end{align*}
  Choosing $\varepsilon>0$ small enough, the differential
  inequality simplifies into
  \[
    \frac{{\mathrm d}}{{\mathrm d}t} \mathcal F_2 [h] \le -\,
    \varepsilon^{62/32}\,\cD_2[h] + 5\,C\,\varepsilon^2\,\| h
    \|^2\,.
  \]
  Because of the equivalences established in
  Lemma~\ref{lem:equivD2h}, there are two constants $K_i > 0$,
  such that
  \[
    \frac{{\mathrm d}}{{\mathrm d}t} \mathcal F_2 [h] \le
    -\,\varepsilon^{62/32} \left(K_1 - K_2\,\varepsilon^{2/32}
    \right)\cF_2[h]\,.
  \]
  Choosing $\varepsilon > 0$ smaller if necessary, we obtain
  \[
    \frac{{\mathrm d}}{{\mathrm d}t} \mathcal F_2 [h] \le -\,\kappa\,
    \cF_2[h]\,,
  \]
  for some $\kappa > 0$, which implies
  $\cF_2[h(t)] \le e^{-\kappa\,t} \cF_2[h_0]$. This completes the
  proof of Proposition~\ref{prop:hypomm}, that is, of Part (2) of
  Theorem~\ref{theo:main}, by using once again the equivalences
  of Lemma~\ref{lem:equivD2h}.
\end{proof}

\section{Proof of hypocoercivity by the commutator method}
\label{sec:cascade}
\setcounter{equation}{0}

In this section we give an alternative proof of our main result
in Theorem~\ref{theo:main} using a commutator method, under the
additional hypotheses that
\begin{equation}
  \label{hyp:boundedCollisionOperator}
  \tag{H9}
  \mbox{\emph{The linear collision operator $\sC$ is bounded in
      $\mathrm L^2(\mu^{-1})$}}\,,
\end{equation}
\begin{equation}
  \label{hyp:PotentialRestrictions}
  \tag{H10}
  \mbox{\emph{$\|\nabla_x^2\,\phi\|_{\mathrm
        L^\infty(\R^d)}<\infty$ and $\displaystyle \int_{\R^d} x\,
      \phi(x)\,e^{-\phi(x)} \dd x =0$}}\,.
\end{equation}
Assumption~\eqref{hyp:boundedCollisionOperator} means that the
operator $\cC$ as defined in~\eqref{eq:cC} is bounded on
$\mathrm L^2(\mu)$ while~\eqref{hyp:PotentialRestrictions} means
that the potential $\phi$ has bounded second derivatives and is
superlinear at infinity. These assumptions are added merely in
order to simplify the computations but the
bound~\eqref{hyp:boundedCollisionOperator} on the collision
operator can be relaxed into just~\eqref{eq:lbound} by using
$\tilde A_i := \Pi^* A_i \Pi$ instead of the $A_i$'s defined
in~\eqref{eq:Ai} in Proposition~\ref{prop:error} (where
$\Pi h = h^\parallel$ is the orthogonal projection on the
macroscopic part), and the bound
in~\eqref{hyp:PotentialRestrictions} can be relaxed into simply
$|\nabla_x ^2 \phi| \lesssim 1 + |\nabla_x \phi|$ thanks to the
additional weight $\lfloor \nabla_x \phi \rceil$ in the Poincaré
inequality~\eqref{eq:poincarenormal}.  We include the commutator
method because of its interesting algebraic properties and
potential applications to a larger class of equations. Then
part (2) of Theorem~\ref{theo:main} writes:
\begin{proposition}
  \label{prop:hypoc}
  Assume that~\eqref{eq:kersC}--\eqref{hyp:PotentialRestrictions}
  hold and consider a solution~$h$
  to~\eqref{eq:defh}--\eqref{eq:mainh} in $\mathrm
  L^2(\cM)$. Then there are explicit constants $C>0$ and $\kappa>0$ such that
  \[
    \| h(t) \| \leq C\,e^{-\kappa\,t} \left\| h_0 \right\|
  \]
  where $C$ and $\kappa$ depend only on bounded moments constants,
  spectral gap constants or explicitly computable quantities
  associated to $\phi$ such as the rigidity constant defined
  in~\eqref{grad}.
\end{proposition}
While $\nabla_x$ and $\nabla_v$ map scalar functions to
vector-valued functions, their adjoints in $\mathrm L^2(\cM)$ are
$\nabla_x^* = -\,\nabla_x \cdot +\,\nabla_x \phi$ and
$\nabla_v^* = -\,\nabla_v \cdot +\,v\,\cdot$ and map
vector-valued functions back to scalar functions. The operators
$\nabla_x$ and $\nabla_v$ commute but each does not commute
with its adjoint. We have
\begin{equation}
  \label{eq:commutators-T}
  \begin{aligned}
    & [\nabla_v,\cT] = -\,\nabla_x\,,\quad [\nabla_x,\cT] =
    H_\phi\,\nabla_v\,, \\
    & [\nabla_v^*,\cT]=-\,\nabla_x^*\,,\quad [\nabla_x^*,\cT] =
    \left( H_\phi \,\nabla_v \right)^*,
  \end{aligned}
\end{equation}
where $[A,B]=A\,B-B\,A$ is the commutator and
$H_\phi := (\partial^2_{x_ix_j} \phi)_{i,j}$.

In addition to $\Omega = \nabla_x^* \cdot \nabla_x+1$ defined
in~\eqref{omegadef}, we also introduce
\[
  \Gamma = \nabla_v^* \cdot \nabla_v +1\,,\quad \Lambda =
  \nabla_v^* \cdot \nabla_v + \nabla_x^* \cdot \nabla_x+1\,.
\]
These scalar operators also act, coordinate by coordinate, on
tensors.

{}From, \emph{e.g.},~\cite{CDHMM21} (see
Sections~\ref{ssec:hypomacro}-\ref{ssec:hypomacro2})
or~\cite{HN05,HN04}, these operators are self-adjoint in
$\mathrm L^2(\cM)$. As in Section~\ref{sec:micmac}, we construct
a cascade of estimates.

\subsection{Cascade of infinite-dimensional
  correctors}
\label{ssec:correctors}

The three following operators play the role of \emph{correctors}:
\begin{equation}
  \label{eq:Ai}
  \begin{dcases}
    A_0 := \nabla_x^* \otimes \nabla_x^* \otimes
    \nabla_x^*\,\Lambda^{-\frac32} : \Lambda^{-\frac32}\,\nabla_v
    \otimes \nabla_x \otimes \nabla_x\,,\\[2mm]
    A_1 := \nabla_x^* \otimes \nabla_v^* \otimes
    \nabla_x^*\,\Gamma^{-\frac12}\,\Lambda^{-1} :
    \Lambda^{-1}\,\Gamma^{-\frac12}\,\nabla_v \otimes \nabla_v
    \otimes \nabla_x\,,\\[2mm]
    A_2 := \nabla_v^* \otimes \nabla_v^* \otimes
    \nabla_x^*\,\Gamma^{-1}\,\Lambda^{-\frac12} :
    \Lambda^{-\frac12}\,\Gamma^{-1}\,\nabla_v \otimes \nabla_v
    \otimes \nabla_v\,.
  \end{dcases}
\end{equation}
{}From~\cite{HN05,HN04} or by standard pseudo-differential
calculus arguments (see Lemma~\ref{lem:pseudobnd}), $A_0, A_1,
A_2$ are bounded operators in $\mathrm L^2(\cM)$. Let 
$\R_n[V]$ be the space of real polynomials of $v$ with degree
less or equal than $n \in \N$, then 
\[
  \begin{dcases}
    \nabla_v^{\otimes 3}\left( \R_2[V] \right) = \{0\}\,,
    \quad \nabla_v^{\otimes 2}\left( \R_1[V] \right) =
    \{0\}\,,\quad \nabla_v \left( \R_0[V] \right) =
    \{0\}\,,\\[2mm]
    A_2 \left( \R_2[V] \right) = \{0\}\,,\quad A_1\left(
      \R_1[V] \right) = \{0\}\,,\quad A_0 \left( \R_0[V]
    \right) = \{0\}\,.
  \end{dcases}
\]
This means for instance that $A_2$ acts at the level of the local energy
without seeing the local density and local momentum, in a
descending cascade.
Note that the simplest ``order $1$'' corrector $A_0$ is inspired
by the corrector
$\nabla^*_x\,\Lambda^{-1/2} : \Lambda^{-1/2}\,\nabla_v$
introduced in~\cite{HN04} for Fokker-Planck type equations.
We define macroscopic deviations from averages quantities that
are slighlty different from~\eqref{rs}--\eqref{es}:
\begin{align*}
  & \tilde{e} := e - \langle e \rangle\,, \\
  & \tilde{m} := m - \langle \nabla_x m \rangle\,x - \langle m
    \rangle\,, \\
  & \tilde{r} := r - \frac12\,\langle \nabla_x^{\otimes 2} r
    \rangle : \big( x\otimes x -\langle x\otimes x \rangle \big)
    - \langle \nabla_x r \rangle \cdot x -\langle r \rangle\,.
\end{align*}
The core commutator estimates are:
\begin{proposition}
  \label{prop:error}
  For all $i \in \{0,1,2\}$, we have
  \begin{equation}
    \label{eq:derivcom}
    \frac{\mathrm d}{\mathrm dt} \langle A_i\,h, h \rangle =
    -\,\langle \Lambda_i\,h,h \rangle + \langle B_i h,h \rangle
  \end{equation}
  where the $B_i$'s are bounded operators that satisfy
  \begin{align*}
    \label{eq:error}
    & \langle B_0 h, h \rangle \lesssim \| h \| \left( \| h^\bot \|
      + \left\| \tilde{e} \right\| + \| \tilde{m} \| \right), \\
    & \langle B_1 h, h \rangle \lesssim \| h \| \left( \| h^\bot \|
      + \left\| \tilde{e} \right\| \right), \\
    & \langle B_2 h, h \rangle \lesssim \| h \|\,\| h^\bot \|\,,
  \end{align*}
  and where the $\Lambda_i$'s are nonnegative self-adjoint
  operators that satisfy, for some $C_0$, $C_1$, $C_2 >0$ and
  $\bar{\lambda}_0$, $\bar{\lambda}_1$, $\bar{\lambda}_2>0$
  \begin{align*}
    & -\,\langle \Lambda_0\,h, h \rangle \leq -\bar{\lambda}_0
      \left\| \tilde{r} \right\|^2 + C_0 \left( \left\| \tilde{m}
      \right\| + \left\| \tilde{e} \right\| + \| h^\bot \|
      \right) \| h \|\,, \\
    & -\,\langle \Lambda_1\,h, h \rangle \leq -\bar{\lambda}_1
      \left\| \tilde{m} \right\|^2 + C_1\left( \left\| \tilde{e}
      \right\| + \| h^\bot \| \right) \| h \|\,, \\
    & -\,\langle \Lambda_2\,h, h \rangle \leq -\bar{\lambda}_2
      \left\| \tilde{e} \right\|^2 + C_2\,\| h^\bot \|\,\| h
      \|\,.
  \end{align*}
\end{proposition}
\begin{proof}[Proof of Proposition~\ref{prop:error}]
  Since $\cC$ is self-adjoint and $\cT$ is skew-adjoint, 
  \begin{align*}
    \frac{\mathrm d}{\mathrm dt} \langle A_i\,h, h \rangle =
    \Big\langle \left(\left[A_i,\cT\right] + A_i\,\cC +
    \cC\,A_i\right)h, h \Big\rangle\quad\mbox{for}\quad
    i \in \{0,1,2\}.
  \end{align*}
  We can therefore write
  $[A_i,\cT] + A_i\,\cC +\cC\,A_i =: B_i - \Lambda_i$ where, by
  using~\eqref{eq:commutators-T}, we have the following explicit
  formulas
  \begin{align*}
    & B_0 := A_0\,\cC +\cC\,A_0 \\
    & \qquad +\,\nabla_x^* \otimes \nabla_x^* \otimes \nabla_x^*
      \big[\Lambda^{-\frac32},\cT \big] :
      \Lambda^{-\frac32}\,\nabla_v \otimes \nabla_x \otimes
      \nabla_x \\
    & \qquad +\,\nabla_x^* \otimes \nabla_x^* \otimes
      \nabla_x^*\,\Lambda^{-\frac32} :
      \big[\Lambda^{-\frac32},\cT \big] \nabla_v \otimes \nabla_x
      \otimes \nabla_x \\
    & \qquad +\,\nabla_x^* \otimes \nabla_x^* \otimes
      \nabla_x^*\,\Lambda^{-\frac32} :
      \Lambda^{-\frac32}\,\nabla_v \otimes \nabla_x \otimes
      (H_\phi \nabla_v) \\
    & \qquad +\,\nabla_x^* \otimes \nabla_x^* \otimes
      \nabla_x^*\Lambda^{-\frac32} : \Lambda^{-\frac32}\,\nabla_v
      \otimes (H_\phi \nabla_v) \otimes \nabla_x \\
    & \qquad +\,(H_\phi \nabla_v)^* \otimes \nabla_x^* \otimes
      \nabla_x^*\,\Lambda^{-\frac32} :
      \Lambda^{-\frac32}\,\nabla_v \otimes \nabla_x \otimes
      \nabla_x \\
    & \qquad +\,\nabla_x^* \otimes \nabla_x^* \otimes (H_\phi
      \nabla_v)^*\,\Lambda^{-\frac32} :
      \Lambda^{-\frac32}\,\nabla_v \otimes \nabla_x \otimes
      \nabla_x \\
    & \qquad +\,\nabla_x^* \otimes (H_\phi \nabla_v)^* \otimes
      \nabla_x ^* \Lambda^{-\frac32} :
      \Lambda^{-\frac32}\,\nabla_v \otimes \nabla_x \otimes
      \nabla_x\,, \\
    & B_1 := A_1\,\cC +\cC\,A_1 \\
    & \qquad +\,\nabla_x^* \otimes \nabla_v^* \otimes \nabla_x^*
      \big[\Gamma^{-\frac12}\,\Lambda^{-1}, \cT\big] :
      \Lambda^{-1}\,\Gamma^{-\frac12}\,\nabla_v \otimes \nabla_v \otimes
      \nabla_x \\
    & \qquad +\,\nabla_x^* \otimes \nabla_v^* \otimes
      \nabla_x^*\,\Gamma^{-\frac12} \Lambda^{-1} :
      \big[\Lambda^{-1}\,\Gamma^{-\frac12}, \cT\big] \nabla_v
      \otimes \nabla_v \otimes \nabla_x \\
    & \qquad +\,\nabla_x^* \otimes \nabla_v^* \otimes
      \nabla_x^*\,\Gamma^{-\frac12}\,\Lambda^{-1} :
      \Lambda^{-1}\,\Gamma^{-\frac12}\,\nabla_v \otimes \nabla_v
      \otimes (H_\phi \nabla_v) \\
    & \qquad +\,\nabla_x^* \otimes \nabla_v^* \otimes (H_\phi
      \nabla_v)^*\,\Gamma^{-\frac12}\,\Lambda^{-1} :
      \Lambda^{-1}\,\Gamma^{-\frac12}\,\nabla_v \otimes \nabla_v
      \otimes \nabla_x \\
    & \qquad -\,\nabla_x^* \otimes \nabla_x^* \otimes
      \nabla_x^*\,\Gamma^{-\frac12}\,\Lambda^{-1} :
      \Lambda^{-1}\,\Gamma^{-\frac12}\,\nabla_v \otimes \nabla_v
      \otimes \nabla_x \\
    & \qquad +\,(H_\phi \nabla_v)^* \otimes \nabla_v^* \otimes
      \nabla_x^*\,\Gamma^{-\frac12}\,\Lambda^{-1} :
      \Lambda^{-1}\,\Gamma^{-\frac12}\,\nabla_v \otimes \nabla_v
      \otimes \nabla_x\,, \\
    & B_2 := A_2\,\cC +\cC A_2 \\
    & \qquad + \nabla_v^* \otimes \nabla_v^* \otimes \nabla_x^*
      \big[\Lambda^{-\frac12}\,\Gamma^{-1}, \cT\big] :
      \Gamma^{-1}\,\Lambda^{-\frac12}\,\nabla_v \otimes \nabla_v
      \otimes \nabla_v \\
    & \qquad + \nabla_v^* \otimes \nabla_v^* \otimes
      \nabla_x^*\,\Lambda^{-\frac12}\,\Gamma^{-1} :
      \big[\Gamma^{-1}\,\Lambda^{-\frac12}, \cT\big] \nabla_v
      \otimes \nabla_v \otimes \nabla_v \\
    & \qquad + \nabla_v^* \otimes \nabla_v^* \otimes (H_\phi
      \nabla_v)^*\,\Gamma^{-1}\,\Lambda^{-\frac12} :
      \Lambda^{-\frac12}\,\Gamma^{-1}\,\nabla_v \otimes \nabla_v
      \otimes \nabla_v \\
    & \qquad - \nabla_v^* \otimes \nabla_x^* \otimes
      \nabla_x^*\,\Gamma^{-1}\,\Lambda^{-\frac12} :
      \Lambda^{-\frac12}\,\Gamma^{-1} \nabla_v \otimes \nabla_v
      \otimes \nabla_v\,,\\
    & \qquad - \nabla_x^* \otimes \nabla_v^* \otimes
      \nabla_x^*\,\Gamma^{-1}\,\Lambda^{-\frac12} :
      \Lambda^{-\frac12}\,\Gamma^{-1} : \nabla_v \otimes \nabla_v
      \otimes \nabla_v\,,
  \end{align*}
  and 
  \begin{align*}
    \Lambda_0 :=
    & \, \nabla_x^* \otimes \nabla_x^* \otimes
      \nabla_x^*\,\Lambda^{-\frac32} :
      \Lambda^{-\frac32}\,\nabla_x\otimes\nabla_x \otimes
      \nabla_x\,, \\
    \Lambda_1 :=
    & \, \nabla_x^* \otimes \nabla_v^* \otimes
      \nabla_x^*\,\Gamma^{-\frac12}\,\Lambda^{-1} :
      \Lambda^{-1}\,\Gamma^{-\frac12}\,\nabla_x\otimes \nabla_v
      \otimes \nabla_x \\
    & +\, \nabla_x^* \otimes \nabla_v^* \otimes
      \nabla_x^*\,\Gamma^{-\frac12}\,\Lambda^{-1} :
      \Lambda^{-1}\,\Gamma^{-\frac12}\,\nabla_v \otimes
      \nabla_x\otimes \nabla_x\,, \\
    \Lambda_2 :=
    & \, \nabla_v^* \otimes \nabla_v^* \otimes
      \nabla_x^*\,\Gamma^{-1}\,\Lambda^{-\frac12} :
      \Lambda^{-\frac12}\,\Gamma^{-1}\,\nabla_x\otimes
      \nabla_v \otimes \nabla_v \\
    & +\, \nabla_v^* \otimes \nabla_v^* \otimes
      \nabla_x^*\,\Gamma^{-1}\,\Lambda^{-\frac12} :
      \Lambda^{-\frac12}\,\Gamma^{-1}\,\nabla_v \otimes
      \nabla_x\otimes \nabla_v \\
    & +\, \nabla_v^* \otimes \nabla_v^* \otimes
      \nabla_x^*\,\Gamma^{-1}\,\Lambda^{-\frac12} :
      \Lambda^{-\frac12}\,\Gamma^{-1}\,\nabla_v \otimes \nabla_v
      \otimes \nabla_x\,.
  \end{align*}
  The $\Lambda_i$'s operators are nonnegative, self-adjoint (see
  Appendix~\ref{verifcom}) and boun\-ded (see
  Lemma~\ref{lem:pseudobnd} again).

  Note that for all $i \in \{0,1,2\}$:
  \begin{align*}
    & |\langle A_i\,\cC\,h, h \rangle| = |\langle A_i\,\cC\,h^\bot, h
      \rangle| \lesssim \| h^\bot \|\,\| h \|\,, \\
    & |\langle \cC\,A_i\,h, h \rangle| = |\langle A_i\,h, \cC h
      \rangle| = |\langle A_i\,h, \cC h^\bot \rangle | \lesssim \|
      h \|\,\| h^\bot \|\,,
  \end{align*}
  because $\cC $ is self-adjoint and bounded
  by~\eqref{hyp:boundedCollisionOperator}.

Now we deal with $B_2$. Let us denote by $B_{2,\ell}$ the
  $\ell$-th line. Since
  $\Lambda^{1/2}\,\Gamma\,[\Lambda^{-1/2}\,\Gamma^{-1}, \cT]$ is
  bounded by standard pseudo-differential calculus
  (see~\cite{HN05}),
  \[b_{2,2} := \nabla_v^* \otimes \nabla_v^* \otimes
  \nabla_x^*\,[\Lambda^{-\frac12}\,\Gamma^{-1}, \cT]\] is bounded
  and
  \begin{multline*}
    \left| \left\langle B_{2,2} \,h, h \right\rangle \right| =
    \left| \left\langle b_{2,2} :
        \Gamma^{-1}\,\Lambda^{-\frac12}\,\nabla_v \otimes
        \nabla_v
        \otimes \nabla_v h, h \right\rangle \right| \\
    = \left| \left\langle b_{2,2} :
        \Gamma^{-1}\,\Lambda^{-\frac12}\,\nabla_v \otimes
        \nabla_v \otimes \nabla_v h^\bot, h \right\rangle \right|
    \lesssim \|h\|\,\| h^\bot \|
  \end{multline*}
  since the three derivatives in velocity on the right hand side
  cancel all macroscopic quantities. The other lines are dealt
  with similarly, using the boundedness of $H_\phi$, and we
  deduce
  $|\langle B_2 h, h \rangle| \lesssim \| h \|\,\| h^\bot \|$.

  Now we deal with $B_1$. We consider the second line $B_{1,2}$.
  Then
  $\Lambda\,\Gamma^{1/2}\,[\Lambda^{-1}\,\Gamma^{-1/2}, \cT]$ is
  bounded by standard pseudo-differential calculus and so
 \[b_{1,2} := \nabla_x^* \otimes \nabla_v^* \otimes
  \nabla_x^*\,\big[\Lambda^{-1}\,\Gamma^{-\frac12}, \cT\big]\] is
  bounded. Thus
  \begin{multline*}
    \left| \left\langle B_{1,2} h, h \right\rangle \right| =
    \left| \left\langle b_{1,2} :
        \Gamma^{-\frac12}\,\Lambda^{-1}\,\nabla_v \otimes
        \nabla_v
        \otimes \nabla_xh, h \right\rangle \right| \\    \quad \lesssim \left| \left\langle b_{1,2} :
        \Gamma^{-\frac12}\,\Lambda^{-1}\,\nabla_v \otimes
        \nabla_v \otimes \nabla_x (e\,\fE + h^\bot), h
      \right\rangle \right|\,.
  \end{multline*}
  The macroscopic quantities $r$ and $m$ are canceled since two
  derivatives in velocity are involved. We then use
  $\nabla_v \otimes \nabla_v \otimes \nabla_x e\,\fE = \nabla_x e
  \otimes \mathrm{Id}_{d \times d}$ and Appendix~\ref{ssec:poin}
  shows (using $\Omega \ge 1$ for the first inequality)
  \[
    \big\|\,\Omega^{-1}\,\nabla_x e\big\| \leq
    \big\|\,\Omega^{-\frac12}\,\nabla_x e\big\| \lesssim
    \|\tilde{e}\| \quad \mbox{ so }
    \quad 
    \left| \left\langle B_{1,2} h, h \right\rangle \right| \lesssim
    \| h \| \left(\| \tilde{e} \| + \|h^\bot\| \right)\,.
  \]
  The other lines are similar and yield the same estimate.

  We then deal with $B_0$, and focus on the second line $B_{0,2}$
  again. The operator
  \[b_{0,2} := \nabla_v^* \otimes \nabla_v^* \otimes \nabla_x^*
  \left[ \Lambda^{-\frac32}, \cT \right]\] is bounded arguing as
  before, and
  \begin{equation}
    \label{eq:B02}
    \left| \left\langle B_{0,2} h, h \right\rangle \right|
    \!=\! \left| \left\langle b_{0,2} : \Lambda^{-\frac32}\,\nabla_v
        \otimes \nabla_x \otimes \nabla_xh, h \right\rangle \right|
    \lesssim \| h \| \! \left( \| \tilde{e} \| + \| \tilde{m} \| + \|h^\bot\|
    \right).
  \end{equation}
  Indeed a direct computation gives
  \[
    \Gamma^{-1}\,\Lambda^{-\frac12}\,\nabla_v \otimes \nabla_x
    \otimes \nabla_xh = \Omega^{-\frac32}\,\nabla_x \otimes
    \nabla_x m + \tilde{\Omega}^{-\frac32}\,\nabla_x \otimes
    \nabla_x \otimes (v\,e)
  \]
  with $\tilde{\Omega} := (\nabla^*_x\,\nabla_x + 2)$. The factor
  $2$ comes from the fact that for all $\alpha \in \R$,
  \[
    \Lambda^{\alpha} (v\,e) = (\nabla^*_x\,\nabla_x +
    \nabla^*_v\,\nabla_v + 1)^\alpha\,(v\,e) =
    (\nabla^*_x\,\nabla_x +1 +1)^\alpha\,(v\,e)
  \]
  since $v$ is an eigenfunction of $\nabla^*_v\,\nabla_v$ with
  eigenvalue $1$. To complete the proof of~\eqref{eq:B02}, it is
  sufficient to notice that
  $\|\Omega^{-\frac32}\,\nabla_{\!x}^2 e \| \leq
  \|\Omega^{-1}\,\nabla_{\!x}^2 e \| \lesssim \|\tilde{e}\|$ and also that,
  by Appendix~\ref{ssec:poin},
  $\|\Omega^{-\frac32}\,\nabla_{\!x}^2 m \| \leq
  \|\Omega^{-1}\,\nabla_{\!x}^2 m \| \lesssim \|\tilde{m}\|$. The
  other lines in $B_0$ are treated similarly again.

  Now we deal with the main nonnegative terms
  $\langle \Lambda_i\,h,h \rangle$ in~\eqref{eq:derivcom}.
  We first compute the contribution of $\Lambda_0$, $\Lambda_1$
  and $\Lambda_2$ acting respectively on $r$, $m$ and~$e$. Due to
  the number of derivatives in velocity in the right hand side of
  the expressions giving the $\Lambda_i$'s, we have
  $\Lambda_i ( \R_{i-1}[V] ) = \{0\}$ for $i \in \{1,2\}$ and
  $\Lambda_i ( \R_i[V] ) \subset \R_i[V] $ for $i \in \{0,1,2\}$
  and
  \begin{equation}
    \label{eq:lambdarem0}
     \left\langle \Lambda_0\,r, r\right\rangle =
      \iint_{\R^d\times\R^d} \left| \Omega^{-\frac32}\,\nabla_x^3
        r \right|^2 \cM\,\dd x \dd v = \int_{\R^d} \left|
        \Omega^{-\frac32}\,\nabla_x^3r \right|^2 \rho\,\dd
      x\,,
    \end{equation}
  \begin{equation}
      \label{eq:lambdarem1}
  \begin{aligned}
      \big\langle \Lambda_1\,(m(x)\cdot v)\,, (m(x)\cdot v)
      \big\rangle
      &= \iint_{\R^d\times\R^d} \Big(
      \left(\Omega^{-1} \partial^2_{x_i x_k} m_j\right)
      \left(\Omega^{-1}\,\partial^2 _{x_i x_k} m_j\right)\\
      &\quad +
      \left(\Omega^{-1}\partial^2 _{x_i x_k} m_j\right)
      \left(\Omega^{-1}\,\partial^2 _{x_j x_k}
        m_i\right) \Big)\,\cM\,\dd x \dd v\\
       &= 2 \int_{\R^d} \left| \Omega^{-1}\nabla_x
        \nabla_x^{\text{\tiny sym}}\,m \right|^2\,\rho\dd
      x\,,
  \end{aligned}
    \end{equation}
    \begin{align}
      \label{eq:lambdarem2}
      & \big\langle \Lambda_2 \left( e(x)\,\fE(v) \right), \left(
        e(x)\,\fE(v) \right) \big\rangle \\ \nonumber
      & \quad = \iint_{\R^d\times\R^d} \Big(
      \left(\Omega^{-\frac12}\,\partial_{x_k}e\right)
      \left(\partial_{v_i\,v_j}^2\,\fE\right)
      \left(\Omega^{-\frac12}\,\partial_{x_k}e\right)
      \left(\partial_{v_i\,v_j}^2\,\fE\right) \\ \nonumber
      & \hspace*{1cm} +
      \left(\Omega^{-\frac12}\,\partial_{x_j}e\right)
      \left(\partial_{v_i v_k}^2\,\fE\right)
      \left(\Omega^{-\frac12}\,\partial_{x_k}e\right)
      \left(\partial_{v_i\,v_j}^2\,\fE\right) \\ \nonumber
      & \hspace*{1cm} +
      \left(\Omega^{-\frac12}\,\partial_{x_i}e\right)
      \left(\partial_{v_j v_k}^2\,\fE\right)
      \left(\Omega^{-\frac12}\,\partial_{x_k}e\right)
      \left(\partial_{v_i\,v_j}^2\,\fE\right) \Big)\,\cM\,\dd x
      \dd v \\ \nonumber
      & \quad = \int_{\R^d} \left( 2
        \left(\Omega^{-\frac12}\,\partial_{x_k}e\right)^2 +
        \frac2d\left(\Omega^{-\frac12}\,\partial_{x_i}e\right)^2
        + \frac2d
        \left(\Omega^{-\frac12}\,\partial_{x_j}e\right)^2 \right)
      \rho\dd x \\ \nonumber
      & \quad = \left( \frac4d+2 \right) \int_{\R^d}
      \left|\Omega^{-\frac12}\,\nabla_x e\right|^2\,\rho\,\dd
      x\,.
    \end{align}
  Next we use the cascade of Poincar\'e inequalities of
  Lemma~\ref{lem:poincaretensoriel}. For the density $r$, this
  implies that there isa constant $\overline{\lambda}_0 >0$ such
  that
  \begin{equation}
    \label{eq:minorr}
    \begin{aligned}
      & \langle \Lambda_0\,r, r \rangle = \int_{\R^d} \left|
        \Omega^{-\frac32}\,\nabla_x^3r \right|^2 \rho\,\dd x \\
      & \geq 2\,\overline\lambda_0 \left\|\, r - \langle r
        \rangle - \langle \nabla_x r \rangle \cdot x -
        \tfrac12\,\langle \nabla_{\!x}^2 r \rangle : \left(x
          \otimes x - \langle x \otimes x\rangle\right)\right\|^2
      = 2\,\overline{\lambda}_0 \left\| \tilde{r} \right\|^2\,.
    \end{aligned}
  \end{equation}
  Regarding the momentum $m$, one first observes that
  \[
    |\Omega^{-1}\nabla_x \nabla_x^{\text{\tiny sym}}\,m|^2 \ge
    \frac19\,|\Omega^{-1}\,\nabla_{\!x}^2 m|^2
  \]
  thanks to the Schwarz lemma written as
  \[
    \forall\,i,j,k \in \{1, \dots, d \}\,,\; \partial_{ij}^2 m_k
    = \partial_i \left( \nabla^{\text{\tiny sym}} m \right)_{jk}
    + \partial_j \left(\nabla_x^{\text{\tiny sym}}\,m
    \right)_{ik} - \partial_k \left(\nabla_x^{\text{\tiny
          sym}}\,m \right)_{ij}\,.
  \]
  Applied to~\eqref{eq:lambdarem0}--\eqref{eq:lambdarem2}, the cascade of Poincar\'e
  inequalities at order $2$ stated in
  Lemma~\ref{lem:poincaretensoriel} implies that there exists a
  constant $\overline\lambda_1$ such that
  \begin{equation}
    \label{eq:minorm}
    \begin{aligned}
      \langle \Lambda_1\,(m \cdot v) , m \cdot v \rangle
      & \ge \frac29\, \|\,\Omega^{-1}\,\nabla_{\!x}^2 m \|^2 \\
      & \ge 2\,\overline\lambda_1 \left\| m - \langle m \rangle -
        \langle\nabla_x m \rangle\,x \right\|^2 =
      2\,\overline\lambda_1 \| \tilde{m} \|^2.
    \end{aligned}
  \end{equation}
  Regarding the energy $e$, we use the standard Poincar\'e
  inequality in $\mathrm L^2(\rho)$ (the order $1$ inequality of
  Lemma~\ref{lem:poincaretensoriel}) to get
  $\overline \lambda_2 >0$ such that
  \begin{equation}
    \label{eq:minore}
    \langle  \Lambda_2\,(e\,\mathfrak{E}(v)),
    e\,\mathfrak{E}(v) \rangle = \left( \tfrac4d + 2 \right)
    \|\,\Omega^{-\frac12}\,\nabla_x e \|^2 \ge 2\,\overline
    \lambda_2\,\|e -\langle e \rangle\|^2 = 2\,\overline
    \lambda_2\,\|\tilde{e} \|^2\,.
  \end{equation}

  With the above estimates in hand, we can investigate all terms
  appearing in $\langle \Lambda_i\,h , h \rangle$. According to
  the number of velocity and space gradients in $\Lambda_2$, we
  get that
  \[
    \langle \Lambda_2\,h, h \rangle = \Big\langle
    \Lambda_2\,(e\,\fE(v)), e\,\fE(v) \Big\rangle + \left\langle
      \Lambda_2\,(e\,\fE(v)), h^\bot \right\rangle + \left\langle
      \Lambda_2\,h^\bot, h \right\rangle\,,
  \]
  from which one obtains with~\eqref{eq:minore} that
  \[
    \label{eq:sg2}
    - \langle \Lambda_2\,h, h \rangle \le
    -\,\overline\lambda_2\,\| \tilde{e} \|^2 + \cO \left( \|
      h^\bot \|\,\| h \| \right).
  \]
  Similarly for $m$, using in addition that $\Lambda_1$ is
  self-adjoint and
  $\Lambda_1\,(e\,\fE(v)) = \Lambda_1\,(\tilde{e}\,\fE(v))$, one
  has
  \begin{multline*}
    \langle \Lambda_1\,h, h \rangle = \left\langle \Lambda_1\,(m
      \cdot v), m \cdot v \right\rangle + \left\langle
      \Lambda_1\,(m \cdot v), (e- \langle e \rangle)\,\fE(v) +
      h^\bot \right\rangle \\
    + \left\langle \Lambda_1 \left( \left( e - \langle e \rangle
        \right) \fE(v) + h^\bot \right), h \right\rangle\,,
  \end{multline*}
  which implies using~\eqref{eq:minorm} that
  \[
    \label{eq:sg1}
    -\,\langle \Lambda_1\,h, h \rangle \le
    -\,\overline\lambda_1\,\| \tilde{m} \|^2 + \cO \big( \|
    \tilde{e} \|\,\| h \|\big) +\cO \big( \| h^\bot \|\,\| h \|
    \big)\,.
  \]
  Finally, regarding the local density $r$, we get similarly
  \begin{align*}
    \langle \Lambda_0\,h, h \rangle =
    & \,\langle \Lambda_0\,r, r \rangle \\
    & + \left\langle \Lambda_0\,r,  \tilde{m} \cdot v + \tilde{e}
      \fE(v) + h^\bot \right\rangle + \left\langle \Lambda_0\,(
      \tilde{m} \cdot v + \tilde{e} \fE(v) + h^\bot) , h
      \right\rangle
  \end{align*}
  and it follows from~\eqref{eq:minorr} that
  \[
    \label{eq:sg0}
    -\,\langle \Lambda_0\,h, h \rangle \le
    -\,\overline\lambda_0\,\| \tilde{r} \|^2 +\cO \big( \|
    \tilde{m} \|\,\| h \| \big) +\cO \big( \|\tilde{e} \|\,\| h
    \| \big) +\cO \big( \| h^\bot \|\,\| h \| \big)\,.
  \]
  The proof of the proposition is complete.
\end{proof}

We collect the previous estimates into a first partial
Lyapunov inequality:
\begin{lemma}
  \label{lem:step1com}
  Define the following norm 
  \begin{align*}
    \| h \|_{\cH_1}^2 := \| h \|^2 + \varepsilon_0\,\langle
    A_0\,h\,, h \rangle + \varepsilon_1\,\langle A_1\,h , h \rangle
    + \varepsilon_2\,\langle A_2\,h,h \rangle
  \end{align*}
  for $\varepsilon_0$, $\varepsilon_1$, $\varepsilon_2>0$, then for
  $\mathrm c_{\sC} \gg \varepsilon_2\,\gg \varepsilon_1\,\gg
  \varepsilon_0$, we have
  \begin{align*}
    \label{eq:hypo-kin1}
    \frac12\,\frac{\mathrm d}{\mathrm dt} \| h \|_{\cH_1}^2
    & \le -\,\frac{\mathrm c_{\sC}}2\,\| h^\bot \|^2 -
      \frac{\varepsilon_2}2\,\bar\lambda_2 \left\| \widetilde e
      \right\|^2 - \frac{\varepsilon_1}2\,\overline\lambda_1\,\|
      \widetilde m \|^2 - \varepsilon_0\,\overline\lambda_0
      \left\| \widetilde r \right\|^2 + \eta_1\,\| h \|^2\,,
  \end{align*}
  for some $0 < \eta_1 \ll \varepsilon_0$.
\end{lemma}
\begin{proof}[Proof of Lemma~\ref{lem:step1com}]
  Propositions~\ref{prop:error} combined with
  Lemma~\ref{lem1:hperpLyap} imply
  \begin{align*}
    \frac12 \frac{\mathrm d}{\mathrm dt} \| h \|_{\cH_1}^2 \le
    & \,-\,\mathrm c_{\sC}\,\| h^\bot \|^2 -
      \varepsilon_2\,\bar\lambda_2 \left\| \widetilde e \right\|^2
      - \varepsilon_1\,\overline\lambda_1\,\| \widetilde m \|^2 -
      \varepsilon_0\,\overline\lambda_0 \left\| \widetilde r
      \right\|^2 \\
    & \,+ C\,\varepsilon_2\,\| h \|\,\| h^\bot \| +
      C\,\varepsilon_1\,\| h \| \left( \| h^\bot \| + \left\|
      \widetilde e \right\| \right) \\
    &\,+ C\,\varepsilon_0\,\| h \| \left( \| h^\bot \| + \|
      \widetilde m \| + \left\| \widetilde e \right\| \right)
  \end{align*}
  for some constant $C>0$. The statement then follows from
  repeated uses of Young's inequality for products. 
\end{proof}

In fact the time derivatives of the local density, momentum and
energy can also be controlled as follows:
\begin{equation}
  \label{eq:control-variation}
  \left\{
  \begin{aligned}
    & \frac{\mathrm d}{\mathrm dt} \left\langle \partial_t
      \widetilde r\,, \Omega^{-1}\,\widetilde r \right\rangle \ge
    \big\|\,\Omega^{-\frac12}\,\partial_t \widetilde r \,\big\|^2
    - \cO \big(\|h\|\,\left\| \widetilde r \right\| \big)\,, \\
    & \frac{\mathrm d}{\mathrm dt} \big\langle \partial_t
    \widetilde m\,, \Omega^{-1}\,\widetilde m \,\big\rangle \ge
    \|\,\Omega^{-\frac12}\,\partial_t \widetilde m \,\|^2 - \cO
    \big(\|h\|\,\| \widetilde m\|\big)\,, \\
    & \frac{\mathrm d}{\mathrm dt} \big\langle \partial_t
    \widetilde e\,, \Omega^{-1}\,\widetilde e \,\big\rangle \ge
    \|\,\Omega^{-\frac12}\,\partial_t \widetilde e \,\|^2 - \cO
    \big(\|h\|\,\| \widetilde e\|\big)\,.
  \end{aligned}
  \right.
\end{equation}
This leads to second improved partial Lyapunov inequality:
\begin{lemma}
  \label{lem:lyap-2}
  Given $1 \gg \varepsilon_2\,\gg \varepsilon'_2\,\gg
  \varepsilon_1\,\gg \varepsilon_1 '\,\gg \varepsilon_0 \gg
  \varepsilon_0' \gg \eta_1$, the norm
  \[
    \label{eq:hypo-H2}
    \| h \|_{\cH_2}^2 := \| h \|_{\cH_1}^2 -
    \varepsilon_2'\,\big\langle \partial_t
    \widetilde e\,, \Omega^{-1}\,\widetilde e \,\big\rangle -
    \varepsilon_1' \big\langle \partial_t \widetilde m\,,
    \Omega^{-1}\,\widetilde m \,\big\rangle - \varepsilon_0 '
    \left\langle \partial_t \widetilde r\,,
      \Omega^{-1}\,\widetilde r \right\rangle\,.
  \]
  satisfies
  \begin{align*}
    \frac12 \frac{\mathrm d}{\mathrm dt} \| h \|_{\cH_2}^2 \le&\,
    -\,\frac{\mathrm c_{\sC}}4\,\| h^\bot \|^2 -
    \frac{\varepsilon_2}4\,\bar\lambda_2 \left\| \widetilde e
    \right\|^2 - \frac{\varepsilon_1}4\,\overline\lambda_1\,\|
    \widetilde m \|^2 - \frac{\varepsilon_0}2\,\overline\lambda_0
    \left\|
      \widetilde r \right\|^2 \\
    &\,-
    \varepsilon'_2\,\overline\lambda_2\,\big\|\,\Omega^{-\frac12}\,\partial_t
    \widetilde e \big\|^2 -
    \varepsilon'_1\,\overline\lambda_1\,\big\|\,\Omega^{-\frac12}\,\partial_t
    \widetilde m \big\|^2 -
    \varepsilon_0'\,\overline\lambda_0\,\big\|\,\Omega^{-\frac12}\,\partial_t
    \widetilde r \big\|^2 \\
    &\, + \eta_2\,\| h \|^2
  \end{align*}
  for some $0 < \eta_2 \ll \varepsilon'_0$.
\end{lemma}
\begin{proof}[Proof of Lemma~\ref{lem:lyap-2}]
  This follows from~\eqref{eq:control-variation} and 
  \[
    \left| \big\langle \partial_t \widetilde e\,,
      \Omega^{-1}\,\widetilde e \,\big\rangle \right| + \left|
      \big\langle \partial_t \widetilde m\,,
      \Omega^{-1}\,\widetilde m \,\big\rangle \right| + \left|
      \left\langle \partial_t \widetilde r\,,
        \Omega^{-1}\,\widetilde r \right\rangle \right|
    \lesssim \| h \|^2
  \]
  and the fact that second order time derivatives of the
  macroscopic quantities can be controlled by $\Omega$ (as in
  Section~\ref{sec:micmac}).
\end{proof}

\subsection{Cascade of finite-dimensional correctors}

In view of Lemma~\ref{lem:step1com}, what remains to be
controlled are the finite dimensional terms
\[
  \langle r \rangle\,,\quad \langle \nabla_x m \rangle\,x +
  \langle m \rangle \quad \mbox{and} \quad \frac12\,\langle
  \nabla_x^{\otimes 2} r \rangle : \big( x\otimes x -\langle
  x\otimes x \rangle \big) - \langle\nabla_x r \rangle \cdot x
  -\langle r \rangle\,.
\]

\subsubsection{Control of moments of the local momentum}

We compute
\begin{align*}
  \frac{\mathrm d}{\mathrm dt} \langle v_i\,v_j\,h \rangle = 2
  \left\langle \left( \nabla_x^{\text{\tiny sym}}\,m \right)_{ij}
  \right\rangle + \left\langle \big( \cT(v_i\,v_j) + \cC
  (v_i\,v_j)\big)\, h^\bot \right\rangle\,,
\end{align*}
which yields 
\begin{align*}
  & \frac{\mathrm d}{\mathrm dt} \left(\langle v_i\,v_j\,h
    \rangle\,\frac{\mathrm d}{\mathrm dt} \langle v_i\,v_j\,h
    \rangle\right) \\
  & \ge 2 \left\langle \left( \nabla_x^{\text{\tiny sym}}\,m
    \right)_{i,j} \right\rangle^2 -
    \left|\big\langle\big(\cT(v_i\,v_j) + \cC (v_i\,v_j)\big)\,
    h^\bot \big\rangle^2 \right| + \langle v_i\,v_j\,h
    \rangle\,\frac{{\mathrm d}^2}{\mathrm d t^2} \langle v_i\,v_j\,h
    \rangle
\end{align*}
and
\[
  \left\{
    \begin{aligned}
    & \langle v_i\,v_j\,h \rangle\,\frac{\mathrm d}{\mathrm dt}
    \langle v_i\,v_j\,h \rangle \lesssim \| h \|\,\| h^\bot \|\,, \\
    & \left|\big\langle\big(\cT(v_i\,v_j) + \cC (v_i\,v_j)\big)\,
      h^\bot \big\rangle^2 \right| + \langle v_i\,v_j\,h
    \rangle\,\frac{{\mathrm d}^2}{\mathrm d t^2} \langle v_i\,v_j\,h
    \rangle \lesssim \| h \|\,\| h^\bot \|\,.
  \end{aligned}
  \right.
\]

Define for all $i \in \{ 1,\dots,d \}$
\begin{align*}
  \psi_i (v) := 1 + \sqrt{\tfrac d2} \left( 1 + \tfrac4d \right)
  \fE(v) - \sqrt{\tfrac d2}\,|v_i|^2\,\fE(v)
\end{align*}
which is orthogonal to $1$, $v$, $|v|^2$. We then compute
\begin{align*}
\frac{\mathrm d}{\mathrm dt} \langle \psi_i\,h \rangle = 4
  \left\langle \tfrac1d\,\nabla_x \cdot m - \partial_{x_i} m_i
  \right\rangle + \left\langle \left[ \cT \left(\psi_i \right) +
  \cC \left(\psi_i \right) \right] h^\bot \right\rangle\,,
\end{align*}
which yields
\begin{align*}
  & \frac{\mathrm d}{\mathrm dt} \left( \langle \psi_i\,h
    \rangle\,\frac{\mathrm d}{\mathrm dt} \langle \psi_i\,h
    \rangle \right) \\
  & \ge 8 \left\langle \tfrac1d\,\nabla_x \cdot m -
    \partial_{x_i} m_i \right\rangle^2 - \left\langle [ \cT(
    \psi_i(v)) + \cC (\psi_i(v))] h^\bot \right\rangle^2 +
    \langle \psi_i\,h \rangle\,\frac{{\mathrm d}^2}{\mathrm d t^2}
    \langle \psi_i\,h \rangle
\end{align*}
and
\[
  \left\{
  \begin{aligned}
    & \langle \psi_i\,h \rangle\,\frac{\mathrm d}{\mathrm dt}
    \langle \psi_i\,h \rangle \lesssim \| h \|^2\,, \\
    & \left|\left\langle [ \cT( \psi_i(v)) + \cC (\psi_i(v))]
        h^\bot \right\rangle \right|^2 + \langle \psi_i\,h
    \rangle\,\frac{{\mathrm d}^2}{\mathrm d t^2} \langle \psi_i\,h
    \rangle \lesssim \| h \|\,\| h^\bot \|\,.
  \end{aligned}
  \right.
\]
We finally introduce the third norm
\begin{multline*}
  \| h \|_{\cH_3}^2 := \| h \|_{\cH_2}^2 - \varepsilon_3 \sum_{i
    \neq j} \langle v_i\,v_j\,h \rangle\,\frac{\mathrm d}{\mathrm
    dt} \langle v_i\,v_j\,h \rangle - \varepsilon_3\,\langle
  \psi_i\,h \rangle\,\frac{\mathrm d}{\mathrm dt} \langle
  \psi_i\,h \rangle \\
  - \varepsilon_3' \sum_{i \neq j} \big\langle \frac{\mathrm
    d}{\mathrm dt} \left( \nabla^{\text{\tiny sym}}_x m
  \right)_{i,j} \big\rangle \big\langle \left(
    \nabla_x^{\text{\tiny sym}}\,m \right)_{i,j} \big\rangle \\
  - \varepsilon_3'\,\big\langle \frac{\mathrm d}{\mathrm dt}
  \left(\tfrac1d\,\nabla_x \cdot m - \partial_{x_i} m_i \right)
  \big\rangle \,\big\langle \tfrac1d\,\nabla_x \cdot m -
  \partial_{x_i} m_i \big\rangle
\end{multline*}
for
$1 \gg \varepsilon_3\,\gg \varepsilon_3'\,\gg \varepsilon_2\,\gg
\varepsilon_2'\,\gg \varepsilon_1\,\gg \varepsilon_1'\,\gg
\varepsilon_0 \gg \varepsilon_0'$. Therefore, defining as
in~\eqref{ms} and~\eqref{es} the \emph{space} inhomogeneous terms
\[
  \label{eq:tilde-mc}
  m_s := m - \langle \nabla_x^{\text{\tiny skew}}m \rangle\,x -
  \tfrac1d\,\langle \nabla_x \cdot m \rangle\,x - \langle m
  \rangle \quad \mbox{and} \quad e_s := \tilde{e} = e-\langle e
  \rangle\,,
\]
we obtain
\[
  \label{eq:hypo-kin3}
  \begin{aligned}
    \frac12\,\frac{\mathrm d}{\mathrm dt} \| h \|_{\cH_3}^2 \le
    & - \frac{\mathrm c_{\sC}}{8}\,\| h^\bot \|^2 -
    \frac{\varepsilon_2}4\,\bar\lambda_2 \left\| e_s \right\|^2 \\
    & - \frac{\varepsilon_1}{8}\,\overline\lambda_1 \left\|
      \widetilde m \right\|^2 - \frac{\varepsilon_1}{8}\,
    \overline\lambda_1 \left\| m_s \right\|^2 -
    \frac{\varepsilon_0}2\,\overline\lambda_0 \left\| \widetilde
      r \right\|^2 \\
    & -
    \frac{\varepsilon'_2}2\,\overline\lambda_2\, \big\|\,
    \Omega^{-\frac12}\, \partial_t \widetilde e \,\big\|^2 \\
    & -
    \frac{\varepsilon'_1} 2\, \overline\lambda_1\, \big\|\,
    \Omega^{-\frac12}\, \partial_t \widetilde m \,\big\|^2 -
    \frac{\varepsilon'_1}2\,\overline\lambda_1\, \big\|\,
    \Omega^{-\frac12}\, \partial_t m_s \,\big\|^2 \\
    & - \varepsilon_0'\,\overline\lambda_0\,
    \big\|\,\Omega^{-\frac12}\, \partial_t \widetilde r
    \,\big\|^2 + \eta_3\,\| h \|^2 
  \end{aligned}
\]
for another $0<\eta_3 \ll \varepsilon_0'$.

\subsubsection{Control of moments of the local density and
  energy}
\label{ssec:moments}

We now control the difference between the finite dimensional
quantities
\[
  r - \frac12\,\langle \nabla_x^{\otimes 2} r \rangle : \big( x\otimes x
  -\langle x\otimes x \rangle \big) - \langle \nabla_x r \rangle
  \cdot x -\langle r \rangle
\]
and
\begin{align*}
  w_s & := r - \frac{1}{2\,d}\,\langle \Delta_x r \rangle \left(
        |x|^2- \langle |x|^2 \right) - \langle \nabla_x r \rangle
        \cdot x \\
      & \quad - \sqrt{\frac{2}{d}}\,\langle e \rangle \left[ \phi -
        \langle \phi \rangle - \frac{1}{2\,d}\,\langle \Delta_x \phi
        \rangle \left( |x|^2 - \langle |x|^2 \rangle \right) \right]\end{align*}
defined in~\eqref{ws} in Section~\ref{sec:micmac}, which is made
of the two terms
(using that $\langle r \rangle =0$)
\[
  \label{eq:remainder}
  \left\{
    \begin{aligned}
      & I_1 := - \frac12 \sum_{1 \le i, j \le d} \left( \langle
        \partial^2_{x_i x_j} r \rangle - \frac1d\,\langle \Delta_x r
        \rangle \delta_{ij} \right) \left( x_i\,x_j - \langle x_i
        x_j \rangle \right)\,,\\
      & I_2:= \sqrt{\frac{2}{d}}\, \langle e
      \rangle \left[ \phi - \langle \phi \rangle - \frac{1}{2\,d}\,
        \langle \Delta_x \phi \rangle \left( |x|^2 - \langle
          |x|^2 \rangle \right) \right]\,.
    \end{aligned}
    \right.
  \]
The first term is controlled by using 
\begin{multline*}
  \frac{{\mathrm d}}{{\mathrm d}t} \left( \big\langle (\nabla^{\text{\tiny
        sym}}_x m)_{ij} \big\rangle - \frac1d\,\langle \nabla_x \cdot
    m \rangle\,
    \delta_{ij} \right) \\
  = - \left( \langle \partial^2_{x_i x_j} r \rangle - \frac1d\,
    \langle \Delta_x r \rangle \delta_{ij} \right) + \
  \mbox{controlled terms}
\end{multline*}
since the left hand side is already under control, and the second
term is controlled by observing that $\langle e \rangle = \langle
r\,\phi \rangle$ due to the energy conservation, and
\[
  \langle r\,\phi \rangle = \langle \tilde r\,\phi \rangle +
  \frac12\,\langle \nabla^2_x r \rangle : \langle \left( x \otimes
    x - \langle x \otimes x \rangle \right) \phi \rangle
\]
provided that $\int_{\R^d} x\,\phi\,e^{-\phi} \dd x=0$, and $\nabla^2_x r =
\nabla^2_x \tilde r$, so finally
\[
  |\langle e \rangle|  =| \langle r\,\phi \rangle | \lesssim \|
  \tilde r \|\,. 
\]
This allows to define the final and fourth norm
\begin{align*}
  \| h \|_{\cH_4}^2 := &\,\| h \|_{\cH_3}^2\\ 
  &\,+ \varepsilon_w\,
  \frac{{\mathrm d}}{{\mathrm d}t} \left(\!\big\langle (\nabla^{\text{\tiny
        sym}}_x m)_{ij} \big\rangle - \frac1d\,\langle \nabla_x \cdot
    m \rangle
    \,\delta_{ij} \right)
  \left(\!\langle \partial^2_{x_i x_j} r \rangle - \frac1d\,
    \langle \Delta_x r \rangle \,\delta_{ij} \right) \\
  &\,- \varepsilon_w '
  \left(\!\langle \partial^2_{x_i x_j} r \rangle - \frac1d\,
    \langle \Delta_x r \rangle \,\delta_{ij} \right)
  \frac{{\mathrm d}}{{\mathrm d}t}
  \left(\!\langle \partial^2_{x_i x_j} r \rangle - \frac1d\,
    \langle \Delta_x r \rangle \,\delta_{ij} \right)\end{align*}
for
$1 \gg \varepsilon_3\,\gg \varepsilon_3'\,\gg \varepsilon_2\,\gg
\varepsilon_2'\,\gg \varepsilon_1\,\gg \varepsilon_1'\, \gg
\varepsilon_0 \gg \varepsilon_0' \gg \varepsilon_w \gg
\varepsilon_w'$, with 
\[
  \label{eq:hypo-kin5}
  \begin{aligned}
    \frac12\,\frac{\mathrm d}{\mathrm dt} \| h \|_{\cH_4}^2 \le &
    -\,\frac{\mathrm c_{\sC}}{16}\,\| h^\bot \|^2 -
    \frac{\varepsilon_2}{8}\,\bar\lambda_2 \left\| e_s \right\|^2
    -
    \frac{\varepsilon_w}2\,\left\| w_s \right\|^2 \\
    & -\, \frac{\varepsilon_1}{8}\,\overline\lambda_1 \left\|
      \widetilde m \right\|^2 -
    \frac{\varepsilon_1}{16}\,\overline\lambda_1 \left\| m_s
    \right\|^2 - \frac{\varepsilon_0}2\,\overline\lambda_0
    \left\|
      \widetilde r \right\|^2 \\
    & -\,\frac{\varepsilon'_2}2\, \overline\lambda_2
    \,\big\|\,\Omega^{-\frac12}\,\partial_t \widetilde e
    \,\big\|^2 - \frac{\varepsilon'_1}2\, \overline\lambda_1
    \,\big\|\,\Omega^{-\frac12}\,\partial_t \widetilde m \,\big\|^2\\
    & -\,\frac{\varepsilon'_1}2\, \overline\lambda_1
    \,\big\|\,\Omega^{-\frac12}\,\partial_t m_s \,\big\|^2-
    \varepsilon_0'\,\overline\lambda_0
    \,\big\|\,\Omega^{-\frac12}\,\partial_t \widetilde r \,\big\|^2\\
    & -\,\varepsilon_w' \,\big\|\,\Omega^{-\frac12}\,\partial_t
    w_s \,\big\|^2 + \eta_4\,\| h \|^2
  \end{aligned}
\]
for $0<\eta_4 \ll \varepsilon_w'$. Denote by $\cC_4[h]$ the
semi-norm of the controlled quantities
\begin{multline*}
  \cC_4[h] := \Big(\| h^\bot \|^2 + \left\| \widetilde e
  \right\|^2 +\left\| \widetilde m \right\|^2 +\left\| m_s
  \right\|^2 + \left\| \widetilde r \right\|^2 + \left\| w_s
  \right\|^2 + \big\|\,\Omega^{-\frac12}\,\partial_t e_s
  \,\big\|^2 \\
  + \big\|\,\Omega^{-\frac12}\,\partial_t \widetilde m \,\big\|^2
  + \big\|\,\Omega^{-\frac12}\,\partial_t m_s \,\big\|^2 +
  \big\|\,\Omega^{-\frac12}\,\partial_t \widetilde r \,\big\|^2 +
  \big\|\,\Omega^{-\frac12}\,\partial_t w_s \,\big\|^2
  \Big)^{1/2}
\end{multline*}
and adjust the constants to
get, for some $0 < \eta \ll \varepsilon \ll 1$,
\begin{equation}
  \label{eq:hypo-kin6}
  \frac12\,\frac{\mathrm d}{\mathrm dt} \| h \|_{\cH_4}^2 \le -\,
  \varepsilon\,\cC_4[h]^2 + \eta\,\| h \|^2\,.
\end{equation}

\subsection{Control of the remaining finite-dimensional
  quantities related to the special macroscopic modes}

Estimate~\eqref{eq:hypo-kin6} controls the same microscopic and
macroscopic parts of the solution as in Lemma~\ref{lem5:F1Lyap}
in the micro-macro method. The remaining finite-dimensional
quantities related to the special macroscopic modes can then be
treated exactly as in Sections~\ref{ssec:hypomacro}
and~\ref{ssec:hypomacro2}. This completes the proof of
Proposition~\ref{prop:hypoc}.

\appendix
\section{Some technical computations}
\label{app:tech}
\setcounter{equation}{0}

\subsection{Momentum conservation \emph{versus} infinitesimal
  rotations}
\label{ssec:momvs}

Here we prove~\eqref{eq:consm} for a solution
$f \in \cC\big(\R^+ ; \mathrm L^2(\cM^{-1})\big)$
to~\eqref{eq:main} with initial datum
$f_0\in \mathrm L^2(\cM^{-1})$. With
$x \mapsto A\,x := \Pp_\phi m_0(x)$,
\begin{align*}
  & m_0(x) := e^{\phi(x)} \int_{\R^d} v\,f_0(x,v) \dd v\,,\,\quad
  & m_f(t,x) := e^{\phi(x)} \int_{\R^d} v\,f(t,x,v) \dd v\,, \\
  & r_f(t,x) := e^{\phi(x)} \int_{\R^d} f(t,x,v) \dd v\,, \quad
  & e_f(t,x) := e^{\phi(x)} \int_{\R^d} \fE(v)\,f(t,x,v) \dd v\,,
\end{align*}
let us define $h^\bot$ such that
$f = r_f\,\cM + m_f\cdot v\,\cM + e_f\,\fE\,\cM + h^\bot \cM$.
\begin{lemma}
  \label{lem:proj-ort-append}
  With the above notations, we have have
  $\Pp(m_f-m_0) \in \cR_\phi^\bot$.
\end{lemma}
\begin{proof}[Proof of Lemma~\ref{lem:proj-ort-append}]
  At $t=0$, we have \hbox{$\Pp\,m_f(0)-\Pp\,m_0 = 0$}. Let
  $B\in\cR_\phi$. To prove that $\Pp(m_f-m_0)$ is orthogonal to
  $x \mapsto B\,x$, it is sufficient to prove that
  \[
    \forall\,t \geq 0\,,\quad \left\langle m_f(t)-m_0, B\,x
    \right\rangle = 0\,.
  \]
  By direct computation, we have
  \[
    \partial_t m_f = -\,\nabla_x r_f +
    \sqrt{\tfrac2d}\,\nabla_x^*\,e_f + \nabla_x^* \cdot E[h^\bot]
  \]
  where
  $E[h^\bot] = \int_{\R^d} \left( v \otimes v - \mathrm{Id}_{d
      \times d} \right) h^\bot\,\mu \dd v$, and use it to compute
  \begin{align*}
    & \frac{\mathrm d}{\mathrm dt} \langle m_f\!-\!m_0, B\,x
      \rangle \\
    & = \langle \partial_t m_f, B\,x \rangle \\
    & = \Big\langle -\nabla_x r_f + \sqrt{\tfrac2d}\,
      \nabla_x^*\,e_f + \nabla_x^* \cdot E\,(h^\bot), B\,x
      \Big\rangle \\
    & = -\,\big\langle r_f,\nabla_x^* \cdot B\,x \big\rangle +
      \sqrt{\tfrac2d}\,\Big\langle e_f, \nabla_x \cdot B\,x
      \Big\rangle + \big\langle E[h^\bot] : \nabla \otimes B\,x
      \big\rangle
  \end{align*}
  where the last line follows from an integration by parts. The
  first term in the right hand side vanishes because
  $\nabla_x^* \cdot B\,x = -\,\nabla_x \cdot B\,x + \nabla \phi
  \cdot B\,x= 0$ since~$B$ is skew-symmetric and
  $(x\mapsto B\,x) \in \cR_\phi$. The second term vanishes as
  well because $\nabla_x \cdot B\,x = 0$. Since
  $E[h^\bot] : \nabla \otimes B\,x = -\,E[h^\bot] : B = 0$
  because $E[h^\bot]$ is symmetric and $B$ is skew-symmetric, the
  third term also vanishes. This proves that
  $\frac{\mathrm d}{\mathrm d t} \langle m(t)-m_0, B\,x \rangle =
  0$ and completes the proof.
\end{proof}

\subsection{Special macroscopic modes: the invertibility and
  rank}

We state and prove two results used in
Section~\ref{ssec:controlbcm} and implicitely in
Section~\ref{ssec:hypomacro2}. The first result deals with the
invertibility of the matrices $M_\phi$ and $\widehat M_\phi$
defined respectively in~\eqref{eq:Mphim} and~\eqref{eq:hatMphim}.
\begin{lemma}
  \label{lem:Mphi-caseI}
  If $d_\phi=d$, the matrix $M_\phi$ is invertible. If
  $1\leq d_\phi \leq d-1$, the matrix $\widehat M_\phi$ is
  invertible.
\end{lemma}
\begin{proof}[Proof of Lemma~\ref{lem:Mphi-caseI}]
  Assume that $d_\phi=d$ in~\eqref{eq:defephi}. Let $u \in \R^d$
  be such that \hbox{$M_\phi u = 0$}. Then
  $M_\phi u \cdot u = \langle |\Phi \cdot u|^2 \rangle = 0$,
  which implies that $\Phi(x) \cdot u = 0$ for any $x \in \R^d$,
  hence $u=0$. This means that $\mathrm{Ker}\, M_\phi =
  \{0\}$. The proof in the case $d_\phi \leq d-1$ follows exactly
  the same scheme.
\end{proof}

The second result deals with the linear independence of the two
functions $\widetilde \Psi_1$ and $\widetilde \Psi_2$ defined
in~\eqref{tildepsipsi}, and similarly for $\widehat \Psi_1$ and
$\widehat \Psi_2$ defined in~\eqref{hatpsipsi}.
\begin{lemma}
  \label{lem:tildePsi}
  If $d_\phi=d$, we have
  $\mathrm{Rank} (\widetilde \Psi_1 , \widetilde \Psi_2 ) =
  2$.\\ If $1 \leq d_\phi \leq d-1$, we have
  $\mathrm{Rank} (\widehat \Psi_1 , \widehat \Psi_2 ) = 2$.
\end{lemma}
\begin{proof}[Proof of Lemma~\ref{lem:tildePsi}]
  Let us assume that $d_\phi=d$ and argue by
  contradiction. Assume that
  $\widetilde \Psi_1 = \lambda\,\widetilde \Psi_2$ for some
  $\lambda \in \R^*$, that is, there are constants $\alpha$,
  $\beta \in \R^d$ and $\gamma \in \R$ such that
  \begin{equation}
    \label{eq:phi-tildePsi}
    \phi + \tfrac12\,\nabla_x \phi \cdot x + \alpha \cdot
    \nabla_x \phi = \tfrac\lambda4\,|x|^2 + \beta \cdot x +
    \gamma\,.
  \end{equation}
  We first look for quadratic solutions
  to~\eqref{eq:phi-tildePsi} of the form
  $\phi_0 = x \cdot M_0\,x + b_0 \cdot x + c_0$ with
  $M_0 \in \mathfrak{M}_{d \times d}(\R)$, $b_0 \in \R^d$ and
  $c_0 \in \R$. Plugging $\phi_0$ into~\eqref{eq:phi-tildePsi},
  one obtains $M_0 = \frac\lambda8\, \mathrm{Id}_{d \times d}$,
  $b_0 = \frac23\,\big(\beta - \frac\lambda4\, \alpha\big)$ and
  $c_0 = \gamma - \frac23\, \alpha \cdot \big(\beta -
  \frac\lambda4\, \alpha\big)$. Now let $\phi$ be a solution
  to~\eqref{eq:phi-tildePsi}. Define
  $\psi_0(x) = \phi(x) - \phi_0(x)$ and then
  $\psi (y) = \psi_0(y-2\,\alpha)$, which hence verifies
  \[
    \label{eq:phi-tildePsi-bis}
    \psi(y) - \tfrac12\,\nabla_y \psi(y) \cdot y = 0\,.
  \]
  Let $\zeta(y) = |y|^2\,\psi(y)$ so that
  $\nabla_y \zeta(y) \cdot y = 2\,|y|^2\,\big(\psi(y) -
  \tfrac12\,\nabla_y \psi(y) \cdot y \big) = 0$ for any
  $y \in \R^d$. In polar coordinates $(r,\theta)$, this implies
  that $\zeta(y) = \zeta(\theta)$ and hence
  \[
    \forall\,r > 0\,, \quad\psi(r,\theta) =
    \frac{\zeta(\theta)}{r^2}\,.
  \]
  But $\psi$ is by assumption continuous at the origin, therefore
  $\lim_{r \to 0} \psi(r,\theta)$ is finite, which in turn
  implies that $\psi(r,\theta) = 0$. Finally one gets
  $\phi = \phi_0$. Thanks to the
  normalizations~\eqref{hyp:intnorm} and~\eqref{eq:phiid}, one
  gets $\phi(x) = \tfrac12\,|x|^2 + \frac d2\,\log(2\,\pi)$ and,
  by definition~\eqref{eq:defephi}, $E_\phi=\{ 0 \big\}$, which
  contradicts the hypothesis $d_\phi = d$. This completes the
  proof when $d_\phi=d$. When $d_\phi \leq d-1$, we argue
  similarly on $\widehat \Psi_1$ and $ \widehat \Psi_2$.
\end{proof}

\subsection{Some computations for the commutator method}
\label{verifcom}

Here we prove technical claims used in
Section~\ref{sec:cascade}. The first result is concerned with
boundedness of the operators defined in
Section~\ref{ssec:correctors} under
Assumptions~\eqref{hyp:boundedCollisionOperator}
and~\eqref{hyp:PotentialRestrictions}.
\begin{lemma}
  \label{lem:pseudobnd}
  Assume that~\eqref{eq:kersC}--\eqref{eq:hyp-sg-C}--\eqref{eq:lbound}--\eqref{hyp:intnorm}--\eqref{hyp:regularity}--\eqref{eq:poincarenormal}--\eqref{eq:momentspace}--\eqref{eq:phiid}--\eqref{hyp:semigroup}--~\eqref{hyp:boundedCollisionOperator}--\eqref{hyp:PotentialRestrictions} hold. Then the operators $\Lambda_i$, $A_i$ and $B_i$,
  $i \in \{ 1, \dots, 3\}$, are bounded.
\end{lemma}
\begin{proof}[Proof of Lemma~\ref{lem:pseudobnd}]
  As a typical example, we focus on $\Lambda_1$ for which it is
  sufficient to show that
  $\Lambda^{-1}\,\Gamma^{-1/2}\,\partial_{x_i} \partial_{v_j}
  \partial_{v_k}$ is bounded in $\mathrm L^2(\cM)$. Adopting the
  point of view of~\cite[Proposition A.7]{HN05}, we first
  conjugate with $\cM^{1/2}$ and only have to check that
  $\tilde{\Gamma}^{-1/2}\,(\partial_{v_k}+v_k/2)$ and
  $\tilde{\Lambda}^{-1}\,(\partial_{x_i}+x_i/2)\,(\partial_{v_j}+v_j/2)
  $ are bounded in $\mathrm L^2(\R^d,\mathrm dx\mathrm dv)$,
  where
  \begin{align*}
    & \tilde{\Lambda} = \sum_i \left(-\,\partial_{x_i}+
      \frac{x_i}2 \right) \left(\partial_{x_i} + \frac{x_i}2
      \right) + \left(-\,\partial_{v_i} + \frac{v_i}2
      \right)\left( \partial_{v_i} + \frac{v_i}2 \right), \\
    & \tilde{\Gamma} = \sum_i \left( -\,\partial_{v_i} +
      \frac{v_i}2 \right) \left( \partial_{v_i} + \frac{v_i}2
      \right).
  \end{align*}
  For $\tilde{\Gamma}^{-1/2} (\partial_{v_k}+v_k/2)$ this is due
  to the fact that $\tilde{\Gamma}^{-1/2}$ is of order $-1$ and
  $\partial_{v_k}+v_k/2$ of order $1$ in the pseudo-differential
  calculus associated to the metric
  $(\mathrm d v^2 + \mathrm d \eta^2)/(1+|v|^2+|\eta|^2)$, $\eta$
  being the dual variable of $v$. The composition is then of
  order $0$ and the Calder\'on-Vaillancourt Theorem
  (see~\cite{MR298480}) implies the boundedness. For
  $\tilde{\Lambda}^{-1} (\partial_{x_i} + x_i/2) (\partial_{v_j}
  +v_j/2) $, the result is also true because $\Lambda^{-1}$ is of
  order $-2$ and
  $(\partial_{x_i} + x_i/2) (\partial_{v_j} +v_j/2) $ is of order
  $2$ in the pseudo-differential calculus associated to the
  metric
  $(\mathrm d x^2+ \mathrm d v^2 + \mathrm d \xi^2 + \mathrm d
  \eta^2)/(1+|\eta|^2+|v|^2 + |\nabla \phi|^2 + |\xi|^2)$, $\xi$
  being the dual variable of $v$. This implies the desired
  boundedness. Such calculus with two levels (involving~$\Lambda$
  in all variables and $\Gamma$ only in velocity variables) is
  also at the core of the boundedness of terms like
  $\Lambda\,\Gamma^{1/2}\,\big[\Lambda^{-1}\,\Gamma^{-1/2},
  \cT\big]$ where we use that $\Lambda$ and~$\Gamma$ commute and
  that the commutation with $\cT$ decreases the order of derivatives by~$1$, so that this
  operator is of order $0$ and therefore bounded by the
  Calder\'on-Vaillancourt Theorem. Note in addition that $H_\phi$
  (which appears, \emph{e.g.}, in the $B_i$'s) is of order $0$
  which greatly simplifies the proofs. For all other terms
  $\Lambda_i$, $A_i$ and~$B_i$, similar computations give the
  result.
\end{proof}

The second result deals with the symmetry and nonnegativity of
$\Lambda_1$.
\begin{lemma}
  \label{lem:lambda1-sym}
  The operator $\Lambda_1$ is symmetric and nonnegative.
\end{lemma}
\begin{proof}[Proof of Lemma~\ref{lem:lambda1-sym}]
  First we check that
  \[
    \nabla_x^*\,\nabla_v^*\,\nabla_x^*\,\Gamma^{-\frac12}\,\Lambda^{-1}
    :
    \Lambda^{-1}\,\Gamma^{-\frac12}\,\nabla_v\,\nabla_x\,\nabla_x
  \]
  is symmetric, since for the other part of $\Lambda_1$ this is
  obvious:
  \begin{align*}
    & \big\langle
      \nabla_x^*\,\nabla_v^*\,\nabla_x^*\,\Gamma^{-\frac12}\,\Lambda^{-1}
      : \Lambda^{-1}\,\Gamma^{-\frac12}\,\nabla_v\,\nabla_x
      \nabla_x f , g \big\rangle \\
    & = \sum_{i,j,k} \big\langle
      \partial_{x_i}^*\,\partial_{v_j}^*\, \partial_{x_k}^*\,
      \Gamma^{-\frac12}\,\Lambda^{-1} :
      \Lambda^{-1}\,\Gamma^{-\frac12}\,\partial_{v_k}\,
      \partial_{x_j}\, \partial_{x_i}\,f,\,g \big\rangle\\
    & = \sum_{i,j,k} \big\langle \Lambda^{-1}\,
      \Gamma^{-\frac12}\, \partial_{v_k}\, \partial_{x_j} \,
      \partial_{x_i} \, f,\, \Lambda^{-1} \,
      \Gamma^{-\frac12}\,\partial_{x_k}\, \partial_{v_j} \,
      \partial_{x_i}\,g \big\rangle \\
    & = \sum_{i,j,k} \big\langle \Lambda^{-1} \,
      \Gamma^{-\frac12}\, \partial_{x_j} \, \partial_{v_k} \,
      \partial_{x_i} \, f,\, \Lambda^{-1} \, \Gamma^{-\frac12} \,
      \partial_{v_j}\,\partial_{x_k}\,\partial_{x_i}\,g \big\rangle\\
    & = \sum_{i,j,k} \langle f,\,\partial_{x_i}^*\,
      \partial_{v_k}^*\, \partial_{x_j}^*\,\Gamma^{-\frac12} \,
      \Lambda^{-2} \,\Gamma^{-\frac12} \,
      \partial_{v_j}\,\partial_{x_k}\, \partial_{x_i}\,g \rangle\\
    & = \big\langle f,\,\nabla_x^*\,\nabla_v^*\,\nabla_x^*\,
      \Gamma^{-\frac12}\, \Lambda^{-1} : \Lambda^{-1}\,
      \Gamma^{-\frac12} \, \nabla_v\,\nabla_x\,\nabla_x\,g
      \big\rangle\,.
  \end{align*}
  Next check that $\Lambda_1$ is indeed a nonnegative operator:
  \begin{align*}
    \langle \Lambda_1\,f,f \big\rangle
    & = \big\langle \nabla_x^* \otimes \nabla_v^* \otimes
      \nabla_x^*\,\Gamma^{-\frac12}\,\Lambda^{-1} :
      \Lambda^{-1}\,\Gamma^{-\frac12}\,\nabla_x^* \otimes
      \nabla_v^* \otimes \nabla_x^*\,f,\,f \big\rangle \\
    &\,\quad + \big\langle \nabla_x^* \otimes \nabla_v^* \otimes
      \nabla_x^*\,\Gamma^{-\frac12}\,\Lambda^{-1} :
      \Lambda^{-1}\,\Gamma^{-\frac12}\,\nabla_v^* \otimes
      \nabla_x^* \otimes \nabla_x^*\,f,f \big\rangle \\
    & = \sum_{i,j,k} \big\langle \Lambda^{-1} \,
      \Gamma^{-\frac12}\, \partial_{x_k} \, \partial_{v_j} \,
      \partial_{x_i}\, f,\, \Lambda^{-1} \, \Gamma^{-\frac12}\,
      \partial_{x_k}\,\partial_{v_j}\,\partial_{x_i}\,f \big\rangle \\
    &\,\quad + \big\langle \Lambda^{-1} \, \Gamma^{-\frac12}\,
      \partial_{v_k}\, \partial_{x_j}\, \partial_{x_i}\, f,\,
      \Lambda^{-1} \, \Gamma^{-\frac12}\,\partial_{x_k}\,
      \partial_{v_j} \,\partial_{x_i}\,f \big\rangle \\
    &= \sum_{i,j,k} \frac12 \left\|
      \Lambda^{-1}\,\Gamma^{-\frac12}\,(\partial_{x_k}\,\partial_{v_j}
      + \partial_{v_k}\,\partial_{x_j})\,\partial_{x_i}\,f
      \right\|^2 \, .
  \end{align*}
  This completes the proof.
\end{proof}

\subsection{A cascade of Poincar\'e-Lions inequalities}
\label{ssec:poin}

Under Assumptions~\eqref{hyp:boundedCollisionOperator}
and~\eqref{hyp:PotentialRestrictions}, we prove several
inequalities used in Section~\ref{ssec:correctors}. Let $\varphi$
a smooth function in $\mathrm L^2(\rho)$ with compact support and
\begin{align*}
  \label{eq:deftayl}
  & P_0 (\varphi) := \langle \varphi \rangle\,,\\[2mm]
  & P_1(\varphi) := \langle \varphi \rangle + \langle \nabla_x
  \varphi \rangle\cdot x\,, \\
  & P_2(\varphi) := \langle \varphi \rangle + \langle \nabla_x
    \varphi \rangle\cdot x + \frac12 \left\langle (\nabla_x
    \otimes \nabla_x)\,\varphi \right\rangle : \big( x \otimes x
    - \langle x \otimes x \rangle \big)\,.
\end{align*}
\begin{lemma}
  \label{lem:poincaretensoriel}
  Let $n \in \{1,2,3\}$. Then there exists a constant
  $c_{\text{\tiny P},n} >0$ such that for all smooth $\varphi$
  with compact support we have
  \[
    c_{\text{\tiny P},n} \left\| \varphi - P_{n-1}(\varphi)
    \right\|^2 \leq \big\|\,\Omega^{-n/2}\,\nabla_x^{\otimes n}
    \varphi \,\big\|^2.
  \]
\end{lemma}
\begin{proof}[Proof of Lemma~\ref{lem:poincaretensoriel}]
  For $n= 1$ this is exactly the Poincar\'e-Lions Theorem as
  stated in~\cite[Proposition 5]{CDHMM21} and recalled
  in~\eqref{eq:poincareL2}. Let us prove the result for
  $n=2$. Let $\varphi$ be smooth and with compact support. We
  have
  \[
    \langle \varphi-P_1(\varphi) \rangle = \langle \varphi
    \rangle - \langle \varphi \rangle - \langle \nabla_x \varphi
    \rangle \cdot \langle x \rangle = 0
  \]
  because $\langle x \rangle = 0$. We therefore apply the
  Poincar\'e-Lions inequality (\emph{i.e.}, the case $n=1$) to
  $\varphi-P_1(\varphi)$, which gives
  \[
    \left\| \varphi-P_1(\varphi) \right\|^2 \leq c_{\text{\tiny
        P},1}^{-1} \,\big\|\,\Omega^{-\frac12}\,\nabla_x
    \left(\varphi -P_1(\varphi) \right)\!\big\|^2 = c_{\text{\tiny
        P},1}^{-1} \,\big\|\,\Omega^{-\frac12}\,\big(\nabla_x \varphi
    - \langle \nabla_x \varphi \rangle \big)\big\|^2.
  \]
  We then apply the ``$-1$-order'' Poincar\'e-Lions inequality
  in~\cite[Lemma 10]{CDHMM21} recalled in~\eqref{eq:poincareH-1}
  to $\nabla_x \varphi$ to get, for some $C_{\text{\tiny LPL}}>0$
  depending only on $\phi$,
  \[
    \,\big\|\,\Omega^{-\frac12}\,\big(\nabla_x\varphi - \langle
    \nabla_x \varphi \rangle \big) \big\|^2 \leq C_{\text{\tiny
        LPL}} \,\big\|\,\Omega^{-1}\,\nabla_x^2\varphi\,\big\|^2.
  \]
  This proves the case $n=2$ with
  $c_{\text{\tiny P},2} = {c_{\text{\tiny P},1}}/{C_{\text{\tiny
        LPL}}}$.

  In the case $n=3$, we define
  $\psi := \varphi - \frac12\,\langle \nabla^2 \phi \rangle :
  x\otimes x $ and we compute
  \begin{align*}
    \psi - P_1(\psi)
    & = \psi - \langle \psi \rangle - \langle \nabla_x \psi
      \rangle \cdot x \\
    & = \varphi - \tfrac12\,\langle \nabla^2 \varphi \rangle:
      x\otimes x - \langle \varphi \rangle \\
    &\qquad + \tfrac12\,\langle \nabla^2 \varphi \rangle :
      \langle x\otimes x \rangle - \langle \nabla_x \varphi
      \rangle \cdot x + \langle \nabla_{\!x}^2 \varphi \rangle :
      \langle x \rangle \otimes x \\
    & = \varphi - P_2(\varphi)
  \end{align*}
  since $\langle x \rangle = 0$. We apply the inequality for
  $n=2$ and obtain
  \begin{multline}
    \label{eq:P23}
    \left\| \varphi - P_2(\varphi) \right\|^2 = \left\| \psi -
      P_1(\psi) \right\|^2 \\
    \leq c_{\text{\tiny P},2}^{-1}
    \,\big\|\,\Omega^{-1}\nabla_{\!x}^2 \psi \,\big\|^2 =
    c_{\text{\tiny P},2}^{-1} \,\big\|\,\Omega^{-1}(
    \nabla_{\!x}^2 \varphi - \langle \nabla_{\!x}^2 \varphi
    \rangle) \,\big\|^2\,.
  \end{multline}
  Arguing as for the proof of the ``$-1$ order'' Poincar\'e-Lions
  inequality in~\cite[Lemma 10]{CDHMM21}, we prove the ``$-2$
  order'' Poincar\'e-Lions inequality with constant
  $C_{\text{\tiny LPL}}>0$:
  \[
    \,\big\|\,\Omega^{-1}(f - \langle f \rangle) \,\big\|^2 \leq
    C_{\text{\tiny LPL}} \,\big\|\,\Omega^{-\frac32}\,\nabla_x f
    \,\big\|^2
  \]
  for any $f \in C^\infty_c$. Applying this estimate
  in~\eqref{eq:P23} to $\nabla_{\!x}^2 \varphi$ gives
  \[
    \left\| \varphi - P_2(\varphi) \right\|^2 \leq c_{\text{\tiny
        P},2}^{-1}\,C_{\text{\tiny LPL}}
    \,\big\|\,\Omega^{-\frac32}\,\nabla_x^3 \varphi \,\big\|^2
  \]
  and concludes for $n=3$ with
  $c_{\text{\tiny P},3} = {c_{\text{\tiny P},2}}/{C_{\text{\tiny
        LPL}}}$, and completes the proof.
\end{proof}

\section{Extension to weakly coercive collision operators}
\label{app:extension}
\setcounter{equation}{0}

Our method covers the
case of collision operators $\sC$ that do not possess a spectral
gap (assumption~\eqref{eq:hyp-sg-C} in Section~\ref{sec:intro}) but only
satisfy a weaker coercivity property
(see~\eqref{eq:hyp-weakcoercivity-C} below). In this appendix, we state some results and changes to be done in the proofs.

\subsection{Results on decay rates}

We assume that $\sC$ satisfies, for some $\alpha>0$, the \emph{weak
  coercivity property}
\begin{align}
  \label{eq:hyp-weakcoercivity-C}
  \tag{H1'}
  - \int_{\R^d} \big( \sC f(v) \big)\,f(v)\,\mu(v)^{-1} \dd v \ge
  \mathrm c_{\sC}\,\| f - \Pi f \|_{\mathrm L^2(\wgt^{-\alpha} \mu^{-1})}^2
\end{align}
for some constant $\mathrm c_{\sC}>0$, and for all $f$ in the
domain of $\sC$, where $\Pi$ is the
$\mathrm L^2(\mu^{-1})$-orthogonal projection onto
$\mathrm{Ker}\,\sC$. Here $\wgt$ denotes the weight $\sqrt{1+|v|^2}$. Moreover, we suppose that for any
polynomial function $p(v) : \R^d \to \R$ of degree at most $4$,
the function~$p\,\mu$ is in the domain of $\sC$ and
\begin{align}
  \label{eq:lbound-weakcoercivity}
  \tag{H2'}
  C(p):=\| \sC (p\,\mu)\|_{\mathrm L^2(\wgt^{\alpha}  \mu^{-1})}<\infty\,.
\end{align}
The analog of our main result in Theorem~\ref{theo:main} then becomes:
\begin{theorem}\label{theo:weakcoercivity}
  Assume that the potential $\phi$ and the collision operator
  $\sC$ satisfy 
  assumptions~\eqref{eq:kersC}--\eqref{eq:hyp-weakcoercivity-C}--\eqref{eq:lbound-weakcoercivity}--\eqref{hyp:intnorm}--\eqref{hyp:regularity}--\eqref{eq:poincarenormal}--\eqref{eq:momentspace}--\eqref{eq:phiid}--\eqref{hyp:semigroup}. Then
  \begin{enumerate}
    
  \item[(1)] All special macroscopic modes
    of~\eqref{eq:mainMicroMacro} are given
    by~\eqref{eq:solMacroF}, \emph{i.e.}, are linear combinations
    of the Maxwellian, the energy mode, rotation modes compatible
    with $\phi$, and harmonic directional or pulsating modes if
    allowed by $\phi$.

  \item[(2)] There exists a norm
    $\vertiii{\cdot}_{\mathrm L^2(\cM^{-1})}$ on
    $\mathrm L^2(\cM^{-1})$, which is equivalent to
    $\|\cdot\|_{\mathrm L^2(\cM^{-1})}$ (with quantitative
    comparison constants), and some explicit $\lambda>0$ such
    that, for any solution
    $f\in\cC\big(\R^+; \mathrm L^2(\cM^{-1})\big)$
    to~\eqref{eq:main} with initial datum
    $f_0\in \mathrm L^2(\cM^{-1})$, there exists a unique special
    macroscopic mode $F$ (determined by $f_0$) such that
    \begin{equation}
      \label{eq:degen-decay}
      \forall\,t\ge0\,,\quad
      \frac12\,\frac{\mathrm{d}}{\mathrm{d}t} \vertiii{f(t) - F(t)
      }_{\mathrm  L^2(\cM^{-1})}^2 \le -\,\lambda\,\left\|
        f(t)-F(t)\right\|_{\mathrm L^2( \wgt^{-\alpha} \cM^{-1})}^2\,.
    \end{equation}
  \end{enumerate}
\end{theorem}
The differential inequality~\eqref{eq:degen-decay} alone is not
sufficient to prove a decay estimate when~$f_0$ is merely in
$\mathrm L^2(\cM^{-1})$. In order to get such an estimate, one needs to assume more
decay at infinity for $f_0$ and the differential
inequality~\eqref{eq:degen-decay} has to be replaced by an inequality in spaces with
stronger weights. For instance, assume that, for some $\beta>0$,
\begin{equation}
  \label{hyp:semigroup-weakcoercivity}
  \tag{H9'}
  \mbox{\emph{$t\mapsto e^{t \sL}$ is a strongly continuous uniformly bounded
      semigroup on $\mathrm L^2(\cM^{-1-\beta})$}}\,,
\end{equation}
with $\sL$ as in~\eqref{eq:main}. In the spirit of~\cite{Guo-Strain-2006}, we obtain the following decay rate.
\begin{corollary}\label{cor:weakcoercivity}
  Assume that the potential $\phi$ and the collision operator
  $\sC$ satisfy 
  assumptions~\eqref{eq:kersC}--\eqref{eq:hyp-weakcoercivity-C}--\eqref{eq:lbound-weakcoercivity}--\eqref{hyp:intnorm}--\eqref{hyp:regularity}--\eqref{eq:poincarenormal}--\eqref{eq:momentspace}--\eqref{eq:phiid}--\eqref{hyp:semigroup}--\eqref{hyp:semigroup-weakcoercivity}. Then there are explicit constants $C_0>0$ and $\Lambda>0$ such that, for any solution
  $f\in\cC\big(\R^+; \mathrm L^2(\cM^{-1})\big)$
  to~\eqref{eq:main} with initial datum
  $f_0\in \mathrm L^2(\cM^{-1-\beta})$, there exists a unique
  special macroscopic mode $F$ (determined by $f_0$) such that
  \begin{equation*}
    \forall\,t\ge0\,,\quad\left\| f(t) - F(t) \right\|_{\mathrm
      L^2(\cM^{-1})} \le C_0\,\exp\left(-\,\Lambda \,
      t^{\frac{2}{2+\alpha}}\right) \left\|
      f_0-F(0)\right\|_{\mathrm L^2(  \cM^{-1-\beta})}\,.
  \end{equation*}
\end{corollary}
As in the proof of Theorem~\ref{theo:main}, it is convenient to work with the function $h$ 
defined by~\eqref{eq:defh}. We use various norms: 
$ \| \cdot \|$ defined by~\eqref{nrm} and also
\begin{align*}
  \| h \|_{\star}^2
   := \iint_{\R^d\times\R^d}|h|^2\,\wgt^{-\alpha}
    \, \cM\dd x\dd v \quad \text{ and } \quad 
  \| h \|_{1-\beta}^2 := \iint_{\R^d\times\R^d}|h|^2\,  \cM^{1-\beta}
  \dd x\dd v \, . 
\end{align*}
Observe that $\| \cdot \| = \| \cdot \|_{1}\le \| \cdot \|_{1-\beta} $,
$\| \cdot \|_{\star}=\|\cdot\|_{\mathrm L^2( \wgt^{-\alpha} \cM^{-1})}^2 \le \| \cdot \|$ and
\begin{equation*}
  \| h \|_{\star}^2 \approx \| h^\perp \|_{\star}^2 + \| r \|^2 +
  \| m \|^2 + \| e \|^2 .
\end{equation*}

\subsection{Proof of Theorem~\texorpdfstring{\ref{theo:weakcoercivity}}{B.1}}

Proposition~\ref{prop:min} applies: the proof of Part~(1) is the same as Part~(1) of Theorem~\ref{theo:main}. To prove
Part~(2), we argue as in the proof of Proposition~\ref{prop:hypomm}, using the new assumptions. Thanks to~\eqref{eq:hyp-weakcoercivity-C} and
using that $\langle \sT h , h \rangle=0$, the function $h$ defined by~\eqref{eq:defh} satisfies
\begin{equation*}
  \frac{1}{2} \frac{\mathrm{d}}{\mathrm{d}t} \| h \|^2 = \langle
  \cC h , h \rangle + \langle \sT h , h \rangle \le -\,\mathrm
  c_{\sC}\,\| h^\perp \|_{\star}^2\,.
\end{equation*}
This replaces the estimate of Lemma~\ref{lem1:hperpLyap}.  We can
then use~\eqref{eq:lbound-weakcoercivity} and the above estimate
to prove counterparts of estimates between
Lemma~\ref{lem2:esLyap} and Lemma~\ref{lem:boundchypo} with
$\| h^\perp \|$ and $\| h \|$ respectively replaced by
$\| h^\perp \|_{\star}$ and $\| h \|_{\star}$. Then
$\mathcal{F}_{2}$ defined in~\eqref{eq:F2h} using~\eqref{eq:F1h}, with an appropriate choice of parameters
$0 \ll \varepsilon_6 \ll \varepsilon_5 \ll \varepsilon_4 \ll
\varepsilon_3 \ll \varepsilon_2 \ll \varepsilon_1 \ll 1$, is
equivalent to $\| \cdot \|^2$ and satisfies
\begin{equation*}
  \begin{aligned}
    \label{eq:estimate_weakcoercivity}
    \frac{\mathrm{d}}{\mathrm{d}t} \mathcal{F}_{2}[h] 
    & \le -\,\kappa\,\overline{\mathcal{D}}_{2}[h]
  \end{aligned}
\end{equation*}
for some constant $\kappa>0$, where
$\overline{\mathcal{D}}_{2}[h]$ is defined as
$\mathcal{D}_{2}[h]$ in~\eqref{eq:D2h} with
$\| h^\perp \|^2$ replaced by $\| h^\perp \|^2_{\star}$ in~\eqref{eq:D1h}. This
concludes the proof with
\[
\vertiii{f-F}_{\mathrm L^2(\cM^{-1})}^2:=\mathcal{F}_{2}[h]
\]
since $\overline{\mathcal{D}}_{2}$ is
equivalent to $\| \cdot \|_{\star}^2$.
\hfill\ \qed

\subsection{Proof of Corollary~\ref{cor:weakcoercivity}}

Let $h_0 \in \mathrm L^2(\cM^{1-\beta})$ and consider the solution
$t\mapsto h(t)$ to the equation $\partial_t h = \mathcal{L} h$
with initial datum $h(0)=h_0$. Thanks to
\eqref{hyp:semigroup-weakcoercivity}, there is some $C>0$ such that
\begin{equation*}
\forall \, t \ge 0\,, \quad \| h(t) \|_{1-\beta} \le C\,\| h_0 \|_{1-\beta}\,.
\end{equation*}
We now observe that, for any $R>0$, the following interpolation
inequality holds
\begin{equation*}
  \| g \|^2
  \le \big(1+R^2\big)^{\alpha/2}\,\| g \|_{\star}^2
  + \bar\mu^\beta(R)\,\| g \|_{1-\beta}^2
\end{equation*}
with $\bar\mu(R):=(2\pi)^{-d/2}\,\nrm{e^{-\phi}}{\mathrm L^\infty(\R^d)}\,e^{-R^2/2}$.
Therefore, thanks to~\eqref{eq:degen-decay} and using
fact that $h\mapsto\vertiii{f-F}_{\mathrm L^2(\cM^{-1})}^2=\mathcal{F}_{2}[h]$ is equivalent to $h\mapsto\| h \|^2$, one deduces that
\begin{equation*}
  \begin{aligned}
    \frac{\mathrm{d}}{\mathrm{d}t} \cF_2[h(t)]
    & \le -\,\Lambda\,\big(1+R^2\big)^{-\alpha/2}\,\cF_2[h(t)] +
    \lambda\,\big(1+R^2\big)^{-\alpha/2}\, \bar\mu^{\beta}(R)\, \| h(t)
    \|_{1-\beta}^2 \\
    & \le -\,\Lambda\,\big(1+R^2\big)^{-\alpha/2}\,\cF_2[h(t)] + \lambda\,C
   \,\big(1+R^2\big)^{-\alpha/2}\, \bar\mu^{\beta}(R)\,\| h_0 \|_{1-\beta}^2
  \end{aligned}
\end{equation*}
for any $R>0$ and all $t \ge 0$, and
for some positive constant $\Lambda$. This yields, for all
$t \ge 0$,
\[
\cF_2[h(t)] \le \exp\left( -\,\Lambda\,\big(1+R^2\big)^{-\alpha/2}\, t
    \right) \cF_2[h_0] + \frac\lambda\Lambda\,C\,\bar\mu^{\beta}(R)\,\| h_0 \|_{1-\beta}^2\,.
    \]
Taking $R>0$ such that $1+R^2 = t^{2/(2+\alpha)}$, we obtain  
\begin{equation*}
    \forall\,t\ge0\,,\quad\cF_2[h(t)] \lesssim \exp\left( -\,\Lambda \,
      t^{\frac{2}{2+\alpha}}\right) \| h_0 \|_{1-\beta}^2\,,
\end{equation*}
which completes the proof.
\hfill\ \qed

\subsection{Comments and open questions}

In order to apply Theorem~\ref{theo:weakcoercivity} to the linearized Boltzmann and Landau operators with very soft potentials, one has to establish~\eqref{hyp:semigroup-weakcoercivity}, which is so far an open question. Instead of proving \emph{stretched exponential} decay rates as in Corollary~\ref{cor:weakcoercivity}, \emph{polynomial} decay rates could also be achieved with~\eqref{hyp:semigroup-weakcoercivity} replaced, for some $k>0$ large enough, by
\begin{equation}
  \label{hyp:semigroup-weakcoercivity2}
  \tag{H9''}
  \mbox{\emph{$t\mapsto e^{t \sL}$ is a strongly continuous uniformly bounded
      semigroup on $\mathrm L^2(\mathscr H^k\cM^{-1})$}}
\end{equation}
where $\mathscr H(x,v) = \phi(x) + \tfrac12\,|v|^2-\min_{\R^d}\phi$. Such a condition is also open in the case of the linearized Boltzmann and Landau operators with very soft potentials, but might be easier to prove in the spirit of~\cite[Appendix~A]{MR3488535}.

\section{Examples and remarks}
\label{app:exples}
\setcounter{equation}{0}

\subsection{Examples of collision operators}
\label{ssec:ecol}

We list some examples of linear collision operators $\sC$
satisfying the hypotheses of Theorem~\ref{theo:main}, in
particular the spectral gap property~\eqref{eq:hyp-sg-C} and the
bounded moment property~\eqref{eq:lbound}.

\begin{example}[The full linear Boltzmann operator]
  Consider
  \[
    \sC f := -\left(f-r_f \,\cM- m_f\cdot v \,\cM - e_f\,\fE\,\cM
    \right)
  \]
  where $r_f$, $m_f$ and $e_f$ are defined by
  \begin{align*}
    & r_f(t,x) := \left( \int_{\R^d} f(t,x,v) \dd v \right)
      e^{\phi(x)}\,,
    & \text{{(local) density}}\,, \\
    & m_f(t,x) := \left( \int_{\R^d} v\,f(t,x,v) \dd v \right)
      e^{\phi(x)}\,,
    & \text{{(local) momentum}}\,, \\
    & e_f(t,x) := \left( \int_{\R^d} \fE(v)\,f(t,x,v) \dd v
      \right) e^{\phi(x)}\,,
    & \text{{(local) thermal energy}}\,.
  \end{align*}
  By construction, $\sC$ satisfies the spectral gap
  condition~\eqref{eq:hyp-sg-C} and since it is bounded, it
  satisfies also the bounded moment property~\eqref{eq:lbound}.
\end{example}

\begin{example}[The linearized Boltzmann collision
  operator]\label{Ex:Boltzmann}
  Consider
  \[
    \sC f := \int_{\R^d} \int_{\mathbb S^{d-1}} \big( f'\,\mu'_*
    + f'_*\,\mu' - f\,\mu_* - f_*\,\mu \big)\,|v-v_*|^\gamma\,
    b(\theta) \, \dd \sigma \dd v_*
  \]
  with the notation $f'=f(v')$, $f_* = f(v_*)$ and $f'_*=f(v'_*)$
  and
  \begin{equation}
    \label{eq:coll-rule}
    v' := \frac{v+v_*}{2} + \sigma\,\frac{|v-v_*|}{2}\,, \quad
    v'_* := \frac{v+v_*}{2} - \sigma\,\frac{|v-v_*|}{2}\,, 
  \end{equation}
  and $\theta$ is the \emph{deviation angle} defined by
  $\cos \theta := \frac{(v-v_*)}{|v-v_*|} \cdot \sigma$, and
  where $\gamma \in (-d,+\infty)$. We assume that $b$ is
  positive, smooth away from $\theta=0$ and bounded by
  $b(\theta) \lesssim \theta^{-(d-1)-s}$ with $s \in [0,2)$. This
  framework includes the short-range so-called \emph{hard
    spheres} interactions, as well as the long-range so-called
  \emph{hard potentials} and \emph{moderately soft potentials}
  interactions. This operator satisfies the spectral gap
  property~\eqref{eq:hyp-sg-C} when $\gamma+s \ge 0$ but only
  satisfies the weaker coercivity
  property~\eqref{eq:hyp-weakcoercivity-C} with $\alpha=\gamma+s$
  when $\gamma+s<0$ (see~\cite{BM05,MR2254617,MR2322149} for
  quantitative estimates). It is in general not bounded on
  $\mathrm L^2(\mu^{-1})$. Polynomials multiplied by $\mu$ are
  however in the domain of $\sC$ and it satisfies the boundedness
  property~\eqref{eq:lbound}.
  \end{example}

\begin{example}[The linearized Landau collision operator]
  \label{Ex:Landau}
  With same convention as in Example~\ref{Ex:Boltzmann}, consider
  $\sC f=\mu\,\cC h$ with $h=f/\mu$ and
  \[
    \cC h := \nabla_v \cdot \left( \int_{\R^d} \mathsf
      B_\gamma(v,v_*) \,\big( \nabla h- \nabla
      h_*\big)\,\mu\,\mu_*\,\dd v_* \right)
  \]
  where the cross-section is defined by
  \[
    \mathsf B_\gamma(v,v_*):=|v-v_*|^{\gamma+2}\left(\mbox{\em
        Id} - \frac{v-v_*}{|v-v_*|} \otimes \frac{v-v_*}{|v-v_*|}
    \right)
  \]
  with parameter $\gamma \in [-d,1]$. This operator is non-local,
  of order $2$ in velocity (of \emph{diffusive} type) and
  therefore not bounded. It satisfies the spectral gap
  condition~\eqref{eq:hyp-sg-C} when $\gamma \in [-2,1]$ but only
  the weaker coercivity property~\eqref{eq:hyp-weakcoercivity-C}
  with $\alpha=\gamma+2$ when $\gamma \in [-d,-2)$ (see,
  \emph{e.g.},~\cite{BM05,MR2322149} for constructive
  estimates). Again all polynomials in velocity multiplied by
  $\mu$ are in its domain and it satisfies the boundedness
  property~\eqref{eq:lbound}. Note that the main physical case,
  the linearisation of the so-called \emph{Landau-Coulomb
    collision operator} (describing statistically collisions for
  a gas of electrons with Coulomb interactions) corresponds to
  $\gamma=-3$ in dimension $d=3$ and is covered by our (extended)
  assumption~\eqref{eq:hyp-weakcoercivity-C}.
\end{example}
\begin{remark}
  Examples~\ref{Ex:Boltzmann} and~\ref{Ex:Landau} are obtained
  after a linearization of the bilinear form associated with the
  original nonlinear collision kernel around the Gaussian $\mu$
  and not around the Maxwellian $\cM$: when linearizing the full
  nonlinear inhomogeneous kinetic models around a
  Maxwellian~$\cM$, one gets an additional term $\rho(x)$ in
  front of the collision operator that goes to zero at
  infinity. We have not considered this degeneracy in the present
  paper: it is likely to create significant difficulties since
  there is then no uniform-in-$x$ spectral gap for $\rho\,\sC$.
\end{remark}

\subsection{Examples of potentials} \label{ssec:expot}

Let us discuss and illustrate the
hypotheses~\eqref{eq:poincarenormal} and~\eqref{eq:momentspace}
on the potential $\phi$. The bounded moment
hypothesis~\eqref{eq:momentspace} is not restrictive. Functions
like $\phi(x) = \frac{d+5}2 \ln(1+|x|^2)-Z_\phi$ which are very
slowly increasing at infinity satisfy this hypothesis, as well as
fast increasing ones like $\phi(x) = e^{|x|^4}-Z_\phi$ (here
$Z_\phi$ is the constant of normalization of $e^{-\phi}$ in
$\mathrm L^1$). Regarding the Poincar\'e
inequality~\eqref{eq:poincarenormal}, many works have been
devoted to the study of sufficient conditions in order to
guarantee the existence of a spectral gap. Here are some
examples.

\begin{example}
  The harmonic potential
  $\phi(x) = \frac12\,|x|^2 + \frac d2 \log(2\,\pi)$ satisfies
  the Poincar\'e inequality with constant
  $c_{\text{\tiny P}} = 1$. The inequality is equivalent to the
  spectral gap inequality for the operator $\Omega$ defined
  in~\eqref{omegadef}. In the flat $\mathrm L^2$ space, the
  change of unknown $u = v\,e^{-\phi/2}$ shows that the
  Poincar\'e inequality is also equivalent to the spectral gap
  inequality for the quantum harmonic oscillator operator
  $-\Delta_x+\frac14\,|x|^2-\frac d2$.
\end{example}

\begin{example}
  For a general $\phi$, the change of unknown
  $u = v\,e^{-\phi/2}$ yields the following Schrödinger-type
  operator
  \[
    P_\phi= -\Delta_x + \frac14\,|\nabla_x \phi|^2
    -\frac12\,\Delta_x \phi\,.
  \]
  According to the so-called \emph{Bakry-Emery theory} (see for
  instance to~\cite{MR3155209}), there is a spectral gap as soon
  as the Hessian $\nabla^2_x \phi$ is uniformly strictly positive
  at infinity. When it is uniformly strictly positive everywhere
  the following estimate is available on the spectral gap
  $c_{\text{\tiny P}}$:
  \[
    c_{\text{\tiny
        P}}\ge\frac12\,\inf_{x\in\R^d}\lambda_1(\nabla^2_x \phi)
  \]
  where $\lambda_1(\nabla^2_x \phi) >0$ is the lowest eigenvalue
  of $\nabla^2_x \phi$.
\end{example}

\begin{example}
  All potentials $\phi$ such that $P_\phi$ has compact resolvent
  satisfy the Poincar\'e
  inequality~\eqref{eq:poincarenormal}. This happens in
  particular when
  \[
    \lim_{|x| \rightarrow \infty} \left(\tfrac14\,|\nabla_x
      \phi|^2 - \tfrac12\,\Delta_x \phi\right) = +\infty\,,
  \]
  which is implied for instance by the stronger assumption
  \begin{equation}
    \label{basicexamplephi}
    \lim_{|x| \rightarrow \infty} |\nabla_x \phi| =
    +\infty\,,\quad \mbox{and} \quad \lim_{|x| \rightarrow
      \infty} \frac{\Delta_x \phi(x)}{|\nabla_x \phi(x)|^2} =0\,.
  \end{equation}
  This is a standard result on Schr\"odinger operators, see for
  instance~\cite[Theorem XIII.67 p.~249]{MR0493421}, and $0$ is
  then a simple discrete eigenvalue. The argument in the latter
  reference is not constructive, and for a simpler constructive
  argument we refer for instance to~\cite[Theorem~A.1]{Vil09} or
  the IMS truncation method in~\cite{MR708966}.
\end{example}

\begin{example}
  Here is an exotic example of potential that does not
  satisfy~\eqref{basicexamplephi} nor the Bakry-\'Emery criterion
  (uniform convexity of $\phi$) and for which the Poincar\'e
  inequality holds. Consider on $\R^2$
  \[
    \phi(x,y) = x^2 \left( 1+y^2 \right)^2- Z_\phi
  \]
  where $Z_\phi$ is the normalization constant so that
  $\rho = e^{-\phi}$ is a probability density. One can check that
  $P_\phi$ has a spectral gap, although $\phi$ is constant on the
  unbounded set $\{x=0\}$.
\end{example}

\subsection{Change of coordinates}
\label{ssec:rescale-change-phi-L}

Let us discuss the reduction to the
normalization~\eqref{eq:phiid}. Note that the formulas for
$\mathrm{Ker}\, \sC$ are invariant by orthonormal change of
coordinates in the velocity variable. By orthonormal change of
coordinates in both the velocity and space variables, we can then
reduce to the case when $\phi$ satisfies
\begin{equation}\label{HessianPhi}
  \langle \nabla_{\!x}^2 \phi \rangle = \left(
    \begin{matrix} p_1^2 & 0 & 0 & \cdots & 0\\0 & p_2^2 & 0 &
      \cdots & 0\\\vdots & &&& \vdots\\0 & 0 & 0 & \ldots & p_d^2
    \end{matrix} \right).
\end{equation}
where we suppose without loss of generality that all $p_j$'s are
positive. The analysis of the present paper can be adapted to
this case, including the main Theorem~\ref{theo:main}, with the
following changes. We define the set of adapted centred
rotational modes compatible with $\phi$ as in~\eqref{eq:Rphi}:
\begin{equation}
  \label{eq:rphip}
  \fR_{\phi} = \left\{ (x,v) \mapsto \left( A\,x\cdot v \right)
    \cM \,:\,A \in \cR_\phi\right\}\,.
\end{equation}
We then choose orthonormal coordinates
$x=(x_1,x_2,\dots,x_d)$ such that
$\partial_{x_j} \phi = p_j^2\,x_j$ for some $p_j>0$ if and only if
$j\in I_\phi:=\{d_\phi + 1 , \ldots , d\}$, and $x_j=0$ for any
$j\in I_\phi$ if $x\in E_\phi$ (the linear subspace defined
in~\eqref{eq:defephi}). We define the set of \emph{harmonic
  directional modes} by
\begin{equation}
  \label{eq:dphip}
  \fD_\phi = \mathop{\mathrm{Span}}
  \big\{f_j^{-}(t,x,v)\,,\,f_j^{+}(t,x,v)\big\}_{j\in I_\phi} \,,
\end{equation}
where
\begin{align*}
  & f_j^{-}(t,x,v):=\left(p_j\,x_j\,\cos(p_j\,t) -
    v_j\,\sin(p_j\,t)\right)\cM(x,v)\,, \\
  & f_j^{+}(t,x,v):=\left(p_j\,x_j\,\sin(p_j\,t) +
    v_j\,\cos(p_j\,t)\right)\cM(x,v)\,.
\end{align*}
If $d_\phi = 0$ and for some $p>0$, $p_j=p$ for all
$j\in\{1\,,\dots,d\}$, we define the set of \emph{harmonic
  pulsating modes} by
\[
  \fP_\phi = \mathop{\mathrm{Span}}
  \big\{f^-(t,x,v)\,,\,f^+(t,x,v)\big\}
\]
where
\begin{align*}
  & f^{-}(t,x,v) := \left( \tfrac12\left(|p\,x|^2-|v|^2\right)
    \cos(2\,p\,t) - p\,x\cdot v\, \sin (2\,p\,t) \right)
    \cM(x,v)\,, \\
  & f^{+}(t,x,v):= \left( \tfrac12 \left(|p\,x|^2-|v|^2\right)
    \sin(2\,p\,t) + p\,x\cdot v\,\cos (2\,p\,t)\right) \cM(x,v)\,.
\end{align*}
The functions in $\fR_{\phi}$, $\fD_\phi$ and $\fP_\phi$ are
\emph{special macroscopic modes} of~\eqref{eq:main}. With these
definitions, the proof of Theorem~\ref{theo:main} can be adapted
to prove a hypocoercivity result taking into account all special
macroscopic modes.

\subsection{Spectral interpretation}\label{ssec:B4}

We have focused so far on \emph{real} solutions
to~\eqref{eq:main}, which is natural since physical solutions
(densities of probability) are real valued. By considering
complex solutions, we can interpret the results in terms of the
\emph{complex} spectrum of the nonnegative operator
\[
  -\sL = v \cdot \nabla_x -\nabla_x \phi \cdot \nabla_v -\sC
\]
in $\mathrm L^2_\C(\cM^{-1})$, the complexification of
$\mathrm L^2(\cM^{-1})$. We consider $\phi$ as in~\eqref{HessianPhi}.
We can then describe precisely the spectrum of $-\sL$ and obtain
resolvent estimates in a half-plane that includes the imaginary
axis. Notice first that~$0$ is in the spectrum of $-\sL$ with
associated eigenspace
\[
  \mathop{\mathrm{Span}}_\C(\cM) \oplus \mathop{\mathrm{Span}}_\C
  (\cH\,\cM) \oplus \fR_{\phi,\C}
\]
where $\fR_{\phi,\C}$ is the set of rotation modes as defined
in~\eqref{eq:rphip} but extended to the corresponding $\C$-vector
space. This set is then of (complex) dimension
$2+ \dim(\fR_\phi)$. Depending on the harmonicity of $\phi$ we
have three cases which are summarized in Figure~\ref{fig:spect}.
\begin{figure}[h!t]
    \subfigure[No harmonic
    modes]{\includegraphics[height=6cm]{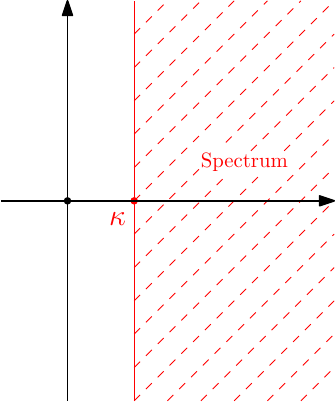}} \quad
    \subfigure[Harmonic directional
    modes]{\includegraphics[height=6cm]{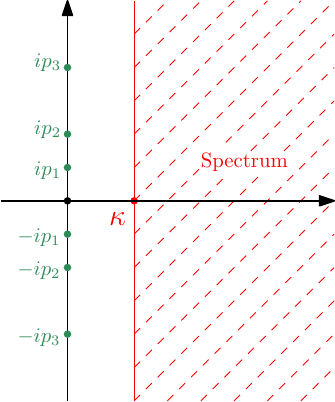}} \quad
    \subfigure[Harmonic directional and pulsating
    modes]{\includegraphics[height=6cm]{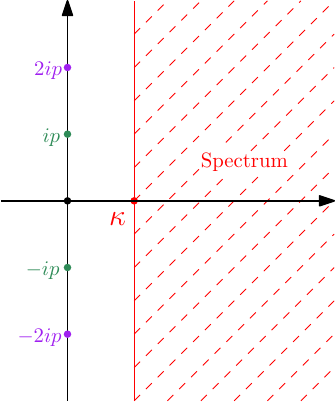}}
  \caption{Complex spectrum of $-\sL$.}
  \label{fig:spect}
\end{figure}

\noindent \textbf{(a) Case with no harmonic modes ($d_\phi=d$).}
In this case $\phi$ has no harmonic directions and there no
non-zero eigenvalue on the imaginary axis.

\noindent \textbf{(b) Case with harmonic directional modes but no
  pulsating modes ($1 \le d_\phi \le d-1$).} In this case, the
\emph{real} vector space of functions $\fD_\phi$
in~\eqref{eq:dphip} yields the complex set
\[
  \fD_{\phi,\C} = \mathop{\mathrm{Span}}_\C\,\big\{\left(
    p_j\,x_j - i\,v_j \right) e^{-i\,p_j\,t}\cM(x,v)\,,\left(
    p_j\,x_j + i\,v_j \right) e^{i\,p_j\,t}\cM(x,v)\big\}_{j\in
    I_\phi}
\]
where $I_\phi:=\left\{d_\phi+1, \dots,d\right\}$, to which we can
associate the eigenfunctions of~$(-\sL)$ corresponding to the
eigenvalues $\mp\,i\,p_j$ and given by
\[
  (x,v) \mapsto f_j^\pm(x,v) = (p_j\,x_j \pm i\,v_j) \cM(x,v)\,.
\]

\noindent \textbf{(c) Case with harmonic directional and
  pulsating modes ($d_\phi=0$).} In this last case necessarily
all $p_j$'s are equal to a common value $p>0$ and
$\phi(x)= \frac12 |p\,x|^2 + \frac d2 \log(2\,\pi) -
d\log(p)$. All possible harmonic directional modes exist, as well
as all possible infinitesimal rotational modes $\fR_{\phi,\C}$
with $\cR_\phi = \fM^{\text{\tiny skew}}_{d \times d}(\C)$. The
complexification of the set $\fP_\phi$ defined
in~\eqref{eq:dphip} is
\[
  \fP_{\phi,\C} = \mathop{\mathrm{Span}}
  \big\{ e^{2\,i\,p\,t}\, f^+(x,v)\,, e^{-\,2\,i\,p\,t}\,
  f^-(x,v) \big\}
\]
where
$f^\pm(x,v):= \left[ p\,x\cdot v \pm \frac
  i2\left(|p\,x|^2-|v|^2\right) \right] \cM(x,v)$
  are eigenfunctions of $(-\sL)$ with eigenvalues
$\pm\,2\,i\,p$.

The analysis of the paper can be extended to the complex
Hilbert space $\mathrm L^2_\C(\cM^{-1})$ with a set of
\emph{special macroscopic modes} defined by
\[
  \cS := \mathop{\mathrm{Span}}_\C(\cM) \oplus
  \mathop{\mathrm{Span}}_\C (\cH\,\cM) \oplus \fR_{\phi,\C}
  \oplus \mathop{\mathrm{Span}} \big\{f_j^\pm\big\}_{j\in I_\phi}
  \oplus \mathop{\mathrm{Span}} \left\{ f^\pm \right\}
\]
where the $f_j$'s and the $f^\pm$'s are defined above (when
$\phi$ has the relevant harmonicity). Let $\cS^\bot$ be the
orthogonal of $\cS$ in $\mathrm L^2_\C(\cM^{-1})$. We note that
since $\sL$ is a real operator, both $\cS$ and $\cS^\bot$ are
stable by conjugation and therefore stable by $\sL$ and
$\sL^*$. Using the Laplace transform, we obtain from
Theorem~\ref{theo:main} the following resolvent estimate for
$-\sL|_{\cS^\bot}$:
\[
  \forall\,z \in \C \quad\mbox{with} \quad \Re(z) <\kappa\,,\quad
  \left\|(z \mathrm{Id} +
    \sL|_{\cS^\bot})^{-1}\right\|_{\cB(\cS^\bot)} \leq
  \frac{\tilde{C}}{\kappa-\Re(z)}
\]
where $\tilde{C}$ is an explicit constant depending on $\kappa$
and $C$ in Theorem~\ref{theo:main} and
$\| \cdot \|_{\cB(\cS^\bot)}$ stands for the operator norm on
$\cS^\bot$. The provides the resolvent estimates in the left
half-planes in Figure~\ref{fig:spect}.

\subsection{Special macroscopic modes for the full nonlinear
  Boltzmann equation}
\label{ssec:fullnl}

The \emph{special macroscopic modes} which minimize the entropy
for the full nonlinear Boltz\-mann equation are the nonlinear
counterparts to the linearized special macroscopic modes studied
in the present paper. They appear for the first time in the
literature in Boltzmann's paper~\cite{Boltzmann} as mentioned in
the introduction. The full nonlinear inhomogeneous Boltzmann
equation is
\begin{align}
  \label{nonlinear}
  \partial_t F + v\cdot \nabla_x F - \nabla_x \phi \cdot \nabla_v
  F = \partial_t F + \sT F = \mathscr{Q}(F,F)
\end{align}
where, with the classical notations $F'=F(v')$, $F_* = F(v_*)$
and $F'_*=F(v'_*)$ associated to elastic collisions
$(v,v_*)\mapsto(v',v_*')$, such that the microscopic conservation
of momentum $v'+v'_*= v+v_*$ and energy
$|v'|^2+|v'_*|^2= |v|^2+|v_*|^2$ hold, the Boltzmann collision
operator writes
\[
  \mathscr{Q} (F,F) := \int_{\R^d} \int_{\mathbb S^{d-1}} \mathsf
  B(v-v_*,\sigma) \left( F' F'_* - FF_* \right) \dd \sigma \dd
  v_*\,.
\]
Here $\mathsf B \ge 0$ is the \emph{cross-section}. We refer
to~\cite{Cercignani} for more details. Let us assume the
normalization~\eqref{eq:phiid} on $\phi$. We consider the
functions in the space~$\cS$ of \emph{special macroscopic modes}
generated by\\
$\rhd$ the set $\mathfrak R_\phi$ of \emph{rotation modes
  compatible with $\phi$} if $\phi$ admits any,\\
$\rhd$ the set $\mathfrak D_\phi$ of \emph{harmonic directional
  modes} if $\phi$ has harmonic directions,\\
$\rhd$ the set $\mathfrak P_\phi$ of \emph{harmonic pulsating
  modes} if $\phi$ is fully harmonic.\\
For any $f\in\cS$, the function
$F(t,x,v) := e^{h(t,x,v)}\,\cM(x,v)$ with $h=f/\cM$ is a
time-periodic solution to~\eqref{nonlinear}. Indeed
$h(t,x,\cdot)$ is a linear combination of $1$, $v_i$,
$i \in \{1,\dots,d\}$ and $|v|^2$ for each $t,x$, and therefore
the microscopic conservation of momentum and energy imply
\[
  \forall \, t,x,v,\sigma, \quad
  h(t,x,v') + h(t,x,v'_*) = h(t,x,v) + h(t,x,v_*)
\]
where the four velocities $v$, $v_*$, $v',v'_*$
satisfy~\eqref{eq:coll-rule}. This proves the identity
$e^{h'}\,\cM'\,e^{h'_*}\,\cM'_* = e^h\,\cM\,e^{h_*}\,\cM_*$ and thus
$\mathscr{Q}(e^h\,\cM,e^h\,\cM)=0$. Finally we obtain
$\sT (e^h\,\cM)= \sT(e^h)\,\cM + e^h\,\sT(\cM) = e^h \left[ \sT(h)
 \,\cM + \sT(\cM) \right] =0$, where we have used that $\sT$ is a
first order operator and $\sT(\cM)=\sT(h)=0$ as calculated
before.

\begin{ack}
The authors
  warmly thank L.~Desvillettes, S.~V\~u Ng\d{o}c, C.~Cheverry and
  M.~Ro\-drigues for fruitful discussions and remarks which led to
  the observations of Appendices~\ref{ssec:B4}
  and~\ref{ssec:fullnl}.
\end{ack}

\begin{funding}
KC and JD have been partially supported by the Project EFI (ANR-17-CE40-0030). FH benefits from the support of the France 2030 framework programme, through the Centre Henri Lebesgue Mathematical Center. CM is partially supported by the European Research Council (ERC) under the European Union’s Horizon 2020 research and innovation programme MAFRAN (grant agreement No. 726386).
\end{funding}

\noindent{\scriptsize \copyright\,2023~by the authors. Reproduction of this article by any means permitted for non-commercial purposes. \hbox{\href{https://creativecommons.org/licenses/by/4.0/legalcode}{CC-BY 4.0}}}
\bibliographystyle{cdhmms}
\bibliography{CDHMMS}

\begin{thebibliography}{10}
\providecommand{\url}[1]{\texttt{#1}}
\providecommand{\urlprefix}{URL }
\providecommand{\eprint}[2][]{\url{#2}}

\bibitem{MR3155209}
Bakry, D., Gentil, I., Ledoux, M.: Analysis and geometry of {M}arkov diffusion
  operators. Grundlehren der Mathematischen Wissenschaften [Fundamental
  Principles of Mathematical Sciences] 348, Springer, Cham (2014) \MR{3155209}

\bibitem{BM05}
Baranger, C., Mouhot, C.: Explicit spectral gap estimates for the linearized
  {B}oltzmann and {L}andau operators with hard potentials. Rev. Mat.
  Iberoamericana \textbf{21}, 819--841 (2005) \MR{2231011}

\bibitem{BGL93}
Bardos, C., Golse, F., Levermore, C.~D.: Fluid dynamic limits of kinetic
  equations. {II}. {C}onvergence proofs for the {B}oltzmann equation. Comm.
  Pure Appl. Math. \textbf{46}, 667--753 (1993) \MR{1213991}

\bibitem{binney2011galactic}
Binney, J., Tremaine, S.: Galactic dynamics. Princeton university press (2011)

\bibitem{Boltzmann}
Boltzmann, L.: {\"U}ber die {A}ufstellung und {I}ntegration der {G}leichungen,
  welche die {M}olekularbewegung in {G}asen bestimmen. In: Wissenschaftliche
  Abhandlungen von L.~Boltz\-mann, vol. 2, Barth, Leipzig, 55--102 (1876)
  (1909)

\bibitem{BosiCac}
Bosi, R., C\'{a}ceres, M.~J.: The {BGK} model with external confining
  potential: existence, long-time behaviour and time-periodic {M}axwellian
  equilibria. J. Stat. Phys. \textbf{136}, 297--330 (2009) \MR{2525248}

\bibitem{BMM19}
Briant, M., Merino-Aceituno, S., Mouhot, C.: From {B}oltzmann to incompressible
  {N}avier-{S}tokes in {S}obolev spaces with polynomial weight. Anal. Appl.
  (Singap.) \textbf{17}, 85--116 (2019) \MR{3894734}

\bibitem{MR298480}
Calder\'{o}n, A.-P., Vaillancourt, R.: A class of bounded pseudo-differential
  operators. Proc. Nat. Acad. Sci. U.S.A. \textbf{69}, 1185--1187 (1972)
  \MR{298480}

\bibitem{CDHMM21}
Carrapatoso, K., Dolbeault, J., H{\'{e}}rau, F., Mischler, S., Mouhot, C.:
  Weighted {K}orn and {P}oincar{\'{e}}-{K}orn inequalities in the {E}uclidean
  space and associated operators. Archive for Rational Mechanics and Analysis
  \textbf{243}, 1565--1596 (2022) \MR{4381147}

\bibitem{Cercignani}
Cercignani, C.: The {B}oltzmann equation and its applications. Applied
  Mathematical Sciences 67, Springer-Verlag, New York (1988) \MR{1313028}

\bibitem{DV02}
Desvillettes, L., Villani, C.: On a variant of {K}orn's inequality arising in
  statistical mechanics. ESAIM Control Optim. Calc. Var. \textbf{8}, 603--619
  (electronic) (2002) \MR{1932965}

\bibitem{DV05}
Desvillettes, L., Villani, C.: On the trend to global equilibrium for spatially
  inhomogeneous kinetic systems: the {B}oltzmann equation. Invent. Math.
  \textbf{159}, 245--316 (2005) \MR{2116276}

\bibitem{DMS09}
Dolbeault, J., Mouhot, C., Schmeiser, C.: Hypocoercivity for kinetic equations
  with linear relaxation terms. C. R. Acad. Sci. Paris, Ser. I \textbf{347},
  511--516 (2009)

\bibitem{DMS15}
Dolbeault, J., Mouhot, C., Schmeiser, C.: Hypocoercivity for linear kinetic
  equations conserving mass. Trans. Amer. Math. Soc. \textbf{367}, 3807--3828
  (2015) \MR{3324910}

\bibitem{Dua11}
Duan, R.: Hypocoercivity of linear degenerately dissipative kinetic equations.
  Nonlinearity \textbf{24}, 2165--2189 (2011) \MR{2813582}

\bibitem{DL12}
Duan, R., Li, W.-X.: Hypocoercivity for the linear {B}oltzmann equation with
  confining forces. J. Stat. Phys. \textbf{148}, 306--324 (2012) \MR{2966364}

\bibitem{Gra65}
Grad, H.: On {B}oltzmann's {$H$}-theorem. J. Soc. Indust. Appl. Math.
  \textbf{13}, 259--277 (1965) \MR{0180278}

\bibitem{GMM17}
Gualdani, M.~P., Mischler, S., Mouhot, C.: Factorization of non-symmetric
  operators and exponential {$H$}-theorem. M\'{e}m. Soc. Math. Fr. (N.S.)
  \textbf{153}, 137 (2017) \MR{3779780}

\bibitem{Trizac}
Gu\'ery-Odelin, D., Muga, J.~G., Ruiz-Montero, M.~J., Trizac, E.:
  Nonequilibrium solutions of the {B}oltzmann equation under the action of an
  external force. Phys. Rev. Lett. \textbf{112}, 180602 (2014)

\bibitem{Guo02b}
Guo, Y.: The {L}andau equation in a periodic box. Comm. Math. Phys.
  \textbf{231}, 391--434 (2002) \MR{MR1946444}

\bibitem{Guo02a}
Guo, Y.: The {V}lasov-{P}oisson-{B}oltzmann system near {M}axwellians. Comm.
  Pure Appl. Math. \textbf{55}, 1104--1135 (2002) \MR{MR1908664}

\bibitem{Guo03a}
Guo, Y.: The {V}lasov-{M}axwell-{B}oltzmann system near {M}axwellians. Invent.
  Math. \textbf{153}, 593--630 (2003) \MR{MR2000470}

\bibitem{MR2172804}
Guo, Y.: Boltzmann diffusive limit beyond the {N}avier-{S}tokes approximation.
  Comm. Pure Appl. Math. \textbf{59}, 626--687 (2006) \MR{2172804}

\bibitem{HN03}
Hanouzet, B., Natalini, R.: Global existence of smooth solutions for partially
  dissipative hyperbolic systems with a convex entropy. Arch. Ration. Mech.
  Anal. \textbf{169}, 89--117 (2003) \MR{2005637}

\bibitem{HN05}
Helffer, B., Nier, F.: Hypoelliptic estimates and spectral theory for
  {F}okker-{P}lanck operators and {W}itten {L}aplacians. Lecture Notes in
  Mathematics 1862, Springer-Verlag, Berlin (2005) \MR{2130405}

\bibitem{Her06}
H{\'e}rau, F.: Hypocoercivity and exponential time decay for the linear
  inhomogeneous relaxation {B}oltzmann equation. Asymptot. Anal. \textbf{46},
  349--359 (2006) \MR{MR2215889}

\bibitem{HN04}
H{\'e}rau, F., Nier, F.: Isotropic hypoellipticity and trend to equilibrium for
  the {F}okker-{P}lanck equation with a high-degree potential. Arch. Ration.
  Mech. Anal. \textbf{171}, 151--218 (2004) \MR{MR2034753}

\bibitem{HR18}
Herda, M., Rodrigues, L.~M.: Large-time behavior of solutions to
  {V}lasov-{P}oisson-{F}okker-{P}lanck equations: from evanescent collisions to
  diffusive limit. J. Stat. Phys. \textbf{170}, 895--931 (2018) \MR{3767000}

\bibitem{Hor67}
H\"{o}rmander, L.: Hypoelliptic second order differential equations. Acta Math.
  \textbf{119}, 147--171 (1967) \MR{222474}

\bibitem{KS88}
Kawashima, S., Shizuta, Y.: On the normal form of the symmetric
  hyperbolic-parabolic systems associated with the conservation laws. Tohoku
  Math. J. (2) \textbf{40}, 449--464 (1988) \MR{957056}

\bibitem{MR3752660}
Knauf, A.: Mathematical physics: classical mechanics. Unitext 109,
  Springer-Verlag, Berlin (2018) \MR{3752660}

\bibitem{Kor06}
Korn, A.: Die {E}igenschwingungen eines elastischen {K\"o}rpers mit ruhender
  {O}ber\-fl{\"a}che. Akad. der Wissensch., Munich, Math. phys. KI.
  \textbf{36}, 351 (1906)

\bibitem{Kor09}
Korn, A.: {\"Uber einige {U}ngleichungen, welche in der {T}heorie der
  elastischen und elektrischen {S}chwingungen eine {R}olle spielen}. {Krak.
  Anz., 705-724 (1909)} (1909)

\bibitem{MR3488535}
Mischler, S., Mouhot, C.: Exponential stability of slowly decaying solutions to
  the kinetic-{F}okker-{P}lanck equation. Arch. Ration. Mech. Anal.
  \textbf{221}, 677--723 (2016) \MR{3488535}

\bibitem{MR2254617}
Mouhot, C.: Explicit coercivity estimates for the linearized {B}oltzmann and
  {L}andau operators. Comm. Partial Differential Equations \textbf{31},
  1321--1348 (2006) \MR{2254617}

\bibitem{MN05}
Mouhot, C., Neumann, L.: Quantitative perturbative study of convergence to
  equilibrium for collisional kinetic models in the torus. Nonlinearity
  \textbf{19}, 969--998 (2006) \MR{MR2214953}

\bibitem{MR2322149}
Mouhot, C., Strain, R.~M.: Spectral gap and coercivity estimates for linearized
  {B}oltzmann collision operators without angular cutoff. J. Math. Pures Appl.
  (9) \textbf{87}, 515--535 (2007) \MR{2322149}

\bibitem{MR0493421}
Reed, M., Simon, B.: Methods of modern mathematical physics. {IV}. {A}nalysis
  of operators. Academic Press [Harcourt Brace Jovanovich, Publishers], New
  York-London (1978) \MR{0493421}

\bibitem{RS04}
Ruggeri, T., Serre, D.: Stability of constant equilibrium state for dissipative
  balance laws system with a convex entropy. Quart. Appl. Math. \textbf{62},
  163--179 (2004) \MR{2032577}

\bibitem{SW03}
Sideris, T.~C., Thomases, B., Wang, D.: Long time behavior of solutions to the
  3{D} compressible {E}uler equations with damping. Comm. Partial Differential
  Equations \textbf{28}, 795--816 (2003) \MR{1978315}

\bibitem{MR708966}
Simon, B.: Semiclassical analysis of low lying eigenvalues. {I}.
  {N}ondegenerate minima: asymptotic expansions. Ann. Inst. H. Poincar\'{e}
  Sect. A (N.S.) \textbf{38}, 295--308 (1983) \MR{708966}

\bibitem{MR2761185}
Spivak, M.: Physics for mathematicians---mechanics {I}. Publish or Perish,
  Inc., Houston, TX (2010) \MR{2761185}

\bibitem{Guo-Strain-2006}
Strain, R.~M., Guo, Y.: Almost exponential decay near {M}axwellian. Comm.
  Partial Differential Equations \textbf{31}, 417--429 (2006) \MR{MR2209761}

\bibitem{Uhlenbeck}
Uhlenbeck, G.~E., Ford, G.~W.: Lectures in statistical mechanics. With an
  appendix on quantum statistics of interacting particles by E.M.Montroll.
  Lectures in Applied Mathematics 1, American Mathematical Society, Providence
  (1963)

\bibitem{Vil02}
Villani, C.: A review of mathematical topics in collisional kinetic theory. In:
  Handbook of mathematical fluid dynamics, {V}ol. {I}, North-Holland,
  Amsterdam, 71--305 (2002) \MR{1942465}

\bibitem{Vil09}
Villani, C.: Hypocoercivity. Mem. Amer. Math. Soc. \textbf{202}, iv+141 (2009)
  \MR{2562709}

\end{thebibliography}
\parindent=0pt
\end{document}